\documentclass[12pt]{amsart}

\usepackage{import} 
\usepackage{transparent} 
\usepackage{lmodern} 
\usepackage[utf8]{inputenc}
\usepackage[english]{babel} 
\usepackage{amsmath}
\usepackage{tikz} 
\usepackage{tikz-cd} 
\usepackage{amssymb} 
\usepackage{amsthm}
\usepackage{mathrsfs}
\usepackage{enumitem}
\usepackage{thmtools} 
\usepackage{extarrows}
\usepackage{float}
\usepackage{mathscinet}
\usepackage{amsfonts}
\usepackage{comment}
\usepackage{mathtools}
\usepackage{array}
\usepackage[margin=1in]{geometry}

\usepackage[toc,page]{appendix}

\usetikzlibrary{babel} 
\usepackage[pdftex,pdfpagelabels,bookmarks,hyperindex,hyperfigures]{hyperref}
\usepackage{pdfpages}
\usepackage{bm}

\newcommand{\Z}{\mathbb{Z}}
\renewcommand{\L}{\Lambda}
\newcommand{\R}{\mathbb{R}}

\newcommand{\act}{\mathfrak{a}}

\newcommand{\area}{\text{Area}}

\newcommand{\delbar}{\Bar{\partial}}

\newcommand{\Mod}{\mathfrak{M}}

\newcommand{\action}[1]{\mathfrak{a}({#1})}

\newtheorem{theorem}{Theorem}[section]
\newtheorem{proposition}[theorem]{Proposition}

\newtheorem{lemma}[theorem]{Lemma}
\newtheorem{corollary}[theorem]{Corollary}

\theoremstyle{definition}

\newtheorem{definition}[theorem]{Definition} 

\newtheorem{question}[theorem]{Question}

\numberwithin{equation}{section}
\theoremstyle{remark}

\newtheorem{remark}[theorem]{Remark}

\graphicspath{ {./figures/} }

\addto\extrasenglish{ 

 }

\title{Floor-crossing for Legendrian Clasp Move}
\author{Filip Strako\v{s} }
\address{Department of Mathematics, Uppsala University,
Box 480, 751 06 Uppsala, Sweden}
\email{filip.strakos@math.uu.se}
\date{September 1, 2026}

\begin{document}


\begin{abstract}
    We investigate the~effect of the~clasp move of Legendrian submanifolds of dimension at least two on the~moduli spaces of pseudo-holomorphic curves that contribute to the~differential of the~Chekanov-Eliashberg algebra. As an application, we compute the~differential of Chekanov-Eliashberg algebra of twist spheres $\Lambda_k$ of dimension at least two. Moreover, we construct an infinite family of Legendrians $\mathcal{T}(\Lambda,k)$ from any horizontally displaceable Legendrian $\Lambda$ in $P\times\R$, such that the~ Chekanov-Eliashberg algebra of $\mathcal{T}(\Lambda,k)$ admits an augmentation, while $\mathcal{T}(\Lambda,k)$ has no exact embedded Lagrangian filling of vanishing Maslov class in the~symplectization of $P\times \R$ for sufficiently large $k$.
\end{abstract}

\maketitle

\tableofcontents
\section{Introduction}
    Let $(P,d\eta)$ be an exact symplectic manifold and $P\times\R$ be its contactisation. We call the~canonical projection $\Pi_P:P\times \R\to P$ the~Lagrangian projection and we consider embedded Legendrian submanifolds $\L$ of $P\times \R$ so that their Lagrangian projection $L=\Pi_P(\Lambda)$ is an exact closed immersed Lagrangian with transverse self-intersections. The~intersection points of $L$ correspond to the~Reeb chords of $\Lambda$ of the~Reeb vector field of the~contact form $\alpha=dz-\eta$. We denote the~set of the~Reeb chords by $\mathcal{R}(\Lambda)$. The~Chekanov-Eliashberg algebra $(\mathcal{A}(\Lambda),\partial)$ is a~tensor algebra over $\Z_2$ of the~vector space spanned by the~elements of $\mathcal{R}(\Lambda)$ with grading determined by the~Conley-Zehnder index of the~chords and differential counting the~pseudo-holomorphic polygons in $P$ with boundary on $L$. Note that the~stable tame isomorphism class of $\mathcal{A}(\Lambda)$ is a~Legendrian isotopy invariant, see \cite{ChekanovLCH},\cite{EEScontacthomology},\cite{EESnoniso},\cite{EESPR}.
    
    In this paper, we will investigate the~change of the~differential $\partial$ under the~(un)clasp move, a~specific regular homotopy through Legendrian immersions. Let us give a~brief description of this move (for more details see Section~\ref{sec:Clasp move}). The~action $\act(c)=\int c^*\alpha$ of the~Reeb chord $c$ corresponds to the~length of the~projection of $c$ onto the~$z$-coordinate. Let us assume that $\Lambda$ has a~Reeb chord $s_0$ of action $0$, meaning $\Lambda$ is an immersed Legendrian with a~transverse self-intersection point $s_0$. We construct a~family of immersed Legendrians $\Lambda_\mu$ so that the~family coincides with $\Lambda$ everywhere except for a~small neighbourhood of the~self-intersection and so $\Lambda_\mu$ is a~Legendrian embedding for all $\mu$ except for $\mu=0$. In the~family $\Lambda_\mu$, we realise the~self-intersection $s_0$ of $\L_0$ by contracting a~Reeb chord $s_\mu$ of $\Lambda_\mu$, where the~sheets with the~endpoints of $s_\mu$ locally pass through each other along the~clasp move. For $\mu>0$, we say that $\L_\mu$ was obtained by a~clasp move of $\L_{-\mu}$ centred at $s_\mu$. Conversely, $\Lambda_{-\mu}$ is obtained by performing a~unclasp move on $\Lambda_\mu$ centred at $s_{-\mu}$. Figure~\ref{fig:PRclasp} illustrates the~move from left to right; the~reverse direction corresponds to the~unclasp move.
    
    In general, the~(un)clasp move might produce a~loose Legendrian out of a~non-loose one and vice versa. As an immediate corollary of this fact, the~$h$-principle for loose Legendrians (see \cite{murphyloose}) yields, in Section~\ref{sec:Legunknotting}, bounds on the~higher-dimensional Legendria unknotting number for Legendrian spheres in terms of the~self-linking of their Lagrangian projections. Thus, in contrast to its smooth equivalent for knots, the~higher-dimensional Legendria unknotting number of Legendrian spheres can attain only a~finite number of values for a~fixed Thurston-Bennequin invariant, see Proposition~\ref{prop:Legunknotting}.
    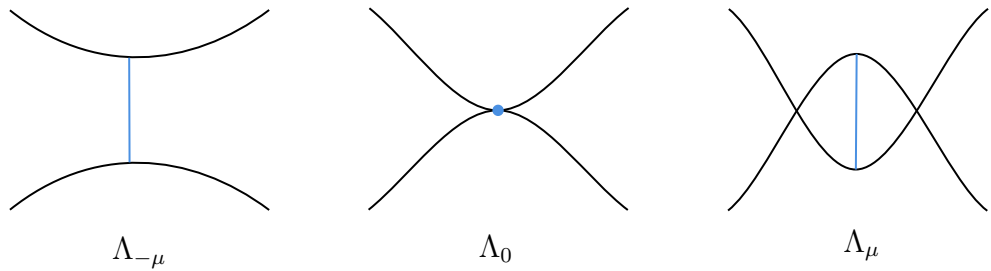
\begin{figure}
        \centering
        \tikzset{every picture/.style={line width=0.75pt}} 
        
        \begin{tikzpicture}[x=0.75pt,y=0.75pt,yscale=-1,xscale=1]
        
        \draw    (88.16,6.19) .. controls (129.41,39.19) and (178.25,36.05) .. (218.25,6.05) ;
        \draw    (88.16,106.3) .. controls (129.41,73.3) and (178.25,76.43) .. (218.25,106.43) ;
        \draw    (268.16,106.3) .. controls (288.51,90.02) and (311.07,56.22) .. (333.14,56.42) .. controls (355.2,56.61) and (377.98,91.23) .. (398.25,106.43) ;
        \draw    (268.56,5.18) .. controls (288.91,21.46) and (311.07,56.61) .. (333.14,56.42) .. controls (355.2,56.22) and (378.38,20.25) .. (398.65,5.05) ;
        \draw    (448.56,106.7) .. controls (468.91,90.42) and (491.03,27.92) .. (513.1,28.11) .. controls (535.16,28.31) and (558.38,91.63) .. (578.65,106.83) ;
        \draw    (448.96,5.58) .. controls (469.31,21.86) and (490.63,86.31) .. (512.7,86.11) .. controls (534.76,85.92) and (558.78,20.65) .. (579.05,5.45) ;
        \draw [color={rgb, 255:red, 74; green, 144; blue, 226 }  ,draw opacity=1 ]   (148,30) -- (148.19,82.51) ;
        \draw [color={rgb, 255:red, 74; green, 144; blue, 226 }  ,draw opacity=1 ]   (513.1,28.11) -- (512.7,86.11) ;
        \draw  [color={rgb, 255:red, 74; green, 144; blue, 226 }  ,draw opacity=1 ][fill={rgb, 255:red, 74; green, 144; blue, 226 }  ,fill opacity=1 ] (330.79,56.42) .. controls (330.79,55.12) and (331.84,54.07) .. (333.14,54.07) .. controls (334.43,54.07) and (335.48,55.12) .. (335.48,56.42) .. controls (335.48,57.71) and (334.43,58.76) .. (333.14,58.76) .. controls (331.84,58.76) and (330.79,57.71) .. (330.79,56.42) -- cycle ;
        
        \draw (138,118.46) node [anchor=north west][inner sep=0.75pt]    {$\Lambda _{-\mu }$};
        \draw (322,117.46) node [anchor=north west][inner sep=0.75pt]    {$\Lambda _{0}$};
        \draw (505,114.65) node [anchor=north west][inner sep=0.75pt]    {$\Lambda _{\mu }$};

        \end{tikzpicture}

        \caption{The~clasp move, read from left to right, in the~front projection of a~Legendrian knot. Read from right to left, the~figure represents the~unclasp move.}
        \label{fig:PRclasp}
    \end{figure}
    \newpage
    An augmentation of $\mathcal{A}(\Lambda)$ is a~degree $0$ map of differential graded algebras $$\varepsilon:(\mathcal{A}(\Lambda),\partial)\to (\Z_2,0),$$ where $\Z_2$ is concentrated in degree $0$. We denote by $Aug(\Lambda)$ the~set of augmentations of $\mathcal{A}(\Lambda)$. The~clasp move of connected Legendrian knots was studied in \cite{PR2019}. There, one considers the~immersed Lagrangian cobordism $\Sigma$ in the~symplectization of $P\times\R$ induced by the~clasp move from $\L_{-\mu}$ to $\L_\mu$. Counting certain Ekholm's flow-trees \cite{Ekholmtrees}, one defines the~so-called immersed cobordism map, and one obtains an algebraic subset $I_\Sigma$ of $Aug(\Lambda_\mu)$. Note that $I_\Sigma$ might not, in general, coincide with $Aug(\Lambda_\mu)$. In this paper, we describe how the~contributions to the~differential of Chekanov-Eliashberg algebra of $\L_{-\mu}$ (of dimension at least two) change under the~clasp move. This means that we understand the~full change of $Aug(\Lambda_{-\mu})$ to $Aug(\L_\mu)$. It is unknown to the~author how the~inclusion $I_\Sigma\subset Aug(\L_\mu)$ and the~variety $Aug(\L_{-\mu})$ relate to each other. Our methods rely on a~certain genericity condition for $J$-holomorphic curves (see the~proof of Proposition~\ref{prop: line segment} and \cite[Proposition 4.10]{EESnoniso}) that fails in the~case of Legendrian knots, and so the~full effect of the~clasp move on the~augmentation variety remains open for $dim \,\L_\mu=1$.
    
    Let us summarise the~effect that the~clasp move has on the~contributions to the~differential of Chekanov-Eliashberg algebra before and after the~clasp move; new contributions come from the~curves with possibly multiple positive punctures asymptotic to the~clasping chord, and the~old curves with a~negative puncture asymptotic to the~clasping chord stop contributing. We call this the~floor-crossing.
    
    More precisely, fix $\mu$. Let $\textbf{w}$ be a~word in Reeb chords of $\Lambda_\mu$ containing exactly one letter $a$ different from $s_\mu$, possibly several letters $s_\mu$, and possibly several letters $b_i$ distinct from $a$ and $s_\mu$. We abuse the~notation and denote the~Reeb chord of $\L_\mu$ and the~corresponding intersection point of $L_\mu$ with the~same symbol. Let $\Mod_\mu(\textbf{w})$ denote the~moduli space of pseudo-holomorphic polygons in $P$ with boundary on $L_\mu$ whose boundary punctures are asymptotic to the~letters of $\textbf{w}$ in the~order prescribed by \textbf{w}. Every element of $\Mod_\mu(\textbf{w})$ has:
    \begin{itemize}
        \item  exactly one positive puncture asymptotic to $a$,
        \item  negative punctures asymptotic to $b_i$,
        \item  negative (positive) puncture asymptotic to $s_\mu$ if $\mu$ is $\mu>0$ (respectively, $\mu<0$).
    \end{itemize}
    Moreover, the~order of letters in $\textbf{w}$ determines counterclockwise the~order of boundary punctures in the~domain of the~polygon.  Consult Section~\ref{sec:COMPACTNESS} for more details and Figure~\ref{fig:modulifloorcrossing} for illustration. We can proceed with the~statement of the~main theorem of this paper.
    
    \begin{theorem}[Floor-crossing]\label{thm:floorcrossing}
        Let $n>1$, let $\Lambda_0$ be an $n$-dimensional Legendrian with a~Reeb chord of action $0$. Let $\mu>0$, and let $\textbf{w}$ be a~word in the~Reeb chords of $\Lambda_{\pm\mu}$. If $\dim \Mod_{_\mu}(\textbf{w})=0$, then the~moduli spaces $\Mod_{-\mu}(\textbf{w})$ and $\Mod_\mu(\textbf{w})$ of rigid curves are compact and diffeomorphic.
    \end{theorem}
    \begin{figure}
        \centering
        \tikzset{every picture/.style={line width=0.75pt}} 
        
        \begin{tikzpicture}[x=0.75pt,y=0.75pt,yscale=-1,xscale=1]
        
        \draw    (446.87,34.2) -- (506.26,34.02) ;
        \draw    (406.53,154.2) -- (436.26,154.33) ;
        \draw    (456.87,154.44) -- (486.6,154.57) ;
        \draw    (506.53,154.44) -- (536.26,154.57) ;
        \draw    (406.53,154.2) .. controls (406.39,110.3) and (446.39,79.3) .. (446.87,34.2) ;
        \draw    (536.26,154.57) .. controls (536.12,110.66) and (505.78,79.12) .. (506.26,34.02) ;
        \draw    (436.26,154.33) .. controls (436.39,104.3) and (456.39,104.96) .. (456.87,154.44) ;
        \draw    (486.6,154.57) .. controls (486.72,104.54) and (506.05,104.96) .. (506.53,154.44) ;
        \draw    (157.33,34) -- (216.73,33.82) ;
        \draw    (117,154) -- (146.73,154.13) ;
        \draw    (167.33,154.24) -- (197.06,154.37) ;
        \draw    (237.33,34.24) -- (267.06,34.37) ;
        \draw    (117,154) .. controls (116.85,110.1) and (156.85,79.1) .. (157.33,34) ;
        \draw    (237.33,34.24) .. controls (236.79,78.76) and (216.25,78.92) .. (216.73,33.82) ;
        \draw    (146.73,154.13) .. controls (146.85,104.1) and (166.85,104.76) .. (167.33,154.24) ;
        \draw    (197.06,154.37) .. controls (197.19,104.34) and (266.79,100.1) .. (267.06,34.37) ;
        \draw    (298.75,94.16) -- (384.5,93.92) ;
        \draw [shift={(386.5,93.91)}, rotate = 179.84] [color={rgb, 255:red, 0; green, 0; blue, 0 }  ][line width=0.75]    (10.93,-3.29) .. controls (6.95,-1.4) and (3.31,-0.3) .. (0,0) .. controls (3.31,0.3) and (6.95,1.4) .. (10.93,3.29)   ;
        \draw [shift={(296.75,94.16)}, rotate = 359.84] [color={rgb, 255:red, 0; green, 0; blue, 0 }  ][line width=0.75]    (10.93,-3.29) .. controls (6.95,-1.4) and (3.31,-0.3) .. (0,0) .. controls (3.31,0.3) and (6.95,1.4) .. (10.93,3.29)   ;
        
        \draw (412.69,160.48) node [anchor=north west][inner sep=0.75pt]    {$b_{1}$};
        \draw (463.69,161.82) node [anchor=north west][inner sep=0.75pt]    {$b_{2}$};
        \draw (514.35,159.85) node [anchor=north west][inner sep=0.75pt]    {$s_{\mu }$};
        \draw (471.69,13.6) node [anchor=north west][inner sep=0.75pt]    {$a$};
        \draw (183.15,14.31) node [anchor=north west][inner sep=0.75pt]    {$a$};
        \draw (124.15,159.28) node [anchor=north west][inner sep=0.75pt]    {$b_{1}$};
        \draw (175.82,158.28) node [anchor=north west][inner sep=0.75pt]    {$b_{2}$};
        \draw (244.82,10.98) node [anchor=north west][inner sep=0.75pt]    {$s_{-\mu }$};
        \draw (325.41,70.43) node [anchor=north west][inner sep=0.75pt]    {$1-1$};
        \draw (105,195.4) node [anchor=north west][inner sep=0.75pt]    {$u_{-\mu } \in \mathfrak{M}_{-\mu }( ab_{1} b_{2} s_{-\mu })$};
        \draw (398,195.66) node [anchor=north west][inner sep=0.75pt]    {$u_{\mu } \in \mathfrak{M}_{\mu }( ab_{1} b_{2} s_{\mu })$};
        \draw (466.23,78.43) node [anchor=north west][inner sep=0.75pt]    {$u_{\mu }$};
        \draw (169.64,76.22) node [anchor=north west][inner sep=0.75pt]    {$u_{-\mu }$};

        \end{tikzpicture}

        \caption{Illustration of the~floor-crossing changing the~clasping puncture from positive to negative.}
        \label{fig:modulifloorcrossing}
    \end{figure}
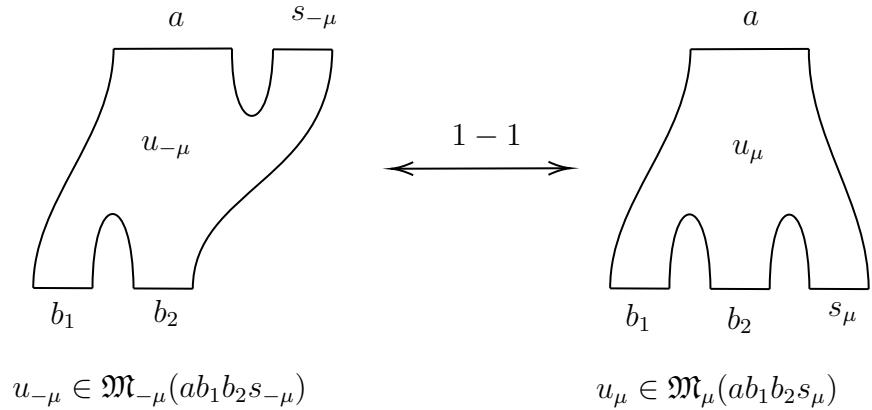
    Observe that (un)linking is an obvious example of the~(un)clasp move for disconnected Legendrians. Here, Theorem~\ref{thm:floorcrossing} follows from \cite[Lemma 3.19]{Avdek_2025}, where the~pseudo-holomorphic curves are counted in the~symplectization. The~main point is that even though changing the~sign of the~asymptotic does not fit into any standard analytic framework in the~symplectization, the~Fredholm index of the~problem in the~Lagrangian projection does not change as $\mu$ varies. One can then relate the~curves in the~symplectization and in the~Lagrangian projection via the~lifting procedure of \cite{DR16Lifting}. The~proof of \cite{Avdek_2025} relies heavily on the~fact that for unlinking, the~Lagrangian projection does not change, which is obviously untrue under the~clasp move of connected Legendrians.
    
    When dealing with pseudo-holomorphic discs with multiple positive boundary punctures, one can a~priori expect the~boundary of the~moduli spaces to degenerate so that contributions interact with the~string topology of the~Legendrian, see \cite{cieliebak2007rolestringtopologysymplectic}, \cite{Ng_2010}, \cite{dukic2024extensionchekanoveliashbergalgebrausing}, \cite{Avdek_2025}. We rule out these contributions in Section~\ref{sec:COMPACTNESS}.

    The~Lagrangian projections $L_\mu$ do not change too much, and so we prove the~floor-crossing considering the~count of pseudo-holomorphic curves in $P$ with boundary on $L_0$, the~Lagrangian projection of the~\textit{immersed} Legendrian $\Lambda_0$ representing the~instance of the~(un)clasp move when the~action of the~Reeb chord that we contract vanishes. The~proof largely follows the~analytical set-up of \cite{EEScontacthomology}, however, we have to modify it a~bit since the~original set-up uses the~embeddedness of the~Legendrian in several key constructions. After that, the~proof boils down to the~fact that the~Cauchy-Riemann equation on dicss with boundary punctures mapping the~boundary to the~Lagrangian projection of Legendrians forming the~(un)clasp move produces a~continuous path in the~space of Fredholm operators of certain regularity, which replaces the~argument of the~invariance of the~Lagrangian projection of \cite[Lemma 3.19]{Avdek_2025}.

    As an application of Theorem~\ref{thm:floorcrossing}, we generalise the~computation of Chekanov-Eliashberg algebra of twist knots from \cite{Meyer1999}, see Figure~\ref{fig:Meyertwist}, to arbitrary dimension.
    \begin{proposition}\label{prop:Poincare computation}
        Let $n\geq 2$ and $k\in\mathbb Z_{\geq 0}$. The~Chekanov--Eliashberg algebra $\mathcal A(\Lambda_k)$ of the~$n$-dimensional Legendrian $k$-twist sphere $\Lambda_k\subset\mathbb R^{2n+1}$ has precisely one dg-homotopy class of augmentations for $k>0$. For $k=0$, there are precisely two distinct dg-homotopy classes. Moreover, for every augmentation $\varepsilon$, its linearised Legendrian contact homology has Poincaré--Chekanov polynomial
        $$
            P_{\Lambda_k}(t)=\sum_j t^j\dim LCH_j^{\varepsilon}(\Lambda_k)=t^{-k}+t^n+t^{n+k-1}.
        $$
    \end{proposition}
    For references on Legendrian twist knots, see \cite{Meyer1999}, \cite{EpsteinFuchsMeyer2001}, \cite{Etnyre_2013}.
    \begin{figure}[htbp]
        \centering
        \tikzset{every picture/.style={line width=0.75pt}} 
        \setlength{\unitlength}{0.1\textwidth}
		\begin{picture}(10,1.8)
			\put(0.5,-0.2){\includegraphics[scale=1]{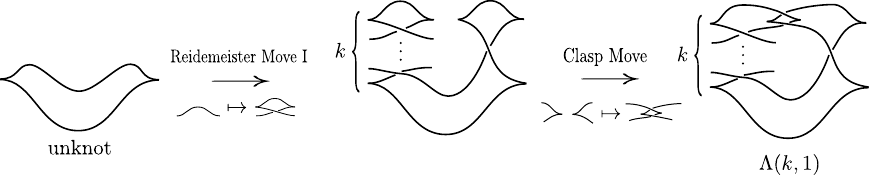}}
		\end{picture}
        \caption{Meyer's examples of Legendrian twist knots $\Lambda(k,l)$ in the~front projection.}
        \label{fig:Meyertwist}
    \end{figure}

    We say that a~Legendrian $\Lambda$ in $P\times\R$ is fillable if it is the~positive cyllindrical end of an exact embedded Lagrangian $L_\Lambda\subset\R\times P\times \R$ with vanishing Maslov class. The~filling $L_\Lambda$ induces an augmentation $\varepsilon_{L_\Lambda}$ of $\mathcal{A}(\Lambda)$. We say that the~augmentation $\varepsilon_{L_\Lambda}$ is a~geometric augmentation. 
    
    If there exists a~Hamiltonian isotopy $\phi_H^t$ in $P$ so that $\phi_H^1(\Pi_P(\L))$ and $\Pi_P(\Lambda)$ are disjoint, we say that $\Lambda$ is horizontally displaceable. Using the~Seidel-Ekholm-Dimitroglou Rizell isomorphism, see \cite{DR16Lifting}, and the~duality long exact sequence of linearised Legendrian homology of horizontally displaceable Legendrian, see \cite{dualityseq}, we provide examples of non-geometric augmentations of Legendrians. To the~best of the~author's knowledge, the~only examples of non-geometric augmentations of Legendrians of arbitrary dimension were produced in \cite{Golovko_2023}, using the~results on the~spinning construction, and so importing the~results from the~low-dimensional case to arbitrary dimension. Nevertheless, this approach restricts the~possible topological type of the~Legendrian to be a~product of spheres. Using the~ambient Legendrian surgery, see \cite{Surgery}, we are now able to produce plenty of non-geometric augmentations; what is more, we produce classes of Legendrians that are non-loose but non-fillable in arbitrary topological type of arbitrary dimension $n>1$.
    \begin{corollary}\label{cor:Nonfillability}
        Let $n\geq 2$ be even, let $k\in\mathbb Z_{\geq 0}$, let $P$ be an exact symplectic manifold of dimension $2n$ with finite geometry at infinity, and let $\Lambda^n\subset P\times\mathbb R$ be a~closed embedded horizontally displaceable Legendrian whose Chekanov--Eliashberg algebra admits an augmentation. Then there is a~family of Legendrians $\mathcal{T}(\Lambda,k)$ with the~following properties.
        \begin{itemize}
            \item All $\mathcal{T}(\Lambda,k)$ are diffeomorphic to $\Lambda$.
            \item Both $\Lambda$ and $\mathcal{T}(\Lambda,k)$ have the~same rotation class and the~same Thurston-Bennequin invariant.
            \item For $k>2$, all $\mathcal{T}(\Lambda,k)$ are mutually Legendrian non-isotopic and none of them is Legendrian isotopic to $\Lambda$.
            \item None of $\mathcal{T}(\Lambda,k)$ is loose.
            \item For $k>2$, none of $\mathcal{T}(\Lambda,k)$ admits an augmentation induced by an embedded exact Lagrangian filling in $\R\times P\times \R$ of Maslov class $0$.
        \end{itemize}
    \end{corollary}
    \subsection*{Acknowledgements} The~author thanks his supervisors Frédéric Bourgeois and Luís Diogo for their patience with his first steps involving pseudo-holomorphic curves and for their helpful advice during this project. The~author is grateful to Russell Avdek and Milica Đuki\'{c} for conversations about the~pseudo-holomorphic curves with multiple positive ends and words of encouragement. The~author also thanks to Georgios Dimitroglou Rizell and Tobias Ekholm for conversations about their work. Finally, the~author would like to thank Paolo Ghiggini for helpful comments that greatly improved the~exposition.

    \subsection*{Structure of the~paper} In Section~\ref{sec:Clasp move} we construct the~local model of the~higher-dimensional clasp-move and prove lemmas controlling basic properties of the~clasping chord and give a~contact topological proof of a~standard result from differential topology.

    In Section~\ref{sec:Applications}, we prove the~main applications of this paper. Namely, when performing a~clasp move produces a~loose Legendrian performing a~clasp move, we obtain bounds on Legendria unknotting number in Section~\ref{sec:Legunknotting}. For non-loose Legendrians, we provide a~generalisation of Legendrian twist spheres from a~particular presentation of a~Legendria unknot to arbitrary dimension in Section~\ref{sec:twist spheres} and compute their Chekanov-Eliashberg algebra. Proposition~\ref{prop:Poincare computation} and Corollary~\ref{cor:Nonfillability} are proved in Section~\ref{sec: Floor-crossing and CE}.
    
    In Section~\ref{sec: Main Res} we prove the~main result of this paper, that is Theorem~\ref{thm:floorcrossing}. We start with uniform compactness in Section~\ref{sec:COMPACTNESS}, and then discuss the~proof using the~implicit function theorem in Section~\ref{sec:analsketch}.

    Appendix~\ref{sec:Flowtreecomputation} contains pictorial details of the~proof of Lemma~\ref{lem:clefstrand}. Appendix~\ref{ap:Analysis} describes the~extended analytic set-up and several technical details of the~proof of Proposition~\ref{prop: line segment}.
    \subsection*{Declaration of AI use} ChatGPT, model GPT 5.5, was used to address grammatical and spelling errors.
    
\section{Clasp move}\label{sec:Clasp move}
    To construct a~twist knot, we have to change one overcrossing to a~undercrossing. This change is called the~clasp move and was first studied in the~Legendrian context by Pan and Rutherford in \cite{PR2019}, see Figure~\ref{fig:PRclasp}. The~trace of the~clasp move is an example of an exact immersed Lagrangian cobordism with one transverse self-intersection. Here we generalise the~clasp move to an arbitrary dimension.


    Let $\Lambda$ be a~disconnected, non-compact, two-sheeted Legendrian submanifold of $J^1(\R^n)$ whose sheets are the~zero section and the~$1$-jet graph of the~quadratic form
    \begin{equation*}
     h(\textbf{x})=-(x_1^2+\ldots+x_r^2)+(x_{r+1}^2+\ldots+x_n^2).
    \end{equation*}
    There is a~unique Reeb chord $s$ of action $0$ above the~origin corresponding to the~critical point of index $r$. Consult Figure~\ref{fig:localmodel} for illustration. \par We introduce a~plateau function $\rho:\R^n\to \R$ to locally shift the~zero section up or down, to resolve the~intersection point for $\mu=0$. So, we define the~front projection of the~local model $\Lambda_\mu$ via the~parametrising functions of the~upper and lower sheet, respectively,
    \begin{align*}
     h_u(\textbf{x})&=h(\textbf{x}),\\
      h_l(\textbf{x})&=\mu \rho(\textbf{x}),
    \end{align*}
    where $\rho$ is a~plateau function that interpolates between $0$ and $1$, $\rho(x)=1$ for $||x||<\varepsilon$, and $\rho$ is supported on the~disc $||x||<3\varepsilon$ centred at the~projection of the~self-intersection point.  

    The~local model contains a~Reeb chord for every point of the~base over which the~tangent planes are parallel, that is, when the~following two vectors in $\R^{n+1}$ are collinear $$(\nabla h,-1)=(\mu\nabla \rho,-1).$$ This happens over the~origin, where we obtain a~Reeb chord $s_\mu$ of action $|\mu|$.
    As we might observe in Figure~\ref{fig:localmodel} for $L_-$, for large values of $|\mu|$, one might introduce new Reeb chords. This may also occur if $d\rho$ is not chosen carefully.
    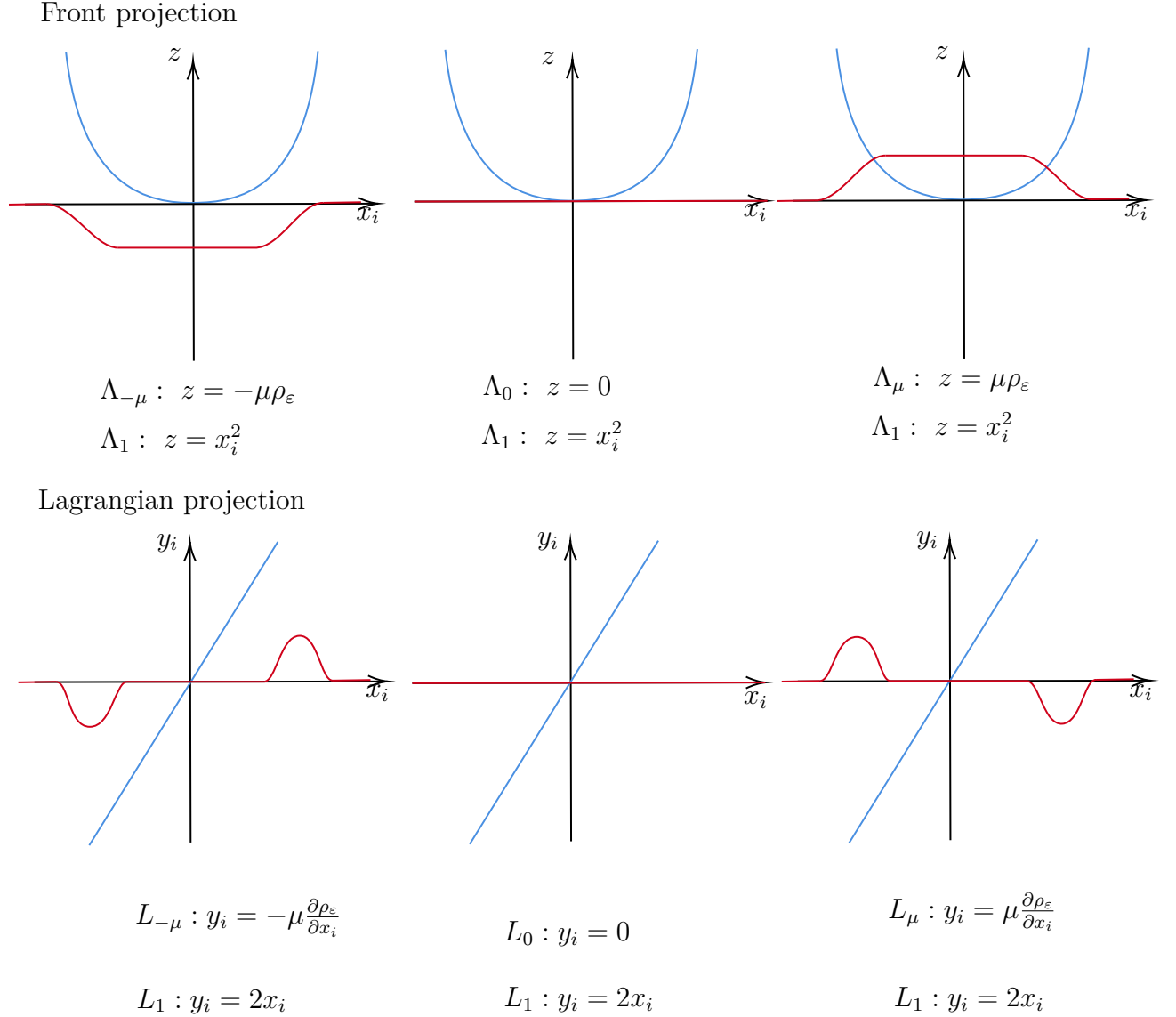
\begin{figure}[htbp!]
    \centering
    \tikzset{every picture/.style={line width=0.75pt}} 
    
    \begin{tikzpicture}[x=0.75pt,y=0.75pt,yscale=-1,xscale=1]
    
    \draw    (323.78,319.11) -- (324.2,486.73) ;
    \draw [shift={(323.78,317.11)}, rotate = 89.86] [color={rgb, 255:red, 0; green, 0; blue, 0 }  ][line width=0.75]    (10.93,-3.29) .. controls (6.95,-1.4) and (3.31,-0.3) .. (0,0) .. controls (3.31,0.3) and (6.95,1.4) .. (10.93,3.29)   ;
    \draw    (431.8,397.14) -- (234.6,397.53) ;
    \draw [shift={(433.8,397.13)}, rotate = 179.88] [color={rgb, 255:red, 0; green, 0; blue, 0 }  ][line width=0.75]    (10.93,-3.29) .. controls (6.95,-1.4) and (3.31,-0.3) .. (0,0) .. controls (3.31,0.3) and (6.95,1.4) .. (10.93,3.29)   ;
    \draw [color={rgb, 255:red, 74; green, 144; blue, 226 }  ,draw opacity=1 ]   (267,488.33) -- (373.4,317.13) ;
    \draw    (324.92,47.44) -- (325.33,215.07) ;
    \draw [shift={(324.91,45.44)}, rotate = 89.86] [color={rgb, 255:red, 0; green, 0; blue, 0 }  ][line width=0.75]    (10.93,-3.29) .. controls (6.95,-1.4) and (3.31,-0.3) .. (0,0) .. controls (3.31,0.3) and (6.95,1.4) .. (10.93,3.29)   ;
    \draw    (432.93,125.47) -- (235.73,125.87) ;
    \draw [shift={(434.93,125.47)}, rotate = 179.88] [color={rgb, 255:red, 0; green, 0; blue, 0 }  ][line width=0.75]    (10.93,-3.29) .. controls (6.95,-1.4) and (3.31,-0.3) .. (0,0) .. controls (3.31,0.3) and (6.95,1.4) .. (10.93,3.29)   ;
    \draw    (545.58,46.78) -- (546,214.4) ;
    \draw [shift={(545.58,44.78)}, rotate = 89.86] [color={rgb, 255:red, 0; green, 0; blue, 0 }  ][line width=0.75]    (10.93,-3.29) .. controls (6.95,-1.4) and (3.31,-0.3) .. (0,0) .. controls (3.31,0.3) and (6.95,1.4) .. (10.93,3.29)   ;
    \draw    (110.92,48.11) -- (111.33,215.73) ;
    \draw [shift={(110.91,46.11)}, rotate = 89.86] [color={rgb, 255:red, 0; green, 0; blue, 0 }  ][line width=0.75]    (10.93,-3.29) .. controls (6.95,-1.4) and (3.31,-0.3) .. (0,0) .. controls (3.31,0.3) and (6.95,1.4) .. (10.93,3.29)   ;
    \draw [color={rgb, 255:red, 74; green, 144; blue, 226 }  ,draw opacity=1 ]   (253.2,40.2) .. controls (261.8,118.2) and (300.6,125.8) .. (325.4,125.4) .. controls (350.2,125) and (387,119.4) .. (395.4,39.8) ;
    \draw [color={rgb, 255:red, 74; green, 144; blue, 226 }  ,draw opacity=1 ]   (39.2,41.4) .. controls (47.8,119.4) and (86.6,127) .. (111.4,126.6) .. controls (123.9,126.4) and (139.46,124.87) .. (152.89,113.21) .. controls (166.09,101.74) and (177.24,80.46) .. (181.4,41) ;
    \draw [color={rgb, 255:red, 74; green, 144; blue, 226 }  ,draw opacity=1 ]   (474,39.53) .. controls (482.6,117.53) and (521.4,125.13) .. (546.2,124.73) .. controls (571,124.33) and (607.8,118.73) .. (616.2,39.13) ;
    \draw    (537.78,318.44) -- (538.2,486.07) ;
    \draw [shift={(537.78,316.44)}, rotate = 89.86] [color={rgb, 255:red, 0; green, 0; blue, 0 }  ][line width=0.75]    (10.93,-3.29) .. controls (6.95,-1.4) and (3.31,-0.3) .. (0,0) .. controls (3.31,0.3) and (6.95,1.4) .. (10.93,3.29)   ;
    \draw [color={rgb, 255:red, 74; green, 144; blue, 226 }  ,draw opacity=1 ]   (481,487.67) -- (587.4,316.47) ;
    \draw    (109.12,319.78) -- (109.53,487.4) ;
    \draw [shift={(109.11,317.78)}, rotate = 89.86] [color={rgb, 255:red, 0; green, 0; blue, 0 }  ][line width=0.75]    (10.93,-3.29) .. controls (6.95,-1.4) and (3.31,-0.3) .. (0,0) .. controls (3.31,0.3) and (6.95,1.4) .. (10.93,3.29)   ;
    \draw [color={rgb, 255:red, 74; green, 144; blue, 226 }  ,draw opacity=1 ]   (52.33,489) -- (158.73,317.8) ;
    \draw    (649.27,396.07) -- (452.07,396.47) ;
    \draw [shift={(651.27,396.07)}, rotate = 179.88] [color={rgb, 255:red, 0; green, 0; blue, 0 }  ][line width=0.75]    (10.93,-3.29) .. controls (6.95,-1.4) and (3.31,-0.3) .. (0,0) .. controls (3.31,0.3) and (6.95,1.4) .. (10.93,3.29)   ;
    \draw [color={rgb, 255:red, 208; green, 2; blue, 27 }  ,draw opacity=1 ]   (464.33,396.27) .. controls (471.67,396.6) and (473,371.42) .. (485.33,371.42) .. controls (497.67,371.42) and (497.84,396.1) .. (504.33,396.27) ;
    \draw [color={rgb, 255:red, 208; green, 2; blue, 27 }  ,draw opacity=1 ]   (581.33,396.27) .. controls (588.38,396.59) and (590.61,421.15) .. (601.33,420.42) .. controls (612.06,419.69) and (613.58,395.43) .. (620.33,395.6) ;
    \draw [color={rgb, 255:red, 208; green, 2; blue, 27 }  ,draw opacity=1 ]   (504.33,396.27) -- (581.33,396.27) ;
    \draw [color={rgb, 255:red, 208; green, 2; blue, 27 }  ,draw opacity=1 ]   (443,396.6) -- (464.33,396.27) ;
    \draw [color={rgb, 255:red, 208; green, 2; blue, 27 }  ,draw opacity=1 ]   (620.33,395.6) -- (641.67,395.27) ;
    \draw [color={rgb, 255:red, 208; green, 2; blue, 27 }  ,draw opacity=1 ]   (234.6,397.53) -- (433.8,397.13) ;
    \draw    (218.6,396.52) -- (21.4,396.91) ;
    \draw [shift={(220.6,396.51)}, rotate = 179.88] [color={rgb, 255:red, 0; green, 0; blue, 0 }  ][line width=0.75]    (10.93,-3.29) .. controls (6.95,-1.4) and (3.31,-0.3) .. (0,0) .. controls (3.31,0.3) and (6.95,1.4) .. (10.93,3.29)   ;
    \draw [color={rgb, 255:red, 208; green, 2; blue, 27 }  ,draw opacity=1 ]   (33.67,396.71) .. controls (41,397.04) and (40.33,422.11) .. (52.67,422.11) .. controls (65,422.11) and (67.17,396.54) .. (73.67,396.71) ;
    \draw [color={rgb, 255:red, 208; green, 2; blue, 27 }  ,draw opacity=1 ]   (150.67,396.71) .. controls (157.72,397.03) and (159.94,371.51) .. (170.67,370.78) .. controls (181.39,370.05) and (182.91,395.87) .. (189.67,396.04) ;
    \draw [color={rgb, 255:red, 208; green, 2; blue, 27 }  ,draw opacity=1 ]   (73.67,396.71) -- (150.67,396.71) ;
    \draw [color={rgb, 255:red, 208; green, 2; blue, 27 }  ,draw opacity=1 ]   (12.33,397.04) -- (33.67,396.71) ;
    \draw [color={rgb, 255:red, 208; green, 2; blue, 27 }  ,draw opacity=1 ]   (189.67,396.04) -- (211,395.71) ;
    \draw    (213.27,127.05) -- (16.07,127.44) ;
    \draw [shift={(215.27,127.04)}, rotate = 179.88] [color={rgb, 255:red, 0; green, 0; blue, 0 }  ][line width=0.75]    (10.93,-3.29) .. controls (6.95,-1.4) and (3.31,-0.3) .. (0,0) .. controls (3.31,0.3) and (6.95,1.4) .. (10.93,3.29)   ;
    \draw [color={rgb, 255:red, 208; green, 2; blue, 27 }  ,draw opacity=1 ]   (28.33,127.24) .. controls (43,127.91) and (55.33,151.58) .. (68.33,151.91) ;
    \draw [color={rgb, 255:red, 208; green, 2; blue, 27 }  ,draw opacity=1 ]   (145.33,151.91) .. controls (160,152.58) and (171.33,126.24) .. (184.33,126.58) ;
    \draw [color={rgb, 255:red, 208; green, 2; blue, 27 }  ,draw opacity=1 ]   (68.33,151.91) -- (145.33,151.91) ;
    \draw [color={rgb, 255:red, 208; green, 2; blue, 27 }  ,draw opacity=1 ]   (7,127.58) -- (28.33,127.24) ;
    \draw [color={rgb, 255:red, 208; green, 2; blue, 27 }  ,draw opacity=1 ]   (184.33,126.58) -- (205.67,126.24) ;
    \draw    (646.6,125.05) -- (449.4,125.44) ;
    \draw [shift={(648.6,125.04)}, rotate = 179.88] [color={rgb, 255:red, 0; green, 0; blue, 0 }  ][line width=0.75]    (10.93,-3.29) .. controls (6.95,-1.4) and (3.31,-0.3) .. (0,0) .. controls (3.31,0.3) and (6.95,1.4) .. (10.93,3.29)   ;
    \draw [color={rgb, 255:red, 208; green, 2; blue, 27 }  ,draw opacity=1 ]   (461.67,125.24) .. controls (476.33,125.91) and (488.67,99.58) .. (501.67,99.91) ;
    \draw [color={rgb, 255:red, 208; green, 2; blue, 27 }  ,draw opacity=1 ]   (578.67,99.91) .. controls (593.33,100.58) and (604.67,124.24) .. (617.67,124.58) ;
    \draw [color={rgb, 255:red, 208; green, 2; blue, 27 }  ,draw opacity=1 ]   (501.67,99.91) -- (578.67,99.91) ;
    \draw [color={rgb, 255:red, 208; green, 2; blue, 27 }  ,draw opacity=1 ]   (440.33,125.58) -- (461.67,125.24) ;
    \draw [color={rgb, 255:red, 208; green, 2; blue, 27 }  ,draw opacity=1 ]   (617.67,124.58) -- (639,124.24) ;
    \draw [color={rgb, 255:red, 208; green, 2; blue, 27 }  ,draw opacity=1 ]   (235.73,125.87) -- (434.93,125.47) ;
    
    \draw (303.67,311.07) node [anchor=north west][inner sep=0.75pt]    {$y_{i}$};
    \draw (22,286.67) node [anchor=north west][inner sep=0.75pt]   [align=left] {Lagrangian projection};
    \draw (420,125.8) node [anchor=north west][inner sep=0.75pt]    {$x_{i}$};
    \draw (517.67,310.4) node [anchor=north west][inner sep=0.75pt]    {$y_{i}$};
    \draw (89,311.73) node [anchor=north west][inner sep=0.75pt]    {$y_{i}$};
    \draw (501.47,514.2) node [anchor=north west][inner sep=0.75pt]    {$L_{\mu} :y_{i} =\mu \frac{\partial \rho _{\varepsilon }}{\partial x_{i}}$};
    \draw (637.53,396.53) node [anchor=north west][inner sep=0.75pt]    {$x_{i}$};
    \draw (504.8,567.53) node [anchor=north west][inner sep=0.75pt]    {$L_{1} :y_{i} =2x_{i}$};
    \draw (284.8,530.2) node [anchor=north west][inner sep=0.75pt]    {$L_{0} :y_{i} =0$};
    \draw (284.8,567.53) node [anchor=north west][inner sep=0.75pt]    {$L_{1} :y_{i} =2x_{i}$};
    \draw (94.93,37.84) node [anchor=north west][inner sep=0.75pt]    {$z$};
    \draw (206.87,396.98) node [anchor=north west][inner sep=0.75pt]    {$x_{i}$};
    \draw (76.8,517.53) node [anchor=north west][inner sep=0.75pt]    {$L_{-\mu} :y_{i} =-\mu \frac{\partial \rho _{\varepsilon }}{\partial x_{i}}$};
    \draw (77.47,568.2) node [anchor=north west][inner sep=0.75pt]    {$L_{1} :y_{i} =2x_{i}$};
    \draw (201.53,127.51) node [anchor=north west][inner sep=0.75pt]    {$x_{i}$};
    \draw (634.87,125.51) node [anchor=north west][inner sep=0.75pt]    {$x_{i}$};
    \draw (420.2,400.98) node [anchor=north west][inner sep=0.75pt]    {$x_{i}$};
    \draw (23.33,11.67) node [anchor=north west][inner sep=0.75pt]   [align=left] {Front projection};
    \draw (57,224.4) node [anchor=north west][inner sep=0.75pt]    {$\Lambda _{-\mu} :\ z=-\mu \rho _{\varepsilon }$};
    \draw (57,250.4) node [anchor=north west][inner sep=0.75pt]    {$\Lambda _{1} :\ z=x_{i}^{2}$};
    \draw (273,222.4) node [anchor=north west][inner sep=0.75pt]    {$\Lambda _{0} :\ z=0$};
    \draw (272,248.4) node [anchor=north west][inner sep=0.75pt]    {$\Lambda _{1} :\ z=x_{i}^{2}$};
    \draw (493,217.4) node [anchor=north west][inner sep=0.75pt]    {$\Lambda _{\mu} :\ z=\mu \rho _{\varepsilon }$};
    \draw (492,243.4) node [anchor=north west][inner sep=0.75pt]    {$\Lambda _{1} :\ z=x_{i}^{2}$};
    \draw (305.93,39.84) node [anchor=north west][inner sep=0.75pt]    {$z$};
    \draw (527.93,36.84) node [anchor=north west][inner sep=0.75pt]    {$z$};
    \end{tikzpicture}
    \caption{Illustration of the~local model for $r+1\leq i\leq n$. The~objects indexed by $1$ are blue, and the~objects indexed by $0$ are red.}
    \label{fig:localmodel}
    \end{figure}

    \begin{lemma}\label{lem:nonewReebchords}
     There is a~choice of $\rho_\varepsilon:\R^n\to \R$ and $\mu_0>0$ so that no new Reeb chords of $\Lambda$ are introduced during the~clasp move $L_\mu=L+\mu d\rho_\varepsilon$ for all $|\mu|<\mu_0$.
    \end{lemma}
    \begin{proof}
     The~new chord would be created when the~normal lines to the~sheets over the~same point were co-linear. That is, $\nabla h=\mu\nabla \rho$. We can construct $\rho$ so that the~interior of the~support of $\nabla\rho$ is $A_\varepsilon$, the~annular region $\varepsilon<||x||<3\varepsilon$ in $\R^{n}$ so that the~maximal $||\nabla\rho||$ is attained along the~locus $||x||=2\varepsilon$. Moreover, as we move along the~rays from the~origin, the~direction of $\nabla\rho$ is constant, but its norm depends on the~radial coordinate.

    Now, for simplicity, we assume that $||\nabla h||\neq 0$ except at the~origin. Consider the~graphs
    $$\Gamma_f=\Big\lbrace\frac{(\nabla f(x),-1)}{||(\nabla f(x),-1)||}:x\in A_\varepsilon\Big\rbrace\subset A_\varepsilon\times S^{n}$$ for $f=h,\mu\rho$. Note that any Reeb chord over $A_\varepsilon$ will be the~intersection of these two graphs. As the~direction is constant along each radial ray, we see that we can retract the~graphs onto the~graphs $\Gamma^\prime$ over the~$(n-1)$-sphere $||x||=2\varepsilon$, without the~danger of introducing new intersection points, so now $\Gamma^\prime_h,\Gamma^\prime_{\mu\rho}\subset S^{n-1}\times S^{n}$. Note that the~graph $\Gamma^\prime_{\mu\rho}$ is just $S^{n-1}\times (0,\dots,0,-1)$ for $\mu=0$. And as $|\mu|$ grows, the~distance of the~points of $\Gamma^\prime_{\mu\rho}$ from $(0,\dots,0,-1)$ grows. Finally, $\Gamma^\prime_h$ does not hit the~point $(0,\dots,0,-1)$. Therefore, there exists a~$n$-dics of points in $S^{n}$ that have strictly smaller distance than any point in $\Gamma^\prime_h$, which provides us with our bound $\mu_0$. We can embed in this dics as we please without intersecting the~$\Gamma^\prime_h$.

    Observe, that this proof applies as well to the~set-up when we have multiple sheets $h_1,\dots,h_k$ that might intersect each other. Then we would take the~minimum distance over all the~points of $\Gamma^\prime_{h_1}\cup \dots\Gamma^\prime_{h_k}$.
            \end{proof}
            \begin{definition}\label{def:clasp move}
                We call $(\Lambda_\mu)_{|\mu|<\mu_0}$ the~clasp move of $\Lambda$ centred at the~chord $s_\mu$, which we call $s_\mu$ the~clasping chord. Reversing the~parameter $\mu$, we obtain the~unclasp move.
            \end{definition}

            To perform the~clasp move in contactisations or general contact manifolds, one needs to choose a~Weinstein neighbourhood of the~Legendrian sufficiently close to the~clasping Reeb chord so that the~clasping Reeb chord lies within this neighbourhood. Then a~standard contactomorphism allows us to implant the~local model from $J^1(\R^n)$, which we defined above.
            
            \begin{lemma}\label{lem:CZ lemma}
                The~action and the~Conley-Zehnder index of the~clasping chord can be expressed as follows:
                \begin{equation*}
                   \mathfrak{a}(s_\mu)=|\mu| \,\,\,\text{ and }\,\,\, CZ(s_\mu)=\begin{cases}
                        r+\mathfrak{m}(u)-\mathfrak{m}(l),\,\, \text{ if } \mu>0,\\
                        (n-r)+\mathfrak{m}(l)-\mathfrak{m}(u),\,\, \text{ if } \mu<0.
                    \end{cases}
                \end{equation*}
                Here, $\mathfrak{m}$ denotes the~Maslov potential of the~sheets in the~local model. The~upper sheet and lower sheet are denoted by $u$ and $l$, respectively. In particular, for $\mu>0$,
                 \begin{align}
                    CZ(s_{\mu})&=n-CZ(s_{-\mu}),\\
                    |s_\mu|&=(n-2)-|s_{-\mu}|,
                \end{align}
                where for a~Reeb chord $c$ the~grading is defined as $|c|=CZ(c)-1$.
            \end{lemma} Proof follows from \cite[Lemma 3.4]{EESnoniso}.
            \begin{corollary}
                The~Thurston-Bennequin invariant of $\Lambda\subset\R^{2n+1}$ does not change under the~clasp move if $n$ is even. Moreover, 
                \begin{equation*}
                    tb(\Lambda_\mu)=tb(\Lambda_{-\mu})+2(-1)^{\frac{(n-2)(n-1)}{2}}(-1)^{|s_{-\mu}|+1}, \text{ when } n \text{ is odd.}
                \end{equation*}
            \end{corollary}
            \begin{proof}
                 We know that $$tb(\Lambda)=(-1)^{\frac{n}{2}+1}\frac{1}{2}\chi(\Lambda)$$ when $n$ is even (see \cite[Proposition 3.2]{EESnoniso}). As clasp move does not change the~topological type of $\L$, it cannot change the~Thurston-Bennequin invariant.
                 
                 Now, if $n$ is odd, by \cite[Proposition 3.3]{EESnoniso}, we have:
                $$tb(\Lambda)=(-1)^{\frac{(n-2)(n-1)}{2}}\sum_{c\in\mathcal{R}(\Lambda)}(-1)^{|c|}.$$
                So,
                \begin{align*}
                    tb(\Lambda_\mu)-tb(\Lambda_{-\mu})&=(-1)^{\frac{(n-2)(n-1)}{2}}((-1)^{(n-2)-|s_{-\mu}|}-(-1)^{|s_{-\mu}|}),\\
                    &=(-1)^{\frac{(n-2)(n-1)}{2}}(-1)^{|s_{-\mu}|}((-1)^{(n-2)}-1),\\
                    &=-(-1)^{\frac{(n-2)(n-1)}{2}}2(-1)^{|s_{-\mu}|}
                    =2(-1)^{\frac{(n-2)(n-1)}{2}}(-1)^{|s_{-\mu}|+1}.
                \end{align*}
            \end{proof}
            
            \begin{lemma}\label{lem:invarofrot}
                The~clasp move does not change the~rotation class.
            \end{lemma}
            \begin{proof}
                The~clasp move produces a~regular homotopy through Legendrian immersions. The~rotation class is an invariant up to regular homotopy through Legendrian immersions. For more details, see Section 3.3 in \cite{EESnoniso}.
            \end{proof}
             
             The~following lemma uses the~higher-dimensional version of the~Legendrian clasp move to reprove a~differential-topological statement.  
                    \begin{lemma}[see \cite{Wu1958}]\label{lem:nosmoothknotting}
                        Any two embeddings of a~smooth closed manifold $M^n$ into $\R^{2n+1}$ are smoothly isotopic if $n>1$.
                    \end{lemma}
                    We follow the~standard proof strategy, which involves first proving that the~two embeddings are isotopic through immersions, then one modifies the~path of immersions so that it passes through embeddings except for a~finite number of instances. Finally, if necessary, one modifies this path through immersions using a~local model, introducing an auxiliary instance when the~path fails to pass through embeddings. The~Whitney trick then finishes the~proof. In our proof, the~contact geometry is used to introduce these new intersection points and to control their sign. This, for the~contact topological audience, simplifies the~constructions of local models in differential topology for which the~sign control is more involved.
                    \begin{proof}
                        Obviously, any two continuous maps $M^n\to \R^{2n+1}$ are continuously homotopic. By Whitney Immersion theorem, this continuous homotopy can be upgraded to a~regular homotopy $H:M^n\times I\to \R^{2n+1}$ through immersions. Unfortunately, we are one codimension short to use the~Whitney embedding theorem. Consider the~trace of $H$ in $R^{2n+2}$. By a~standard general position argument, this can be made into an immersion $M^n\times I\to \R^{2n+2}$ with $\mathcal{P}$ a~finite set of transverse self-intersection points.
                        Recall the~Whitney invariant is an algebraic intersection number of $H$
                        $$I_W(H)=\sum_{p\in\mathcal{P}}I_H(p)\in\begin{cases}
                            \Z, \text{ if } n+1 \text{ even,
                            } \\
                            \Z_2, \text{ if } n+1 \text{ odd.
                            }
                        \end{cases}$$
                        By Whitney, $H$ is an isotopy if $I_W(H)=0$. But $H$ can be modified via the~surgery of a~Whitney dics carefully so that the~surgered manifold is again a~trace of a~family of smooth maps $M^n\to \R^{2n+1}$. Each surgery eliminates two intersection points (with different sign, if defined).
                        \par In case $n+1$ is odd and the~cardinality of $\mathcal{P}$ is even, we are done and we have an isotopy.
                        \par In case $n+1$ is even, split $\mathcal{P}=\mathcal{P}_+\cup \mathcal{P}_-$ into $\mathcal{P}_+$ the~self-intersection points contributing to $I_W(H)$ with positive integer, similarly for $\mathcal{P}_-$. If the~cardinalities of $\mathcal{P}_+$ and $\mathcal{P}_-$ are the~same, then we are done.
            
                        The~remaining cases will be tackled by introducing a~new self-intersection point (of the~required sign). This will be done using contact topology. First choose a~neighbourhood of a~point and isotope this neighbourhood so that it looks like a~plane. Then, rotate the~plane so that it is an open subset of a~Legendrian plane in $\R^{2n+1}$. Then perform the~higher-dimensional equivalent of the~Reidemeister I move in $\R^{2n+1}$, consult Figure~\ref{fig:Reidemeistermove} for low-dimensional examples.
                        \begin{figure}
                            \centering
                            \tikzset{every picture/.style={line width=0.75pt}} 
                            \begin{tikzpicture}[x=0.75pt,y=0.75pt,yscale=-1,xscale=1]
            
            \draw    (141.27,40.89) .. controls (154.67,40.36) and (153.04,23.93) .. (168.18,24.2) .. controls (183.33,24.46) and (179.13,40.09) .. (192.88,41.15) ;
            \draw    (141.27,40.89) .. controls (161.19,40.89) and (175.64,72.41) .. (205,73.47) ;
            \draw    (169.35,54.13) .. controls (173.78,50.42) and (184.26,40.36) .. (192.88,41.15) ;
            \draw    (130.2,74) .. controls (140.22,73.21) and (153.97,68.44) .. (166.67,56.91) ;
            \draw    (286.47,81) -- (499.13,80.66) ;
            \draw    (340.05,26.76) .. controls (351.41,26.25) and (350.03,10.43) .. (362.87,10.69) .. controls (375.72,10.94) and (372.16,25.99) .. (383.82,27.01) ;
            \draw    (340.05,26.76) .. controls (356.95,26.76) and (369.2,57.12) .. (394.09,58.14) ;
            \draw    (363.86,39.51) .. controls (367.62,35.94) and (376.51,26.25) .. (383.82,27.01) ;
            \draw    (330.67,58.65) .. controls (339.16,57.88) and (350.82,53.29) .. (361.59,42.19) ;
            \draw    (458.11,26.76) .. controls (469.47,26.25) and (468.09,10.43) .. (480.93,10.69) .. controls (493.77,10.94) and (490.22,25.99) .. (501.87,27.01) ;
            \draw    (458.11,26.76) .. controls (475,26.76) and (487.25,57.12) .. (512.15,58.14) ;
            \draw    (481.92,39.51) .. controls (485.67,35.94) and (494.56,26.25) .. (501.87,27.01) ;
            \draw    (448.72,58.65) .. controls (457.22,57.88) and (468.88,53.29) .. (479.65,42.19) ;
            \draw  [line width=0.75]  (340.05,26.76) .. controls (340.05,19.71) and (376.26,14) .. (420.93,14) .. controls (465.6,14) and (501.81,19.71) .. (501.81,26.76) .. controls (501.81,33.8) and (465.6,39.51) .. (420.93,39.51) .. controls (376.26,39.51) and (340.05,33.8) .. (340.05,26.76) -- cycle ;
            \draw  [line width=0.75]  (383.82,26.76) .. controls (383.82,23.71) and (400.45,21.23) .. (420.96,21.23) .. controls (441.48,21.23) and (458.11,23.71) .. (458.11,26.76) .. controls (458.11,29.81) and (441.48,32.29) .. (420.96,32.29) .. controls (400.45,32.29) and (383.82,29.81) .. (383.82,26.76) -- cycle ;
            \draw  [line width=0.75]  (362.87,10.69) .. controls (362.87,9.68) and (389.3,8.87) .. (421.9,8.87) .. controls (454.5,8.87) and (480.93,9.68) .. (480.93,10.69) .. controls (480.93,11.69) and (454.5,12.5) .. (421.9,12.5) .. controls (389.3,12.5) and (362.87,11.69) .. (362.87,10.69) -- cycle ;
            \draw    (346.33,4.34) -- (559,4) ;
            \draw    (346.33,4.34) -- (286.47,81) ;
            \draw    (559,4) -- (499.13,80.66) ;
            \draw    (128,161) -- (207.33,161.33) ;
            \draw    (285.97,232.67) -- (498.63,232.33) ;
            \draw    (345.83,156.01) -- (558.5,155.67) ;
            \draw    (345.83,156.01) -- (285.97,232.67) ;
            \draw    (558.5,155.67) -- (498.63,232.33) ;
            \draw    (168.33,147) -- (168.97,101.67) ;
            \draw [shift={(169,99.67)}, rotate = 90.81] [color={rgb, 255:red, 0; green, 0; blue, 0 }  ][line width=0.75]    (10.93,-3.29) .. controls (6.95,-1.4) and (3.31,-0.3) .. (0,0) .. controls (3.31,0.3) and (6.95,1.4) .. (10.93,3.29)   ;
            \draw    (418.83,144) -- (419.47,98.67) ;
            \draw [shift={(419.5,96.67)}, rotate = 90.81] [color={rgb, 255:red, 0; green, 0; blue, 0 }  ][line width=0.75]    (10.93,-3.29) .. controls (6.95,-1.4) and (3.31,-0.3) .. (0,0) .. controls (3.31,0.3) and (6.95,1.4) .. (10.93,3.29)   ;

            \end{tikzpicture}
                \caption{The~front projection of the~Reidemeister I move of a~Legendrian knot on the~left. On the~right, the~front projection of a~higher-dimensional equivalent of the~first Reidemeister move for a~$2$-dimensional Legendrian in $\R^5$.}
                \label{fig:Reidemeistermove}
            \end{figure}
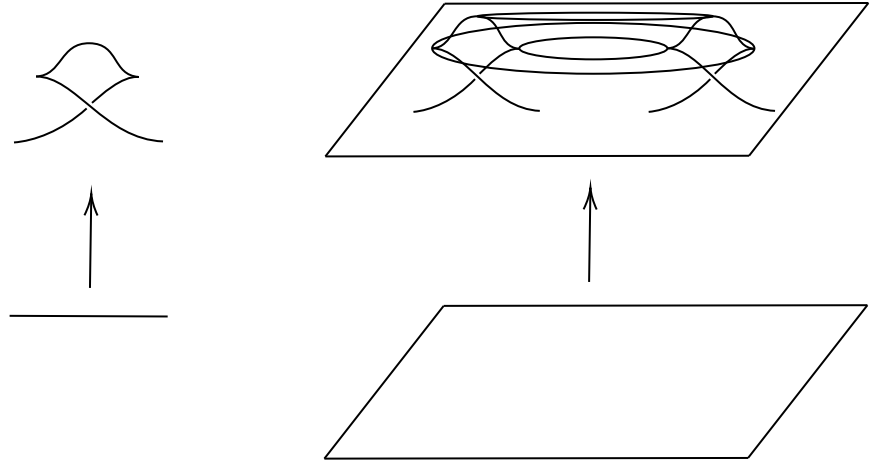
            Now, we will define two regular homotopies $\Sigma^\pm_t$, that vary through Legendrian embeddings except for $t=0$. Fix the~Maslov potential on the~lowest sheet to be $0$.
            
            We construct the~$\Sigma^+$ move as the~concatenation of several moves (for low-dimensional example see Figure~\ref{fig:smthclasp}):
            \begin{enumerate}
                \item introduce a~small bump on the~lowest sheet, this move is an isotopy and so the~value of $I_W$ does not change,
                \item move the~top side of the~bump upwards, until we clasp the~sheet of Maslov potential $0$ with the~other sheet with Maslov potential $0$, this is a~regular homotopy with precisely one transverse self-intersection point that appears at $t=0$.
            \end{enumerate}
            Similarly, the~$\Sigma^-$ move is given by composition (see Figure~\ref{fig:smthclasp}) 
            \begin{enumerate}
                \item introduce a~small negative bump on the~sheet with Maslov potential $1$, this move is an isotopy and so the~value of $I_W$ does not change,
                \item move the~bottom side of the~bump downwards, until we clasp the~sheet of Maslov potential $1$ with the~sheet with Maslov potential $0$ below. That is again a~regular homotopy with precisely one transverse self-intersection point that appears at $t=0$.
            \end{enumerate}
            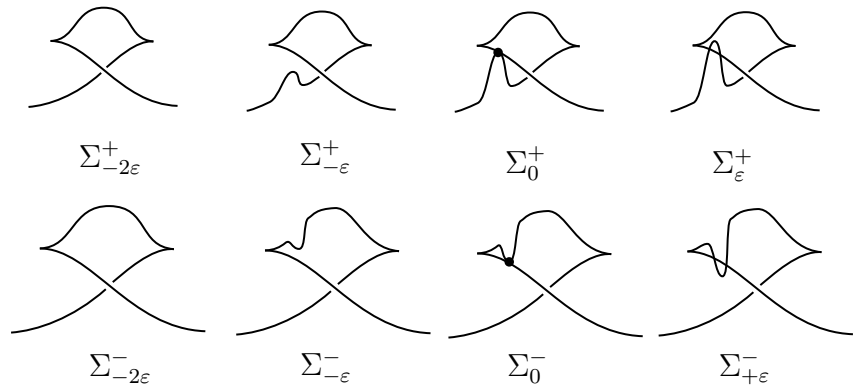
\begin{figure}[htbp]
                \centering
                    \tikzset{every picture/.style={line width=0.75pt}} 
                    
                    \begin{tikzpicture}[x=0.75pt,y=0.75pt,yscale=-1,xscale=1]
                    
                    \draw    (157.32,29.89) .. controls (170.72,29.36) and (169.09,12.93) .. (184.23,13.2) .. controls (199.38,13.46) and (195.18,29.09) .. (208.93,30.15) ;
                    \draw    (157.32,29.89) .. controls (177.24,29.89) and (191.69,61.41) .. (221.05,62.47) ;
                    \draw    (185.4,43.13) .. controls (189.83,39.42) and (200.31,29.36) .. (208.93,30.15) ;
                    \draw    (146.25,63) .. controls (156.27,62.21) and (170.02,57.44) .. (182.72,45.91) ;
                    \draw    (266.82,31.8) .. controls (280.22,31.27) and (278.59,14.85) .. (293.73,15.12) .. controls (308.88,15.38) and (304.68,31.01) .. (318.43,32.07) ;
                    \draw    (266.82,31.8) .. controls (286.74,31.8) and (301.19,63.33) .. (330.55,64.39) ;
                    \draw    (294.9,45.05) .. controls (299.33,41.34) and (309.81,31.27) .. (318.43,32.07) ;
                    \draw    (255.75,64.92) .. controls (257.83,64.75) and (262.17,62.97) .. (267.21,60.59) .. controls (272.26,58.21) and (274.5,45.67) .. (278.75,45.42) .. controls (283,45.17) and (280.58,52.1) .. (284.17,52.01) .. controls (287.75,51.92) and (290.37,48.75) .. (292.22,47.83) ;
                    \draw    (371.24,32.3) .. controls (384.63,31.77) and (383,15.35) .. (398.15,15.62) .. controls (413.3,15.88) and (409.1,31.51) .. (422.85,32.57) ;
                    \draw    (371.24,32.3) .. controls (391.16,32.3) and (405.61,63.83) .. (434.97,64.89) ;
                    \draw    (399.31,45.55) .. controls (403.74,41.84) and (414.23,31.77) .. (422.85,32.57) ;
                    \draw    (360.17,65.42) .. controls (362.25,65.25) and (366.58,63.47) .. (371.63,61.09) .. controls (376.68,58.71) and (377.67,35.92) .. (381.92,35.67) .. controls (386.17,35.42) and (385,52.6) .. (388.58,52.51) .. controls (392.17,52.42) and (394.78,49.25) .. (396.63,48.33) ;
                    \draw    (479.57,32.14) .. controls (492.97,31.61) and (491.34,15.18) .. (506.48,15.45) .. controls (521.63,15.71) and (517.43,31.34) .. (531.18,32.4) ;
                    \draw    (479.57,32.14) .. controls (499.49,32.14) and (513.94,63.66) .. (543.3,64.72) ;
                    \draw    (507.65,45.38) .. controls (512.08,41.67) and (522.56,31.61) .. (531.18,32.4) ;
                    \draw    (468.5,65.25) .. controls (470.58,65.09) and (474.92,63.31) .. (479.96,60.93) .. controls (485.01,58.54) and (486.25,30.25) .. (490.5,30) .. controls (494.75,29.75) and (493.33,52.43) .. (496.92,52.34) .. controls (500.5,52.25) and (503.12,49.08) .. (504.97,48.16) ;
                    \draw  [fill={rgb, 255:red, 0; green, 0; blue, 0 }  ,fill opacity=1 ] (380.04,35.67) .. controls (380.04,34.63) and (380.88,33.79) .. (381.92,33.79) .. controls (382.95,33.79) and (383.79,34.63) .. (383.79,35.67) .. controls (383.79,36.7) and (382.95,37.54) .. (381.92,37.54) .. controls (380.88,37.54) and (380.04,36.7) .. (380.04,35.67) -- cycle ;
                    \draw    (151.93,134.03) .. controls (169.39,133.36) and (167.27,112.42) .. (187.01,112.75) .. controls (206.75,113.09) and (201.29,133.02) .. (219.21,134.37) ;
                    \draw    (151.93,134.03) .. controls (177.9,134.03) and (196.73,174.22) .. (235,175.57) ;
                    \draw    (188.53,150.92) .. controls (194.3,146.19) and (207.97,133.36) .. (219.21,134.37) ;
                    \draw    (137.5,176.25) .. controls (150.56,175.24) and (168.48,169.16) .. (185.04,154.47) ;
                    \draw    (264.93,135.03) .. controls (269.4,134.86) and (272.54,133.4) .. (275.19,131.23) .. controls (277.83,129.05) and (278.29,134.23) .. (282,134.25) .. controls (285.71,134.27) and (283.96,120.55) .. (286.75,118.75) .. controls (289.54,116.95) and (290.25,114.25) .. (300.01,113.75) .. controls (309.77,113.26) and (314.29,134.02) .. (332.21,135.37) ;
                    \draw    (264.93,135.03) .. controls (290.9,135.03) and (309.73,175.22) .. (348,176.57) ;
                    \draw    (301.53,151.92) .. controls (307.3,147.19) and (320.97,134.36) .. (332.21,135.37) ;
                    \draw    (250.5,177.25) .. controls (263.56,176.24) and (281.48,170.16) .. (298.04,155.47) ;
                    \draw    (371.43,136.53) .. controls (375.9,136.36) and (379.04,134.9) .. (381.69,132.73) .. controls (384.33,130.55) and (383.79,140.73) .. (387.5,140.75) .. controls (391.21,140.77) and (390.46,122.05) .. (393.25,120.25) .. controls (396.04,118.45) and (396.75,115.75) .. (406.51,115.25) .. controls (416.27,114.76) and (420.79,135.52) .. (438.71,136.87) ;
                    \draw    (371.43,136.53) .. controls (397.4,136.53) and (416.23,176.72) .. (454.5,178.07) ;
                    \draw    (408.03,153.42) .. controls (413.8,148.69) and (427.47,135.86) .. (438.71,136.87) ;
                    \draw    (357,178.75) .. controls (370.06,177.74) and (387.98,171.66) .. (404.54,156.97) ;
                    \draw    (476.93,135.53) .. controls (481.4,135.36) and (484.12,131.45) .. (487.19,131.73) .. controls (490.25,132) and (490.79,147.98) .. (494.5,148) .. controls (498.21,148.02) and (495.96,121.05) .. (498.75,119.25) .. controls (501.54,117.45) and (502.25,114.75) .. (512.01,114.25) .. controls (521.77,113.76) and (526.29,134.52) .. (544.21,135.87) ;
                    \draw    (476.93,135.53) .. controls (502.9,135.53) and (521.73,175.72) .. (560,177.07) ;
                    \draw    (513.53,152.42) .. controls (519.3,147.69) and (532.97,134.86) .. (544.21,135.87) ;
                    \draw    (462.5,177.75) .. controls (475.56,176.74) and (493.48,170.66) .. (510.04,155.97) ;
                    \draw  [fill={rgb, 255:red, 0; green, 0; blue, 0 }  ,fill opacity=1 ] (385.63,140.75) .. controls (385.63,139.71) and (386.46,138.88) .. (387.5,138.88) .. controls (388.54,138.88) and (389.38,139.71) .. (389.38,140.75) .. controls (389.38,141.79) and (388.54,142.63) .. (387.5,142.63) .. controls (386.46,142.63) and (385.63,141.79) .. (385.63,140.75) -- cycle ;
                    
                    \draw (383.92,81.32) node [anchor=north west][inner sep=0.75pt]    {$\Sigma _{0}^{+}$};
                    \draw (487.75,82.4) node [anchor=north west][inner sep=0.75pt]    {$\Sigma _{\varepsilon }^{+}$};
                    \draw (170,78.4) node [anchor=north west][inner sep=0.75pt]    {$\Sigma _{-2\varepsilon }^{+}$};
                    \draw (280.75,78.32) node [anchor=north west][inner sep=0.75pt]    {$\Sigma _{-\varepsilon }^{+}$};
                    \draw (384.92,183.57) node [anchor=north west][inner sep=0.75pt]    {$\Sigma _{0}^{-}$};
                    \draw (491.17,183.4) node [anchor=north west][inner sep=0.75pt]    {$\Sigma _{+\varepsilon }^{-}$};
                    \draw (175,183.65) node [anchor=north west][inner sep=0.75pt]    {$\Sigma _{-2\varepsilon }^{-}$};
                    \draw (281.08,183.57) node [anchor=north west][inner sep=0.75pt]    {$\Sigma _{-\varepsilon }^{-}$};

                    \end{tikzpicture}

                \caption{Slices of the~clasp moves introducing the~additional double point of correct sign}
                \label{fig:smthclasp}
            \end{figure}

            To finish the~proof, we have to prove that the~moves $\Sigma^\pm$ produce the~correct sign of the~contribution of the~point $p_\pm$ they introduce. We can view the~trace of $\Sigma^\pm$ as an exact immersed Lagrangian in $\R^{2n+2}$ with a~single intersection point. It was observed in \cite[Proposition 3.2]{EESnoniso} that the~Thurston-Bennequin invariant of the~Legendrian lift of $\Sigma^\pm$ to $\R^{2n+1}$ is equal to $I_W(\Sigma^\pm)$ and that the~contribution of the~intersection point is equal to $$I_H(p_\pm)=(-1)^{\frac{n(n-1)}{2}}(-1)^{|s_\pm|},$$ where $s_\pm$ is the~Reeb chord coming from the~intersection point $p_\pm$. Now, the~same local computation extends to higher dimension, see \cite[Proposition 8.3]{PR2019} ,stating that the~grading $$|s_\pm|=\mathfrak{m}(\text{upper sheet before clasping})-\mathfrak{m}(\text{lower sheet before clasping})$$ is equal to the~difference between the~Maslov potentials of the~sheets that we clasped. To conclude, we point out that $|s_+|=0-0=0$ and $|s_-|=1-0=1$. Note that
            $$\frac{(n-1)n}{2}\equiv 0 \pmod 2\quad\Longleftrightarrow\quad n\equiv 0 \text{ or } 1 \pmod 4.$$ So, for $n\equiv 1 \pmod 4 $, the~auxiliary point $p_\pm$ has the~correct sign. In the~other case, when $n\equiv 3 \pmod 4$, the~$\pm$-notation above has to be flipped.
        \end{proof}
\section{Applications}\label{sec:Applications}
\subsection{Legendrian Unknotting Number and Loose Legendrians}\label{sec:Legunknotting}
    As we have shown in Lemma~\ref{lem:nosmoothknotting}, there is no smooth knotting of embeddings of $S^n$ in $\R^{2n+1}$ for $n>1$, and so there is no nontrivial smooth unknotting number in this setting. However, we can restrict our attention to Legendrian embeddings as the~clasp move provides us with a~Legendrian model for the~crossing change. We call the~Legendrian lift of the~Whitney immersion $S^n\to \R^{2n}$ the~Legendria unknot $\L_{un}$.

    \begin{definition}
        Let $\Lambda$ be a~Legendrian $n$-sphere in $\R^{2n+1}$. Define the~(Legendrian) unknotting number $U(\Lambda)$ of $\Lambda$ as the~minimal number of clasp moves or unclasp moves up to Legendrian isotopy one needs to perform on $\Lambda$ to obtain the~Legendria unknot in $\R^{2n+1}$. If it is impossible to modify $\Lambda$ by (un)clasp moves to the~unknot, we define $U(\Lambda)=\infty$.
    \end{definition}
     We will see that the~unknotting number is bounded both from above and below by the~self-intersection information of its Lagrangian projection for all the~Legendrian spheres with finite unknotting number. Moreover, the~clasp move can produce loose Legendrians and, using the~h-principle results on the~loose Legendrian embeddings, one can show that for a~very negative Thurston-Bennequin invariant, the~value of the~unknotting number provides an obstruction to a~Legendrian sphere being loose.

    In \cite[Proposition A.4]{murphyloose}, we learn that there is only one formal Legendrian isotopy class of $\Lambda$, for each pair of values of $tb$ and $rot$, where $\Lambda$ is a~Legendrian submanifold in $\R^{2n+1}$ for $n$ odd and $\Lambda$ is a~stably parallelizable manifold. In particular, all spheres are stably parallelizable, and so, we denote those formal Legendrian isotopy classes by $\ell(tb,rot)$.
     
    \begin{proposition}\label{prop:Legunknotting}
        For $n\equiv 1 \pmod 4$ and for any Legendrian sphere $\Lambda$ so that $rot(\Lambda)=rot(\Lambda_{un})=0$, we have that
        \begin{equation*}
            \begin{cases}
                U(\Lambda)=\infty, \text{ if } tb(\Lambda) \text{ is even,}\\
                m\leq U(\Lambda)\leq M, \text{ if } tb(\Lambda)=2k-1,\\
            \end{cases}
        \end{equation*}
        where
        \begin{equation*}
            m=\begin{cases}
                |k| & \text{if } \Lambda \text{ is not loose,}\\
                |k+1| & \text{otherwise.}
            \end{cases} \,\,\,\, M=\begin{cases}
                |k|+2 & \text{if } \Lambda \text{ is not loose,}\\
                |k+1|+1 & \text{otherwise.}
            \end{cases}
        \end{equation*}
        In particular, if $U(\Lambda)>|k+1|+1$ and $tb(\Lambda)=2k-1$, then $\Lambda$ cannot be loose. Finally, no Legendrian of $tb=-1$ has $U=1$.
    \end{proposition}
    The~proposition implies that for every odd Thurston-Bennequin invariant the~unknotting number of a~non-loose Legendrian sphere has only three possible values, except for $tb=-1$, where there are only two possible values. For an illustration of the~bounds, see Figure~\ref{fig:graphsforMs}.
    \begin{figure}
        \centering

\tikzset{every picture/.style={line width=0.75pt}} 

\begin{tikzpicture}[x=0.75pt,y=0.75pt,yscale=-1,xscale=1]

\draw    (216,184) -- (483,183.13) ;
\draw    (350,28.63) -- (350,234) ;
\draw [color={rgb, 255:red, 74; green, 144; blue, 226 }  ,draw opacity=1 ]   (249.5,83.56) -- (349.5,183.56) ;
\draw [color={rgb, 255:red, 74; green, 144; blue, 226 }  ,draw opacity=1 ]   (453,85.13) -- (349.5,183.56) ;
\draw [color={rgb, 255:red, 126; green, 211; blue, 33 }  ,draw opacity=1 ]   (251.5,45.63) -- (350,139.5) ;
\draw [color={rgb, 255:red, 245; green, 166; blue, 35 }  ,draw opacity=1 ]   (350,139.5) -- (440.5,50.63) ;
\draw [color={rgb, 255:red, 245; green, 166; blue, 35 }  ,draw opacity=1 ]   (327.5,162.13) -- (350,139.5) ;
\draw [color={rgb, 255:red, 245; green, 166; blue, 35 }  ,draw opacity=1 ]   (249.5,83.56) -- (327.5,162.13) ;
\draw  [dash pattern={on 4.5pt off 4.5pt}]  (327.5,29.13) -- (327.5,234.5) ;
\draw [color={rgb, 255:red, 184; green, 233; blue, 134 }  ,draw opacity=1 ]   (440.5,50.63) -- (350,139.5) ;
\draw [color={rgb, 255:red, 208; green, 2; blue, 27 }  ,draw opacity=1 ]   (227,84.06) -- (327.5,184.13) ;
\draw [color={rgb, 255:red, 208; green, 2; blue, 27 }  ,draw opacity=1 ]   (445.5,70.63) -- (327.5,184.13) ;

\draw (466,189.4) node [anchor=north west][inner sep=0.75pt]    {$k$};
\draw (357,17.4) node [anchor=north west][inner sep=0.75pt]    {$U( \Lambda )$};
\draw (413.5,130.9) node [anchor=north west][inner sep=0.75pt]  [color={rgb, 255:red, 74; green, 144; blue, 226 }  ,opacity=1 ]  {$|k|$};
\draw (184.5,58.4) node [anchor=north west][inner sep=0.75pt]  [color={rgb, 255:red, 245; green, 166; blue, 35 }  ,opacity=1 ]  {$|k+1|+1$};
\draw (351.5,186.96) node [anchor=north west][inner sep=0.75pt]    {$0$};
\draw (304,187.9) node [anchor=north west][inner sep=0.75pt]    {$-1$};
\draw (365,51.4) node [anchor=north west][inner sep=0.75pt]  [color={rgb, 255:red, 184; green, 233; blue, 134 }  ,opacity=1 ]  {$|k|+2$};
\draw (226,145.9) node [anchor=north west][inner sep=0.75pt]  [color={rgb, 255:red, 208; green, 2; blue, 27 }  ,opacity=1 ]  {$|k+1|$};

\end{tikzpicture}

        \caption{Plot of the~possible bounds $m$ and $M$ from Proposition~\ref{prop:Legunknotting}.}
        \label{fig:graphsforMs}
    \end{figure}
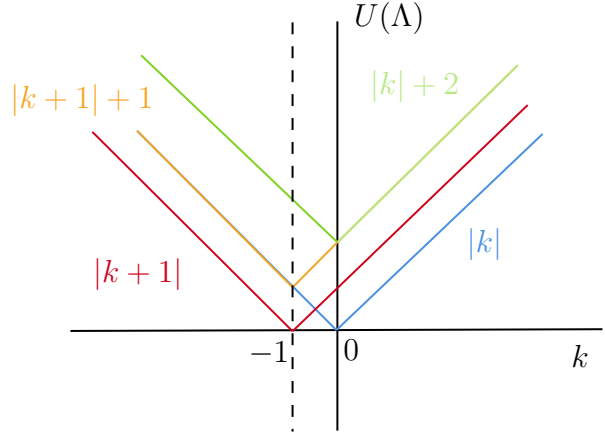
    \begin{proof}
        Take a~Legendrian embedding $\Lambda$ of a~sphere, and choose an arbitrary point in the~front projection which is not on a~cusp-edge. Now, perform locally a~higher-dimensional equivalent of the~Reidemeister I move, and then clasp the~upper sheet with the~lower one as in the~move $\Sigma^-$ in the~proof of Lemma~\ref{lem:nosmoothknotting}. Call the~composition of the~previous two moves a~loose clasp (and its inverse a~loose unclasp). To see that the~loose clasp move produces a~loose chart, see \cite[Example 10.14]{sen2023hprinciplelooselegendrianembeddings}.  
        
        In Lemma~\ref{lem:CZ lemma} we learn that the~loose clasp move around $s_{-\mu}$ of any grading preserves the~parity of $tb$ and so, Legendrian spheres of even $tb$ cannot be clasped to the~Legendria unknot, so for those, we have defined $U(\Lambda)=\infty$.
        
        \par Assume that $tb(\Lambda)=2k-1$. Now, the~loose clasp around Reeb chord of even grading decreases the~$tb$ by $2$ since $|s_{-\mu}|=0$ and $\frac{(n-2)(n-1)}{2}\equiv 0 \pmod 2$. Thus, the~minimal lower bound on the~necessary number of clasp moves must be $|k|$. {\color{blue} } Therefore, performing a~loose clasp on the~unknot $\L_{un}$, we obtain a~loose Legendrian sphere  isotopic to $\ell(-3,0)$. Now, any Legendrian sphere of rotation number $0$ falls into one of these two categories:
        \begin{itemize}
            \item loose, in which case isotopic to $\ell(2k-1,0)$ and so, isotopic to $\ell(-3,0)$ after $|k+1|$ loose clasps,
            \item not loose, in which case, we perform $|k|+1$ loose clasps to decrease the~$tb$ enough to be able to isotope to $\ell(-3,0)$.
        \end{itemize}
        To get from $\ell(-3,0)$ to the~unknot we have to unclasp, this explains the~$+1$ contribution to the~bound $M$.
    \end{proof}
    \begin{remark}{\color{white}Surprise!}
    \par
    \begin{itemize}
        \item The~assumption $n\equiv 1 \pmod 4$ ensures that the~loose clasp decreases the~$tb$-invariant. For $n\equiv 3 \pmod 4$, one obtains an analogous statement and proof, where the~loose clasp will increase $tb$ invariant, and we will work with $\ell(1,0)$ instead of $\ell(-3,0)$.
        \item So far, $n$ was assumed to be odd. This assumption is significant since there are, in general, two classes of loose Legendrians for every pair of Thurston-Bennequin invariant and the~rotation class if $n$ is even. As commented in \cite[Appendix]{murphyloose}, we do not know how to distinguish these two classes.
    \end{itemize}
    \end{remark}
    The~unknotting number is very difficult to compute; nevertheless, one could define a~unobstructed unknotting number, where one forbids the~regular homotopy through Legendrian immersions to pass through embedded loose Legendrians. 
    \begin{question}
        Can counts of $J$-holomorphic curves with boundary on a~connected immersed Lagrangian cobordism, cylindrical over embedded Legendrians, provide a~lower bound for the~unobstructed unknotting number?
    \end{question}
    \subsection{Twist Spheres}\label{sec:twist spheres}
    We will define a~class of Legendrian spheres $\Lambda^u_k$ in $\R^{2n+1}$ isotopic to the~Legendrian unknot in $\R^{2n+1}$ for $n> 1$. Then, we will perform a~clasp move along a~specific chord of $\Lambda^u_k$ to produce a~Legendrian sphere $\Lambda_k$. This will generalise the~construction of twist knots in $\R^3$ (see Figure~\ref{fig:Meyertwist}).

    First, we take the~Legendrian strand ending with a~Reeb chord so that the~front projection is as in Figure~\ref{fig:strandchords} in $\R^2$ with coordinates $x_1,z$. That is, we start with an opened-up Legendria unknot of $tb=-1$ and $rot=0$, then we perform a~Reidemeister I move. After that, we modify the~sheets of Maslov potential $1$ to create two Reeb chords between them, $s^0_{11}$ and $s^1_{11}$. We create two other pairs of Reeb chords labelled with the~letter $d$ along this isotopy by arranging the~slopes, whose derivatives are shown in Figure~\ref{fig:graph of differentials}.
    \begin{figure}[htbp!]
        \centering
        \tikzset{every picture/.style={line width=0.75pt}} 
        \setlength{\unitlength}{0.1\textwidth}
		\begin{picture}(10,2)
			\put(-0.3,0){\includegraphics[scale=1]{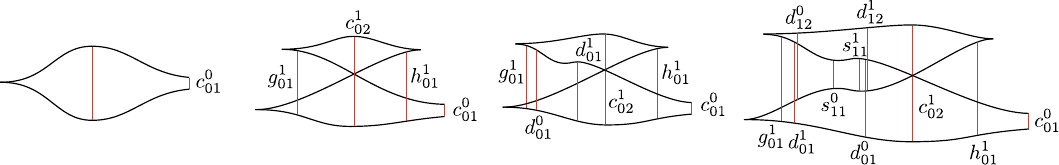}}
		\end{picture}
        \caption{Slice of the~front of the~Legendrian $0$-clef sphere $\Lambda^u_0$.}
        \label{fig:strandchords}
    \end{figure}

    \begin{figure}[htbp!]
        \centering
        \tikzset{every picture/.style={line width=0.75pt}} 
        \setlength{\unitlength}{0.1\textwidth}
		\begin{picture}(10,4)
			\put(0.8,0){\includegraphics[scale=1]{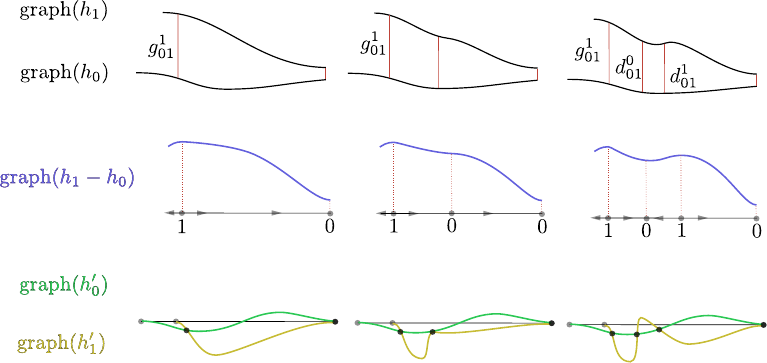}}
		\end{picture}
        \caption{Graph of differentials of the~local height functions}
        \label{fig:graph of differentials}
    \end{figure}
    
    We then perform $k$ very small Reidemeister $I$ moves close to the~lower end of $s^0_{11}$, see Figure~\ref{fig:strandchords}. Having this one-dimensional front, we rotate the~front around the~axis given by the~line which contains the~Reeb chord bounding the~Legendrian strand from the~right in $x_1,x_2,\dots,x_{n},z$ (for details see \cite[Section 2]{G2014}). Since the~tangent vector to the~Legendrian strand is flat at the~endpoints, the~tangent plane to the~rotated front is perpendicular to the~$z$-axis, keeping the~centre Reeb chord and proving that the~higher-dimensional spin creates an embedded Legendrian $n$-sphere. This would not be true if the~slope at the~endpoints of the~strand were non-zero.
    
    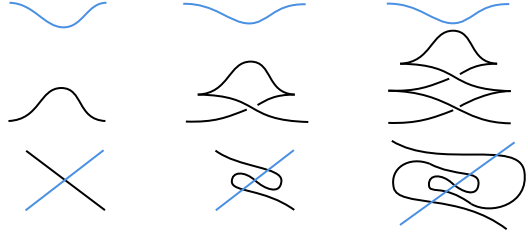
\begin{figure}[htbp!]
        \centering
        
        \tikzset{every picture/.style={line width=0.75pt}} 
        
        \begin{tikzpicture}[x=0.75pt,y=0.75pt,yscale=-0.5,xscale=0.5]
        
        \draw    (263.05,101.81) .. controls (327.49,102.71) and (307.72,127.51) .. (374.03,129.3) ;
        \draw    (323.03,112.3) .. controls (332.68,106.92) and (339.83,101.81) .. (360.37,101.81) ;
        \draw    (263.05,101.81) .. controls (294.98,100.46) and (293.82,68.68) .. (316.13,68.68) .. controls (338.45,68.68) and (331.52,100.46) .. (360.37,101.81) ;
        \draw    (251.12,128.86) .. controls (280.38,129.75) and (291.73,124.87) .. (312.28,116.8) ;
        \draw [color={rgb, 255:red, 74; green, 144; blue, 226 }  ,draw opacity=1 ]   (248.46,10.69) .. controls (280.71,9.92) and (295.14,30.28) .. (314.61,30.38) .. controls (334.08,30.49) and (338.72,10.43) .. (371.36,11.13) ;
        \draw    (280.82,157.96) .. controls (304.03,175.3) and (348.03,174.3) .. (347.03,188.55) .. controls (346.03,202.8) and (330.46,196.96) .. (320.37,187.76) .. controls (310.28,178.55) and (297.03,178.55) .. (297.28,189.3) .. controls (297.53,200.05) and (330.03,196.05) .. (359.91,217.55) ;
        \draw    (73.05,128.31) .. controls (104.98,126.96) and (103.82,95.18) .. (126.13,95.18) .. controls (148.45,95.18) and (141.52,126.96) .. (170.37,128.31) ;
        \draw [color={rgb, 255:red, 74; green, 144; blue, 226 }  ,draw opacity=1 ]   (74.05,9.65) .. controls (98.8,10.11) and (108.81,34.16) .. (128.28,34.27) .. controls (147.74,34.38) and (152.3,9.36) .. (171.37,9.65) ;
        \draw    (90.82,158.21) -- (169.91,217.8) ;
        \draw    (466.21,70.81) .. controls (530.65,71.71) and (510.88,96.51) .. (577.19,98.3) ;
        \draw    (526.19,81.3) .. controls (535.84,75.92) and (542.99,70.81) .. (563.54,70.81) ;
        \draw    (466.21,70.81) .. controls (498.14,69.46) and (496.99,37.68) .. (519.3,37.68) .. controls (541.61,37.68) and (534.69,69.46) .. (563.54,70.81) ;
        \draw    (454.29,97.86) .. controls (483.55,98.75) and (494.9,93.87) .. (515.44,85.8) ;
        \draw [color={rgb, 255:red, 74; green, 144; blue, 226 }  ,draw opacity=1 ]   (452.79,10.52) .. controls (485.05,9.75) and (499.47,30.11) .. (518.94,30.22) .. controls (538.41,30.32) and (543.05,10.26) .. (575.69,10.97) ;
        \draw    (454.29,97.86) .. controls (518.73,98.75) and (525.47,130.14) .. (577.19,130.63) ;
        \draw    (526.19,113.63) .. controls (535.84,108.26) and (556.65,98.3) .. (577.19,98.3) ;
        \draw    (454.29,130.19) .. controls (483.55,131.09) and (498.93,125.54) .. (519.47,117.48) ;
        \draw    (495.94,171.25) .. controls (522.28,184.91) and (545.86,177.13) .. (544.86,191.38) .. controls (543.86,205.63) and (528.29,199.79) .. (518.2,190.59) .. controls (508.11,181.38) and (494.86,181.38) .. (495.11,192.13) .. controls (495.36,202.88) and (507.72,190.08) .. (537.61,211.58) ;
        \draw    (495.94,171.25) .. controls (475.14,163.21) and (455.09,175.44) .. (459.84,195.44) .. controls (464.59,215.44) and (524.84,200.19) .. (573.14,236.96) ;
        \draw    (537.61,211.58) .. controls (557.34,223.19) and (596.09,211.94) .. (590.84,178.94) .. controls (585.59,145.94) and (504.59,176.94) .. (457.84,148.19) ;
        \draw [color={rgb, 255:red, 74; green, 144; blue, 226 }  ,draw opacity=1 ]   (90.16,218.05) -- (168.55,157.64) ;
        \draw [color={rgb, 255:red, 74; green, 144; blue, 226 }  ,draw opacity=1 ]   (281.17,217.96) -- (359.56,157.55) ;
        \draw [color={rgb, 255:red, 74; green, 144; blue, 226 }  ,draw opacity=1 ]   (466.05,232.64) -- (581.06,149.72) ;

        \end{tikzpicture}

        \caption{Reeb chord creation close to the~Reidemeister I (RMI) moves. The~upper part shows the~RMI in the~front projection, and the~lower part shows RMI in the~Lagrangian projection. We can observe the~creation of the~intersection points in the~Lagrangian projection.}
        \label{fig:cleftower}
    \end{figure}
    
    We obtain an $S^{n-1}$ Morse-Bott family of Reeb chords for each chord in the~front projection except for the~Reeb chord that coincides with the~axis of rotation. We choose a~Morse function on $S^{n-1}$ with precisely two critical points and perturb the~Morse-Bott families. Say that the~lowest sheet has the~Maslov potential equal to $0$. Then, we find the~shortest chord $s^0$, which is the~chord starting on the~sheet with Maslov potential $k+1$ at the~top of the~$k$-fold iterated Reidemeister I move and ending on the~sheet of Maslov potential $1$. This iterated Reidemeister I move makes the~Lagrangian projection look like a~gluing of two reflected G-clefs, see the~lower right part of Figure~\ref{fig:cleftower}. Note that the~chord $s^0$ corresponds to the~self-intersection in the~centre of Figure~\ref{fig:cleftower}. We call this particular presentation of the~Legendria unknot the~$k$-clef sphere $\L^u_k$.

    \begin{figure}[htbp!]
        \centering
        \tikzset{every picture/.style={line width=0.75pt}} 
        \setlength{\unitlength}{0.1\textwidth}
		\begin{picture}(10,3.0)
			\put(0.5,0){\includegraphics[scale=1.1]{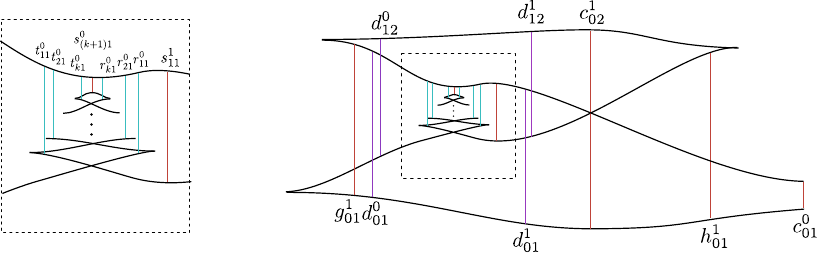}}
		\end{picture}
        \caption{Slice of the~front of the~Legendrian $k$-clef sphere $\Lambda^u_k$.}
        \label{fig:halfclefsphere}
    \end{figure}
    
    \begin{definition}
        The~Legendrian $n$-sphere $\Lambda_k$ in $\R^{2n+1}$ that is obtained by the~clasp move of $\Lambda^u_k$ centred at $s^0$ is called the~$k$-twist sphere. 
    \end{definition} For an example of the~front projection of $\Lambda_k\subset \R^5$ in $\R^3$ see Figure~\ref{fig:twist2sphere}.
    
    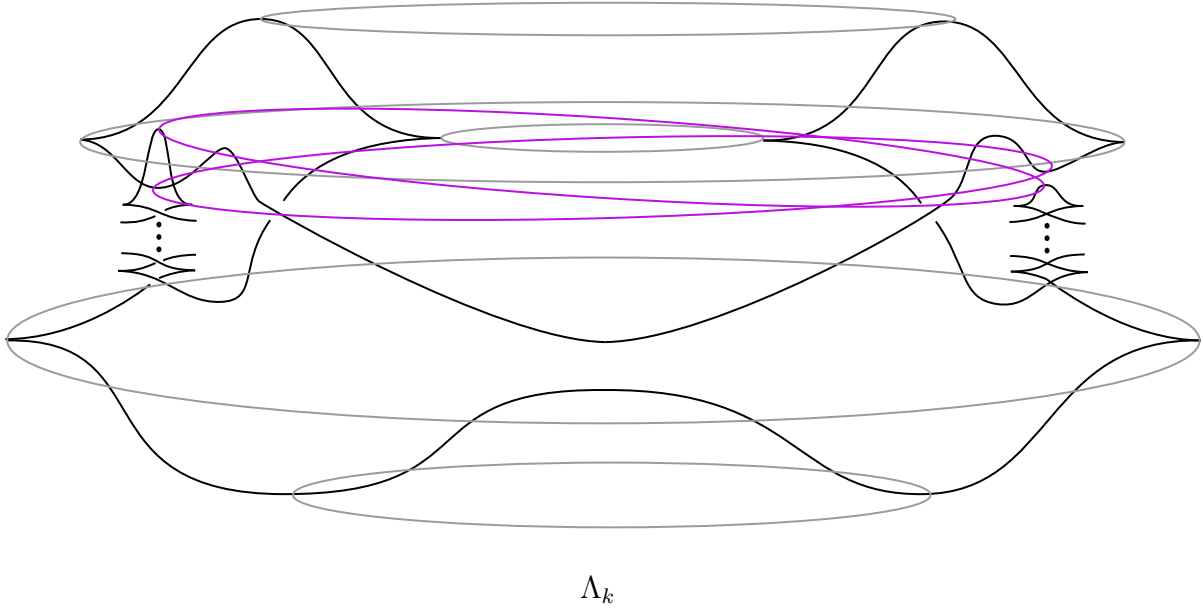
\begin{figure}[htbp!]
        \centering 
        
\tikzset{every picture/.style={line width=0.75pt}} 

\begin{tikzpicture}[x=0.75pt,y=0.75pt,yscale=-1,xscale=1]

\draw    (35.83,182.44) .. controls (115.4,181.19) and (71.08,259.57) .. (174.83,259.57) .. controls (278.58,259.57) and (235.94,206.49) .. (335.32,207.32) ;
\draw    (72.42,81.66) .. controls (126.14,80.83) and (124.13,21.12) .. (163.08,21.12) .. controls (202.03,21.12) and (194.64,81.66) .. (253.73,80.83) ;
\draw    (72.42,81.66) .. controls (90.8,82.84) and (95.34,105.24) .. (110.95,105.86) .. controls (126.56,106.48) and (139.15,85.33) .. (145.2,85.95) .. controls (151.24,86.58) and (158.29,109.59) .. (162.32,112.7) .. controls (166.35,115.81) and (280.87,182.34) .. (335.99,183.27) ;
\draw    (93.79,114.3) .. controls (106.95,115.09) and (105.37,75.15) .. (112.04,76.48) .. controls (118.72,77.81) and (114.47,115.09) .. (128.7,114.3) ;
\draw    (174.49,112.35) .. controls (187.92,92.45) and (214.11,81.66) .. (253.73,80.83) ;
\draw    (91.47,147.53) .. controls (115.9,148.05) and (122.79,163.57) .. (142.68,163.09) .. controls (162.57,162.61) and (151.66,144.49) .. (167.78,122.1) ;
\draw    (34.82,181.81) .. controls (70.35,181.55) and (98.72,161.09) .. (107.58,154.31) ;
\draw    (112.04,151.18) .. controls (115.36,149.61) and (123.07,147.27) .. (130.13,147.27) ;
\draw    (93.08,138.66) .. controls (115.23,138.92) and (107.34,146.74) .. (130.13,147.27) ;
\draw    (114.05,142.31) .. controls (117.37,140.74) and (123.47,139.44) .. (130.53,139.44) ;
\draw    (91.47,147.53) .. controls (100.86,147.53) and (103.28,146.48) .. (110.26,143.61) ;
\draw    (93.79,114.3) .. controls (112.32,114.56) and (111.78,121.87) .. (130.85,122.39) ;
\draw    (114.91,117.17) .. controls (117.69,115.61) and (122.79,114.3) .. (128.7,114.3) ;
\draw    (92.63,123.17) .. controls (101.04,123.44) and (104.26,122.39) .. (110.17,120.04) ;
\draw  [fill={rgb, 255:red, 0; green, 0; blue, 0 }  ,fill opacity=1 ] (111.2,136.72) .. controls (111.2,136.2) and (111.54,135.78) .. (111.96,135.78) .. controls (112.38,135.78) and (112.71,136.2) .. (112.71,136.72) .. controls (112.71,137.23) and (112.38,137.65) .. (111.96,137.65) .. controls (111.54,137.65) and (111.2,137.23) .. (111.2,136.72) -- cycle ;
\draw  [fill={rgb, 255:red, 0; green, 0; blue, 0 }  ,fill opacity=1 ] (111.2,130.49) .. controls (111.2,129.98) and (111.54,129.56) .. (111.96,129.56) .. controls (112.38,129.56) and (112.71,129.98) .. (112.71,130.49) .. controls (112.71,131.01) and (112.38,131.43) .. (111.96,131.43) .. controls (111.54,131.43) and (111.2,131.01) .. (111.2,130.49) -- cycle ;
\draw  [fill={rgb, 255:red, 0; green, 0; blue, 0 }  ,fill opacity=1 ] (111.2,124.27) .. controls (111.2,123.76) and (111.54,123.34) .. (111.96,123.34) .. controls (112.38,123.34) and (112.71,123.76) .. (112.71,124.27) .. controls (112.71,124.79) and (112.38,125.21) .. (111.96,125.21) .. controls (111.54,125.21) and (111.2,124.79) .. (111.2,124.27) -- cycle ;
\draw    (634.82,182.44) .. controls (555.24,181.19) and (561.86,258.57) .. (495.81,259.57) .. controls (429.77,260.57) and (434.71,206.49) .. (335.32,207.32) ;
\draw    (596.54,82.91) .. controls (542.82,82.08) and (544.83,22.36) .. (505.89,22.36) .. controls (466.94,22.36) and (474.33,82.91) .. (415.23,82.08) ;
\draw    (596.54,82.91) .. controls (583.14,83.05) and (566.56,97.57) .. (556.53,97.77) .. controls (546.51,97.98) and (547.97,80.36) .. (532.36,79.73) .. controls (516.74,79.11) and (519.92,99.96) .. (510.2,110.21) .. controls (500.47,120.46) and (391.61,182.33) .. (335.99,183.27) ;
\draw    (578.17,148.15) .. controls (553.74,148.67) and (551.55,166.68) .. (532.75,164.19) .. controls (513.95,161.7) and (517.97,145.11) .. (501.86,122.72) ;
\draw    (634.82,182.44) .. controls (599.29,182.17) and (566.46,158.59) .. (557.6,151.8) ;
\draw    (557.6,151.8) .. controls (554.28,150.24) and (546.57,147.89) .. (539.51,147.89) ;
\draw    (576.56,139.28) .. controls (554.41,139.54) and (562.3,147.37) .. (539.51,147.89) ;
\draw    (555.59,142.93) .. controls (552.27,141.36) and (546.17,140.06) .. (539.11,140.06) ;
\draw    (578.17,148.15) .. controls (568.77,148.15) and (562.57,145.8) .. (555.59,142.93) ;
\draw    (575.84,114.93) .. controls (557.32,115.19) and (557.85,122.49) .. (538.79,123.01) ;
\draw    (554.72,117.8) .. controls (551.95,116.23) and (546.85,114.93) .. (540.94,114.93) ;
\draw    (575.84,114.93) .. controls (562.69,115.71) and (564.57,103.97) .. (556.78,104.49) .. controls (548.99,105.01) and (555.17,115.71) .. (540.94,114.93) ;
\draw    (577.01,123.8) .. controls (568.59,124.06) and (560.63,120.14) .. (554.72,117.8) ;
\draw    (494.47,113.6) .. controls (481.04,93.69) and (454.85,82.91) .. (415.23,82.08) ;
\draw  [fill={rgb, 255:red, 0; green, 0; blue, 0 }  ,fill opacity=1 ] (558.43,137.34) .. controls (558.43,136.82) and (558.09,136.4) .. (557.68,136.4) .. controls (557.26,136.4) and (556.92,136.82) .. (556.92,137.34) .. controls (556.92,137.85) and (557.26,138.27) .. (557.68,138.27) .. controls (558.09,138.27) and (558.43,137.85) .. (558.43,137.34) -- cycle ;
\draw  [fill={rgb, 255:red, 0; green, 0; blue, 0 }  ,fill opacity=1 ] (558.43,131.12) .. controls (558.43,130.6) and (558.09,130.18) .. (557.68,130.18) .. controls (557.26,130.18) and (556.92,130.6) .. (556.92,131.12) .. controls (556.92,131.63) and (557.26,132.05) .. (557.68,132.05) .. controls (558.09,132.05) and (558.43,131.63) .. (558.43,131.12) -- cycle ;
\draw  [fill={rgb, 255:red, 0; green, 0; blue, 0 }  ,fill opacity=1 ] (558.43,124.9) .. controls (558.43,124.38) and (558.09,123.96) .. (557.68,123.96) .. controls (557.26,123.96) and (556.92,124.38) .. (556.92,124.9) .. controls (556.92,125.41) and (557.26,125.83) .. (557.68,125.83) .. controls (558.09,125.83) and (558.43,125.41) .. (558.43,124.9) -- cycle ;
\draw  [color={rgb, 255:red, 155; green, 155; blue, 155 }  ,draw opacity=1 ] (35.83,182.44) .. controls (35.83,159.46) and (169.78,140.83) .. (335.03,140.83) .. controls (500.27,140.83) and (634.23,159.46) .. (634.23,182.44) .. controls (634.23,205.41) and (500.27,224.04) .. (335.03,224.04) .. controls (169.78,224.04) and (35.83,205.41) .. (35.83,182.44) -- cycle ;
\draw  [color={rgb, 255:red, 155; green, 155; blue, 155 }  ,draw opacity=1 ] (72.42,82.91) .. controls (72.42,71.8) and (189.75,62.8) .. (334.48,62.8) .. controls (479.21,62.8) and (596.54,71.8) .. (596.54,82.91) .. controls (596.54,94.01) and (479.21,103.02) .. (334.48,103.02) .. controls (189.75,103.02) and (72.42,94.01) .. (72.42,82.91) -- cycle ;
\draw  [color={rgb, 255:red, 155; green, 155; blue, 155 }  ,draw opacity=1 ] (163.08,21.12) .. controls (163.08,16.61) and (241.12,12.95) .. (337.39,12.95) .. controls (433.66,12.95) and (511.7,16.61) .. (511.7,21.12) .. controls (511.7,25.63) and (433.66,29.29) .. (337.39,29.29) .. controls (241.12,29.29) and (163.08,25.63) .. (163.08,21.12) -- cycle ;
\draw  [color={rgb, 255:red, 155; green, 155; blue, 155 }  ,draw opacity=1 ] (179.19,259.99) .. controls (179.19,251.01) and (250.84,243.73) .. (339.22,243.73) .. controls (427.61,243.73) and (499.25,251.01) .. (499.25,259.99) .. controls (499.25,268.97) and (427.61,276.24) .. (339.22,276.24) .. controls (250.84,276.24) and (179.19,268.97) .. (179.19,259.99) -- cycle ;
\draw  [color={rgb, 255:red, 155; green, 155; blue, 155 }  ,draw opacity=1 ] (253.73,80.83) .. controls (253.73,77.01) and (289.89,73.91) .. (334.48,73.91) .. controls (379.08,73.91) and (415.23,77.01) .. (415.23,80.83) .. controls (415.23,84.66) and (379.08,87.76) .. (334.48,87.76) .. controls (289.89,87.76) and (253.73,84.66) .. (253.73,80.83) -- cycle ;
\draw  [color={rgb, 255:red, 189; green, 16; blue, 224 }  ,draw opacity=1 ] (112.04,76.51) .. controls (112.51,65.4) and (212.31,62.76) .. (334.96,70.6) .. controls (457.6,78.45) and (556.65,93.81) .. (556.18,104.92) .. controls (555.72,116.02) and (455.92,118.66) .. (333.27,110.82) .. controls (210.62,102.97) and (111.58,87.61) .. (112.04,76.51) -- cycle ;
\draw  [color={rgb, 255:red, 189; green, 16; blue, 224 }  ,draw opacity=1 ] (108.99,107.05) .. controls (108.79,95.93) and (209.62,84.19) .. (334.19,80.81) .. controls (458.77,77.44) and (559.91,83.72) .. (560.11,94.84) .. controls (560.3,105.95) and (459.48,117.7) .. (334.9,121.07) .. controls (210.33,124.44) and (109.18,118.16) .. (108.99,107.05) -- cycle ;

\draw (321.9,298.93) node [anchor=north west][inner sep=0.75pt]    {$\Lambda _{k}$};

\end{tikzpicture}

        \caption{Front of twist spheres $\Lambda_k\subset\R^5$. Most of the~sheets are rotated about the~vertical axis. Two purple sheets exchange the~vertical position along the~rotation.}
        \label{fig:twist2sphere}
    \end{figure}
    \begin{lemma}\label{lem:clefstrand}
        The~Reeb chords of the~slice of Legendrian $k$-clef sphere are shown in Figure~\ref{fig:halfclefsphere}. Moreover, their grading is as follows:
        \vspace{-6 pt}\begin{align*}
            |s^0_{(k+1)1}|&=-(k+1), \,\,|s_{11}^1|=0,\,\,|t^0_{m1}|=|r^0_{m1}|=-m, \\
            |c^1_{02}|&=2,\,\, |c^0_{01}|=0,\,\, |g^1_{01}|=1,\,\, |h^1_{01}|=1,\\
            |d^0_{01}|&=0,\,\,|d^1_{01}|=1,\,\,
            |d^0_{12}|=0,\,\,|d^1_{12}|=1,
        \end{align*} 
        for $1\leq m\leq k$. Moreover, the~differential counting flow-trees for the~slice is given by the~following equations:
        \begin{align*}
            \partial c^0_{01}&=\partial s^0_{(k+1)1}=0, & \partial h^1_{01}&=1+c^0_{01}, \\
            \partial r^0_{m1}&=\partial t^0_{m1}=\begin{cases}
             r^0_{(m+1)1}+t^0_{(m+1)1}, \text{ if } m<k,\\
             s^0_{(k+1)1}, \text{ if } m=k.
         \end{cases} & \partial d^1_{01}&=c^0_{01}+s^1_{11}+d^0_{01},\\
            \partial s^1_{11}&=\partial d^0_{01}=\partial d^0_{12}=\begin{cases}
                r^0_{11}+t^0_{11}, \text{ if } k>0,\\
                s^0_{11}, \text{ if } k=0,
            \end{cases}& \partial d^1_{12}&=1+s^1_{11}+d^0_{12},\\
            \partial g^1_{01}&=d^0_{12}+d^0_{01},& \partial c^1_{02}&=g^1_{01}+h^1_{01}+d^1_{01}+d^1_{12}.
        \end{align*}
    \end{lemma}
     To compute the~contributions to the~differential of the~Chekanov-Eliashberg algebra, we will use Ekholm's flow trees as the~rigid flow trees are in bijective correspondence with rigid pseudo-holomorphic curves; for a~concise presentation of the~relevant notions, see \cite[Section 2.2]{DR2011knotted}, a~sample computation, see \cite{geog}, and for details, see \cite{Ekholmtrees}.
    
    Let us remark that the~differential above does not define an invariant of the~Legendrian strand. One can observe that the~chords coming from the~minima of the~Morse-Bott perturbation form a~sub-algebra. And so, in view of a~bijection of rigid flow-trees (see proof of Lemma~\ref{lem:dgaofclef}), it will be enough to enumerate these one-dimensional flow-trees of the~Legendrian strand giving rise to the~$k$-clef sphere under spinning.
    \begin{proof}
        First, we identify all the~Reeb chords of the~$k$-clef strand, see Figure~\ref{fig:halfclefsphere}. The~labels in the~notation $x^{a}_{cb}$ correspond to: $a$ is the~Morse index of the~difference function of the~local height functions, $c$ is the~Maslov potential of the~starting sheet of the~Reeb chord, $b$ is the~Maslov potential of the~sheet where the~Reeb chord ends. Now,
    \begin{itemize}
        \item the~chord $c_{02}^1$ is the~unknot Reeb chord that we start with (see the~leftmost part of Figure~\ref{fig:strandchords}),
        \item the~chords $g^1_{01}$ and $h^1_{01}$ are the~Reeb chords created by the~big RMI move (see the~part of Figure~\ref{fig:cleftower} that is second from the~left),
        \item the~chords $x_{cb}^a$ for $x\in\lbrace d,s\rbrace$ come from the~arrangement of the~slopes of the~front (see the~part of Figure~\ref{fig:strandchords} that is third from the~left),
        \item the~chords $x_{cb}^a$ for $x\in\lbrace t,r\rbrace$ come from the~small iterative RMI moves (see the~part of Figure~\ref{fig:cleftower}).
    \end{itemize}
    We determine the~grading of the~Reeb chords, via the~formula
    $$|x^{a}_{cb}|=a+b-c-1,$$ where $x$ is a~letter representing a~labelling of a~Reeb chord and $0\leq m\leq k$.

    We describe the~flow tree contribution in words below. Readers who prefer a~pictorial proof are referred to Appendix~\ref{sec:Flowtreecomputation} for detailed computation. 

     $$\partial s^0_{(k+1)1}=0$$  \par This is because there is no other chord of lower grading.
         $$\partial r^0_{m1}=\partial t^0_{m1}=\begin{cases}
             r^0_{(m+1)1}+t^0_{(m+1)1}, \text{ if } m<k,\\
             s^0_{(k+1)1}, \text{ if } m=k.
         \end{cases}$$\par We prove the~first identity, the~second one follows from a~symmetric argument. The~index of the~critical point representing the~Reeb chord $t^0_{m1}$ is $0$ and so the~only possibility is to split the~Reeb chord along a~sheet that it intersects using a~$P_0$ followed by a~zero-length edge to a~$Y_0$-vertex. Between the~sheets with Maslov potential $m$ and $m+1$, there is a~unique trajectory going to the~edge. Between the~sheet $1$ and $m+1$, there is a~unique trajectory going to the~minimum between those two sheets, which is the~chord $r^0_{(m+1)1}$ or $s^0_{(k+1)1}$ depending on the~Maslov potential.
         \par The~contribution $t^0_{(m+1)1}$ is given by a~$Y_1$-vertex at the~singular edge, the~lower part between sheets of Maslov potential $m+2$ and $m+1$, the~lower between the~sheets of Maslov potential $m+1$ and $1$.
         
         $$\partial s^1_{11}=t^0_{11}+r^0_{11}$$ \par The~Morse index of the~chord is $1$ and both emanating trajectories are rigid. The~trajectory going towards the~axis of rotation loses energy at the~intersection locus of sheets with Maslov potential $1$. This intersection is not an edge and so the~tree is not admissible. When following the~other trajectory, $Y_1$ splitting along the~singular edges between the~sheets with Maslov potential $1$ and $2$ of the~small Reidemeister move yields the~first contribution. Moreover, one has one obvious contribution via going to the~minimum between the~sheets of Maslov potential $1$, which is the~chord $t^0_{11}$.
          $$\partial c^0_{01}=0$$ \par The~chord is represented by a~minimum and there are no intermediate sheets to interact with.
          $$\partial d^0_{01}=\partial d^0_{12}=r^0_{11}+t^0_{11}$$ \par The~contribution comes from splitting the~minimal chord along the~sheet with Maslov potential $1$ with a~$Y_0$-vertex and a~ghost trajectory, and then flowing to the~minimum between sheets of Maslov potential $1$ along one edge and to the~singular edge locus of the~front along the~other edge. The~$t^0_{11}$ contribution comes from a~tree with $Y_1$-vertex along the~edge of the~sheet with Maslov potential $1$ and $2$ as in $\partial s^1_{11}$.
          $$\partial g^1_{01}=d^0_{12}+d^0_{01}$$
         \par The~Morse index of the~chord is $1$ and so taking the~radial trajectory towards the~axis of rotation, we arrive at the~minimum $d^0_{01}$ along a~rigid trajectory, which yields the~$d^0_{01}$ contribution. Now, taking the~opposite trajectory, we arrive at the~edge, where we perform a~$Y_0$-vertex splitting along the~sheet of Maslov potential $1$. The~lower part of the~tree follows to the~edge between the~sheet of Maslov potential $0$ and the~intermediary sheet of Maslov potential $1$. The~upper part of the~tree after splitting flows to the~minimum representing $d^0_{12}$.
          $$\partial h^1_{01}=1+c^0_{01}$$
         \par This is analogous to the~previous case.
          \begin{align*}
             \partial d^1_{01}&=c^0_{01}+d^0_{01}+s^1_{11},\\
             \partial d^1_{12}&=1+d^0_{12}+s^1_{11}.
         \end{align*}
         \par We will comment on the~first differential. The~first two contributions come from trees between the~sheets $S_0$ and $S_1$, ending with a~$P_0$-vertex. The~last contribution comes from a~$P_1$-vertex at $s^1_{11}$ followed by a~sub-tree that possibly has a~$Y_1$-vertex along over the~edge of the~small Reidemeister moves; it then climbs the~tower upwards in the~same fashion until it finishes with an $E$-vertex. Consult Appendix~\ref{sec:Flowtreecomputation} for illustration. The~second differential follows by an analogous argument.
         \begin{equation*}
             \partial c^1_{02}=g^1_{01}+h^1_{01}+d^1_{01}+d^1_{12}.
         \end{equation*}
         \par The~first three terms come from trees with $P_1$-vertices at the~chords followed by an $E$-vertex. The~last contribution differs by $Y_1$-vertices climbing the~sheets created by the~small Reidemeister I moves as in the~differential of $\partial s^1_{11}$.

         Note that we have considered all possibilities for rigid flow-trees for each chord and so there are no terms in the~differential of word length greater than one.
    \end{proof}

    Recall that the~$k$-clef sphere arises from the~rotation of the~$k$-clef strand. And so, for each Reeb chord, except $c^0_{01}$, we obtain an $S^{n-1}$-Morse-Bott family of Reeb chords. Choose a~Morse-Smale function on $S^{n-1}$ with precisely two critical points. We perturb out the~Morse-Bott loci; the~Reeb chords $x$ corresponding to the~critical point of index $i$ will be denoted by $x[i]$.
    \begin{lemma}\label{lem:dgaofclef}
        The~Reeb chords of the~$k$-clef sphere in $\R^{2n+1}$ are as in Figure~\ref{fig:halfclefsphere}. Moreover, their grading is as follows:
        \vspace{-6 pt}\begin{align*}
            |s^0_{(k+1)1}[i]|&=i-(k+1), \,\,|s_{11}^1[i]|=i,\,\,|t^0_{m1}[i]|=|r^0_{m1}[i]|=i-m, \\
            |c^1_{02}[i]|&=i+2,\,\, |c^0_{01}|=0,\,\, |g^1_{01}[i]|=i+1,\,\, |h^1_{01}[i]|=i+1,\\
            |d^0_{01}[i]|&=i,\,\,|d^1_{01}[i]|=i+1,\,\,
            |d^0_{12}[i]|=i,\,\,|d^1_{12}[i]|=i+1,
        \end{align*} 
        for $i=0,n-1$ and $1\leq m\leq k$.

        Moreover, the~differential of the~Chekanov-Eliashberg algebra coincides with the~one for the~slice of the~$k$-clef sphere computed in Lemma~\ref{lem:clefstrand} for all chords by replacing the~Reeb chords $x$ with $x[i]$, where $i=0,n-1$, except for the~chords below:
        \begin{align*}
            \partial h^1_{01}[0]&=1+c^0_{01}&\partial h^1_{01}[n-1]&=0,\\
             \partial d^1_{12}[0]&=1+s^1_{11}[0]+d^0_{12}[0]&\partial d^1_{12}[n-1]&=s^1_{11}[n-1]+d^0_{12}[n-1].
        \end{align*}
        In particular, the~minimal chords form a~dg sub-algebra.
    \end{lemma}
    \begin{proof}
        We perturb out the~Morse-Bott loci, the~Reeb chords $x$ corresponding to the~critical point of index $i$ will be denoted by $x[i]$. We determine the~grading of the~Reeb chords of $\Lambda^u_k$, via the~formula
    $$|x^{a}_{cb}[i]|=i+a+b-c-1,$$ where $x$ is a~letter representing a~labelling of a~Reeb chord, $i=0,n-1$ is the~index of the~critical point of the~Morse function used for the~perturbation of the~Morse-Bott loci and $1\leq m\leq k$.
        
     We organise the~Morse-Bott perturbation so that the~minima are close to each other and are separated from the~maxima by a~hyperplane in $\R^{n+1}\times \R$. This immediately implies that no Reeb chords with $i=n-1$ can contribute to the~differential of a~Reeb chord with $i=0$.

    \par As the~Legendrian sphere was defined via the~spinning construction of the~front around a~sphere $S^n$, we define an ''angular slice'' of this Legendrian to be an intersection with the~ray given fixed values of all the~angles giving us the~higher-dimensional spherical coordinates used for spinning. This gives us a~half-plane $\R_{\geq 0}\times \R_z$ in $\R^{n}\times \R_z$, there the~front of the~sphere coincides with the~one-dimensional front.   
    \par Now, we will explain that the~flow-tree contributions of the~Reeb chords corresponding to the~minima of the~Morse-Bott perturbation can be seen in the~one-dimensional front corresponding to one angular slice. This is because there is a~bijection of rigid flow trees.
    \par Note that the~Morse-Bott loci of the~Reeb chord form a~system of concentric spheres $S^{n-1}$ in the~base $\R^{n}$. Before the~Morse-Bott perturbation, the~Morse-Bott loci have index either $0$ or $1$ with their trajectories of the~local difference functions parallel to the~radial direction emanating from the~origin. We can easily arrange the~perturbation of the~Morse-Bott locus to be such that for different local parametrising functions, the~radial trajectories after perturbation intersect the~Morse-Bott loci of all other difference functions transversely.
    
    \par The~bijection is given by projection onto one angular slice, as any trajectories leaving the~perturbed Morse-Bott locus intersect the~Morse-Bott loci of other sheets transversely at a~point in the~stable manifold of the~minimum, inducing a~unique and so rigid flow-line to the~minimum. The~bijection sends the~flow-trees of the~angular slice with a~$P_1$-vertex (which is a~$Y_0$-vertex with a~ghost edge leading to a~$P_0$-vertex) to flow trees, where this vertex is substituted with a~$Y_0$-vertex and the~trajectory contained in the~former Morse-Bott locus leading to the~$P_0$-vertex is uniquely determined by the~intersection of the~radial trajectory with this locus. This means that the~ghosts in the~angular slice trees are substituted by trajectories in the~Morse-Bott loci when lifted to the~higher-dimensional front, see Figure~\ref{fig:minbijectriontrees}. The~bijection also substitutes the~$P_0$-vertex in the~angular slice for a~$Y_0$-vertex and with a~$P_0$-vertex at the~Morse-Bott minimum and edge ending at an $E$-vertex. This proves that the~Morse-Bott minima form a~dg-subalgebra with a~differential as in Lemma~\ref{lem:clefstrand}.
    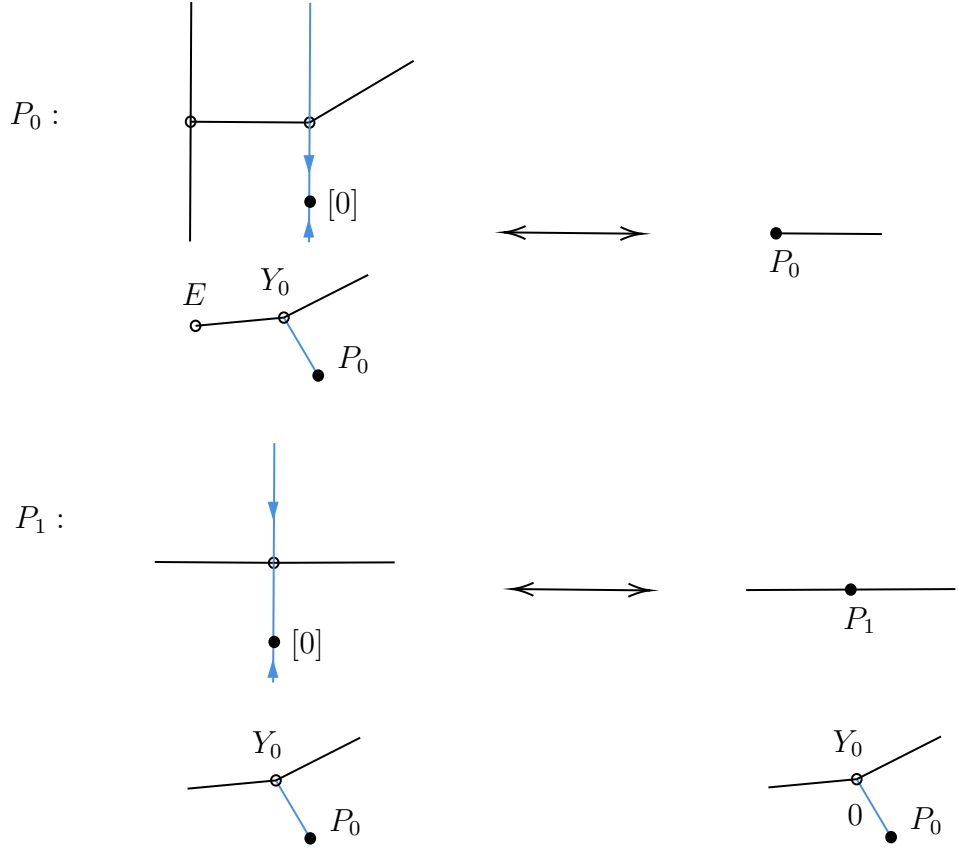
\begin{figure}
        \centering
        \tikzset{every picture/.style={line width=0.75pt}} 
        
        \begin{tikzpicture}[x=0.75pt,y=0.75pt,yscale=-1,xscale=1]
        
        \draw    (467.16,303.64) -- (519.66,303.39) ;
        \draw  [fill={rgb, 255:red, 0; green, 0; blue, 0 }  ,fill opacity=1 ] (517.22,303.39) .. controls (517.22,301.95) and (518.31,300.78) .. (519.66,300.78) .. controls (521.01,300.78) and (522.1,301.95) .. (522.1,303.39) .. controls (522.1,304.84) and (521.01,306.01) .. (519.66,306.01) .. controls (518.31,306.01) and (517.22,304.84) .. (517.22,303.39) -- cycle ;
        \draw [color={rgb, 255:red, 74; green, 144; blue, 226 }  ,draw opacity=1 ]   (248.39,428.26) -- (231.15,399.17) ;
        \draw  [fill={rgb, 255:red, 0; green, 0; blue, 0 }  ,fill opacity=1 ] (245.95,428.26) .. controls (245.95,426.81) and (247.05,425.64) .. (248.39,425.64) .. controls (249.74,425.64) and (250.84,426.81) .. (250.84,428.26) .. controls (250.84,429.7) and (249.74,430.87) .. (248.39,430.87) .. controls (247.05,430.87) and (245.95,429.7) .. (245.95,428.26) -- cycle ;
        \draw   (228.71,399.17) .. controls (228.71,397.72) and (229.8,396.55) .. (231.15,396.55) .. controls (232.5,396.55) and (233.59,397.72) .. (233.59,399.17) .. controls (233.59,400.61) and (232.5,401.79) .. (231.15,401.79) .. controls (229.8,401.79) and (228.71,400.61) .. (228.71,399.17) -- cycle ;
        \draw    (231.15,399.17) -- (186.83,403.31) ;
        \draw    (231.15,399.17) -- (273.45,377.73) ;
        \draw    (248.06,69.01) -- (188.36,68.59) ;
        \draw   (245.62,69.01) .. controls (245.62,67.57) and (246.71,66.4) .. (248.06,66.4) .. controls (249.41,66.4) and (250.5,67.57) .. (250.5,69.01) .. controls (250.5,70.46) and (249.41,71.63) .. (248.06,71.63) .. controls (246.71,71.63) and (245.62,70.46) .. (245.62,69.01) -- cycle ;
        \draw    (535.27,124.97) -- (482.16,124.61) ;
        \draw  [fill={rgb, 255:red, 0; green, 0; blue, 0 }  ,fill opacity=1 ] (479.72,124.61) .. controls (479.72,123.17) and (480.81,121.99) .. (482.16,121.99) .. controls (483.51,121.99) and (484.6,123.17) .. (484.6,124.61) .. controls (484.6,126.05) and (483.51,127.23) .. (482.16,127.23) .. controls (480.81,127.23) and (479.72,126.05) .. (479.72,124.61) -- cycle ;
        \draw [color={rgb, 255:red, 74; green, 144; blue, 226 }  ,draw opacity=1 ]   (252.39,195.98) -- (235.15,166.89) ;
        \draw  [fill={rgb, 255:red, 0; green, 0; blue, 0 }  ,fill opacity=1 ] (249.95,195.98) .. controls (249.95,194.53) and (251.05,193.36) .. (252.39,193.36) .. controls (253.74,193.36) and (254.84,194.53) .. (254.84,195.98) .. controls (254.84,197.42) and (253.74,198.59) .. (252.39,198.59) .. controls (251.05,198.59) and (249.95,197.42) .. (249.95,195.98) -- cycle ;
        \draw   (232.71,166.89) .. controls (232.71,165.44) and (233.8,164.27) .. (235.15,164.27) .. controls (236.5,164.27) and (237.59,165.44) .. (237.59,166.89) .. controls (237.59,168.33) and (236.5,169.5) .. (235.15,169.5) .. controls (233.8,169.5) and (232.71,168.33) .. (232.71,166.89) -- cycle ;
        \draw    (235.15,166.89) -- (190.83,171.03) ;
        \draw    (235.15,166.89) -- (277.45,145.45) ;
        \draw   (185.92,68.59) .. controls (185.92,67.14) and (187.02,65.97) .. (188.36,65.97) .. controls (189.71,65.97) and (190.8,67.14) .. (190.8,68.59) .. controls (190.8,70.03) and (189.71,71.21) .. (188.36,71.21) .. controls (187.02,71.21) and (185.92,70.03) .. (185.92,68.59) -- cycle ;
        \draw   (188.39,171.03) .. controls (188.39,169.59) and (189.48,168.42) .. (190.83,168.42) .. controls (192.17,168.42) and (193.27,169.59) .. (193.27,171.03) .. controls (193.27,172.48) and (192.17,173.65) .. (190.83,173.65) .. controls (189.48,173.65) and (188.39,172.48) .. (188.39,171.03) -- cycle ;
        \draw    (347.5,124.14) -- (413.01,124.79) ;
        \draw [shift={(415.01,124.81)}, rotate = 180.57] [color={rgb, 255:red, 0; green, 0; blue, 0 }  ][line width=0.75]    (10.93,-3.29) .. controls (6.95,-1.4) and (3.31,-0.3) .. (0,0) .. controls (3.31,0.3) and (6.95,1.4) .. (10.93,3.29)   ;
        \draw [shift={(345.5,124.12)}, rotate = 0.57] [color={rgb, 255:red, 0; green, 0; blue, 0 }  ][line width=0.75]    (10.93,-3.29) .. controls (6.95,-1.4) and (3.31,-0.3) .. (0,0) .. controls (3.31,0.3) and (6.95,1.4) .. (10.93,3.29)   ;
        \draw    (519.66,303.39) -- (537,303.31) -- (572.16,303.14) ;
        \draw [color={rgb, 255:red, 74; green, 144; blue, 226 }  ,draw opacity=1 ]   (539.89,427.64) -- (522.65,398.55) ;
        \draw  [fill={rgb, 255:red, 0; green, 0; blue, 0 }  ,fill opacity=1 ] (537.45,427.64) .. controls (537.45,426.19) and (538.55,425.02) .. (539.89,425.02) .. controls (541.24,425.02) and (542.34,426.19) .. (542.34,427.64) .. controls (542.34,429.08) and (541.24,430.25) .. (539.89,430.25) .. controls (538.55,430.25) and (537.45,429.08) .. (537.45,427.64) -- cycle ;
        \draw   (520.21,398.55) .. controls (520.21,397.11) and (521.3,395.93) .. (522.65,395.93) .. controls (524,395.93) and (525.09,397.11) .. (525.09,398.55) .. controls (525.09,400) and (524,401.17) .. (522.65,401.17) .. controls (521.3,401.17) and (520.21,400) .. (520.21,398.55) -- cycle ;
        \draw    (522.65,398.55) -- (478.33,402.69) ;
        \draw    (522.65,398.55) -- (564.95,377.12) ;
        \draw    (351.5,303.14) -- (417.01,303.79) ;
        \draw [shift={(419.01,303.81)}, rotate = 180.57] [color={rgb, 255:red, 0; green, 0; blue, 0 }  ][line width=0.75]    (10.93,-3.29) .. controls (6.95,-1.4) and (3.31,-0.3) .. (0,0) .. controls (3.31,0.3) and (6.95,1.4) .. (10.93,3.29)   ;
        \draw [shift={(349.5,303.12)}, rotate = 0.57] [color={rgb, 255:red, 0; green, 0; blue, 0 }  ][line width=0.75]    (10.93,-3.29) .. controls (6.95,-1.4) and (3.31,-0.3) .. (0,0) .. controls (3.31,0.3) and (6.95,1.4) .. (10.93,3.29)   ;
        \draw    (188.67,8.5) -- (188.06,128.68) ;
        \draw [color={rgb, 255:red, 74; green, 144; blue, 226 }  ,draw opacity=1 ]   (248.36,8.92) -- (247.76,129.1) ;
        \draw  [fill={rgb, 255:red, 0; green, 0; blue, 0 }  ,fill opacity=1 ] (245.88,108.71) .. controls (245.88,107.27) and (246.97,106.1) .. (248.32,106.1) .. controls (249.67,106.1) and (250.76,107.27) .. (250.76,108.71) .. controls (250.76,110.16) and (249.67,111.33) .. (248.32,111.33) .. controls (246.97,111.33) and (245.88,110.16) .. (245.88,108.71) -- cycle ;
        \draw  [color={rgb, 255:red, 74; green, 144; blue, 226 }  ,draw opacity=1 ][fill={rgb, 255:red, 74; green, 144; blue, 226 }  ,fill opacity=1 ][line width=0.75]  (247.66,119.01) -- (249.73,126.22) -- (245.55,126.22) -- cycle ;
        \draw  [color={rgb, 255:red, 74; green, 144; blue, 226 }  ,draw opacity=1 ][fill={rgb, 255:red, 74; green, 144; blue, 226 }  ,fill opacity=1 ][line width=0.75]  (247.83,93.05) -- (245.68,85.87) -- (249.85,85.82) -- cycle ;
        \draw    (248.06,69.01) -- (300.29,38.02) ;
        \draw    (230.06,289.97) -- (170.36,289.55) ;
        \draw   (227.62,289.97) .. controls (227.62,288.52) and (228.71,287.35) .. (230.06,287.35) .. controls (231.41,287.35) and (232.5,288.52) .. (232.5,289.97) .. controls (232.5,291.41) and (231.41,292.58) .. (230.06,292.58) .. controls (228.71,292.58) and (227.62,291.41) .. (227.62,289.97) -- cycle ;
        \draw [color={rgb, 255:red, 74; green, 144; blue, 226 }  ,draw opacity=1 ]   (230.36,229.88) -- (229.76,350.06) ;
        \draw  [fill={rgb, 255:red, 0; green, 0; blue, 0 }  ,fill opacity=1 ] (227.88,329.67) .. controls (227.88,328.22) and (228.97,327.05) .. (230.32,327.05) .. controls (231.67,327.05) and (232.76,328.22) .. (232.76,329.67) .. controls (232.76,331.11) and (231.67,332.28) .. (230.32,332.28) .. controls (228.97,332.28) and (227.88,331.11) .. (227.88,329.67) -- cycle ;
        \draw  [color={rgb, 255:red, 74; green, 144; blue, 226 }  ,draw opacity=1 ][fill={rgb, 255:red, 74; green, 144; blue, 226 }  ,fill opacity=1 ][line width=0.75]  (229.66,339.97) -- (231.73,347.18) -- (227.55,347.18) -- cycle ;
        \draw  [color={rgb, 255:red, 74; green, 144; blue, 226 }  ,draw opacity=1 ][fill={rgb, 255:red, 74; green, 144; blue, 226 }  ,fill opacity=1 ][line width=0.75]  (229.83,267.01) -- (227.68,259.83) -- (231.85,259.78) -- cycle ;
        \draw    (230.06,289.97) -- (290.77,289.74) ;
        
        \draw (514.39,310.59) node [anchor=north west][inner sep=0.75pt]    {$P_{1}$};
        \draw (217.69,372.91) node [anchor=north west][inner sep=0.75pt]    {$Y_{0}$};
        \draw (256.34,411.9) node [anchor=north west][inner sep=0.75pt]    {$P_{0}$};
        \draw (98.5,260.02) node [anchor=north west][inner sep=0.75pt]    {$P_{1} :$};
        \draw (476.89,131.81) node [anchor=north west][inner sep=0.75pt]    {$P_{0}$};
        \draw (221.69,140.63) node [anchor=north west][inner sep=0.75pt]    {$Y_{0}$};
        \draw (260.34,179.62) node [anchor=north west][inner sep=0.75pt]    {$P_{0}$};
        \draw (182.38,148.37) node [anchor=north west][inner sep=0.75pt]    {$E$};
        \draw (255,100.91) node [anchor=north west][inner sep=0.75pt]    {$[ 0]$};
        \draw (96,57.24) node [anchor=north west][inner sep=0.75pt]    {$P_{0} :$};
        \draw (509.19,372.3) node [anchor=north west][inner sep=0.75pt]    {$Y_{0}$};
        \draw (547.84,411.28) node [anchor=north west][inner sep=0.75pt]    {$P_{0}$};
        \draw (516.5,409.8) node [anchor=north west][inner sep=0.75pt]    {$0$};
        \draw (237,321.86) node [anchor=north west][inner sep=0.75pt]    {$[ 0]$};

        \end{tikzpicture}

        \caption{The~bijection of flow trees of the~whole Legendrian (left) and the~angular slice (right). The~black vertical edge corresponds to the~projection of the~singular locus of the~front to the~base. The~vertical blue lines correspond to the~trajectories created by the~perturbation of the~Morse-Bott locus. The~circles correspond to the~image of vertices in the~base, and the~filled circles correspond to the~$P$-type vertices.}
        \label{fig:minbijectriontrees}
    \end{figure}

    \vspace{10 pt}
    \par The~differential of Reeb chords corresponding to the~maximum (denoted by $[n-1]$) is more involved and uses the~additional rotation dimensions.
    \begin{itemize}
        \item Constant terms: If there is a~constant contribution to the~differential of $c[0]$, then there is no contribution to the~differential of $c[n-1]$ because of the~grading shift. The~only negatively graded Reeb chords are the~$r[0],t[0]$-chords and the~$s^0[0]$-chord. And for those, it is impossible to obtain the~constant term even after the~degree shift by $n$ because we enumerated all possible single positive puncture trees in the~computation of their differential in the~$[0]$-case.
        \item Linear terms: The~term $c^0_{01}$ cannot contribute to any linear term of the~differential of the~$[n-1]$-chords. Either the~grading is too big, or it is impossible to find a~rigid intersection of flow-lines. The~rest of the~linear terms are going to be preserved with the~$(n-1)$-grading shift since there is a~bijection of flow-trees realising the~linear contributions to the~$[0]$-chords and the~linear contributions of the~$[n-1]$-chords, see Figure~\ref{fig:linearbijection}.
        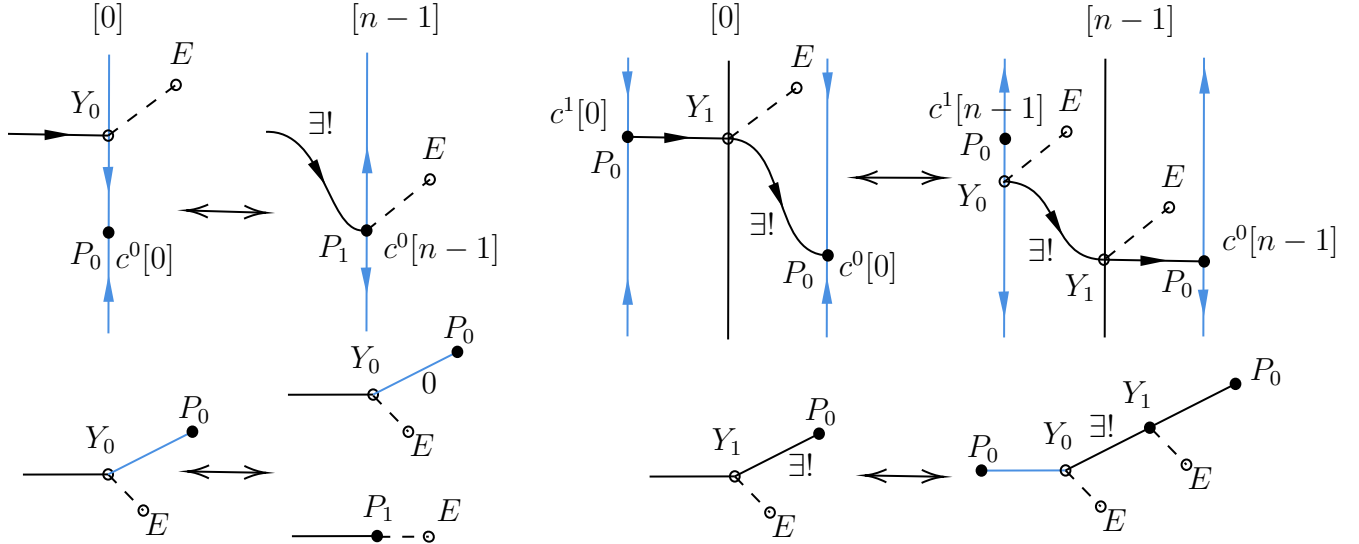
\begin{figure}[htbp!]
            \centering

\tikzset{every picture/.style={line width=0.75pt}} 

\begin{tikzpicture}[x=0.75pt,y=0.75pt,yscale=-1,xscale=1]

\draw    (372.3,43.84) -- (371.96,183.75) ;
\draw [color={rgb, 255:red, 74; green, 144; blue, 226 }  ,draw opacity=1 ]   (510.8,42.84) -- (510.46,182.75) ;
\draw    (561.3,42.84) -- (560.96,182.75) ;
\draw [color={rgb, 255:red, 74; green, 144; blue, 226 }  ,draw opacity=1 ]   (610.8,42.34) -- (610.46,182.25) ;
\draw [color={rgb, 255:red, 74; green, 144; blue, 226 }  ,draw opacity=1 ]   (61.3,40.79) -- (60.96,180.7) ;
\draw  [fill={rgb, 255:red, 0; green, 0; blue, 0 }  ,fill opacity=1 ] (58.61,130.06) .. controls (58.61,128.62) and (59.71,127.45) .. (61.05,127.45) .. controls (62.4,127.45) and (63.49,128.62) .. (63.49,130.06) .. controls (63.49,131.51) and (62.4,132.68) .. (61.05,132.68) .. controls (59.71,132.68) and (58.61,131.51) .. (58.61,130.06) -- cycle ;
\draw    (10.55,80.56) -- (60.87,81.44) ;
\draw  [draw opacity=0][fill={rgb, 255:red, 0; green, 0; blue, 0 }  ,fill opacity=1 ][line width=0.75]  (42.51,80.94) -- (28.95,83.75) -- (28.88,78.41) -- cycle ;
\draw  [draw opacity=0][fill={rgb, 255:red, 74; green, 144; blue, 226 }  ,fill opacity=1 ][line width=0.75]  (61.13,110.75) -- (58.07,97.24) -- (63.41,97.08) -- cycle ;
\draw  [draw opacity=0][fill={rgb, 255:red, 74; green, 144; blue, 226 }  ,fill opacity=1 ][line width=0.75]  (61.08,151.25) -- (63.73,164.85) -- (58.39,164.85) -- cycle ;
\draw   (58.43,81.44) .. controls (58.43,79.99) and (59.52,78.82) .. (60.87,78.82) .. controls (62.22,78.82) and (63.31,79.99) .. (63.31,81.44) .. controls (63.31,82.88) and (62.22,84.05) .. (60.87,84.05) .. controls (59.52,84.05) and (58.43,82.88) .. (58.43,81.44) -- cycle ;
\draw  [dash pattern={on 4.5pt off 4.5pt}]  (60.87,81.44) -- (94.87,55.98) ;
\draw   (92.43,55.98) .. controls (92.43,54.54) and (93.52,53.37) .. (94.87,53.37) .. controls (96.22,53.37) and (97.31,54.54) .. (97.31,55.98) .. controls (97.31,57.43) and (96.22,58.6) .. (94.87,58.6) .. controls (93.52,58.6) and (92.43,57.43) .. (92.43,55.98) -- cycle ;
\draw [color={rgb, 255:red, 74; green, 144; blue, 226 }  ,draw opacity=1 ]   (190.8,39.79) -- (190.46,179.7) ;
\draw  [fill={rgb, 255:red, 0; green, 0; blue, 0 }  ,fill opacity=1 ] (188.11,129.06) .. controls (188.11,127.62) and (189.21,126.45) .. (190.55,126.45) .. controls (191.9,126.45) and (192.99,127.62) .. (192.99,129.06) .. controls (192.99,130.51) and (191.9,131.68) .. (190.55,131.68) .. controls (189.21,131.68) and (188.11,130.51) .. (188.11,129.06) -- cycle ;
\draw    (560.55,143.71) -- (610.87,144.59) ;
\draw  [draw opacity=0][fill={rgb, 255:red, 0; green, 0; blue, 0 }  ,fill opacity=1 ][line width=0.75]  (170,106.45) -- (161.07,95.86) -- (165.74,93.26) -- cycle ;
\draw  [draw opacity=0][fill={rgb, 255:red, 74; green, 144; blue, 226 }  ,fill opacity=1 ][line width=0.75]  (190.63,161.25) -- (187.57,147.74) -- (192.91,147.58) -- cycle ;
\draw  [draw opacity=0][fill={rgb, 255:red, 74; green, 144; blue, 226 }  ,fill opacity=1 ][line width=0.75]  (190.83,86.75) -- (193.48,100.35) -- (188.14,100.35) -- cycle ;
\draw  [dash pattern={on 4.5pt off 4.5pt}]  (190.55,129.06) -- (222.37,103.44) ;
\draw   (219.93,103.44) .. controls (219.93,101.99) and (221.02,100.82) .. (222.37,100.82) .. controls (223.72,100.82) and (224.81,101.99) .. (224.81,103.44) .. controls (224.81,104.88) and (223.72,106.05) .. (222.37,106.05) .. controls (221.02,106.05) and (219.93,104.88) .. (219.93,103.44) -- cycle ;
\draw    (140.05,79.56) .. controls (168.87,78.44) and (170.87,128.44) .. (188.11,129.06) ;
\draw [color={rgb, 255:red, 74; green, 144; blue, 226 }  ,draw opacity=1 ]   (321.8,42.29) -- (321.46,182.2) ;
\draw  [fill={rgb, 255:red, 0; green, 0; blue, 0 }  ,fill opacity=1 ] (319.11,82.06) .. controls (319.11,80.62) and (320.21,79.45) .. (321.55,79.45) .. controls (322.9,79.45) and (323.99,80.62) .. (323.99,82.06) .. controls (323.99,83.51) and (322.9,84.68) .. (321.55,84.68) .. controls (320.21,84.68) and (319.11,83.51) .. (319.11,82.06) -- cycle ;
\draw    (321.55,82.06) -- (371.87,82.94) ;
\draw  [draw opacity=0][fill={rgb, 255:red, 74; green, 144; blue, 226 }  ,fill opacity=1 ][line width=0.75]  (321.74,62.83) -- (318.68,49.32) -- (324.02,49.15) -- cycle ;
\draw  [draw opacity=0][fill={rgb, 255:red, 74; green, 144; blue, 226 }  ,fill opacity=1 ][line width=0.75]  (321.58,152.75) -- (324.23,166.35) -- (318.89,166.35) -- cycle ;
\draw  [draw opacity=0][fill={rgb, 255:red, 0; green, 0; blue, 0 }  ,fill opacity=1 ][line width=0.75]  (353.51,82.44) -- (339.95,85.25) -- (339.88,79.91) -- cycle ;
\draw  [dash pattern={on 4.5pt off 4.5pt}]  (372.37,82.94) -- (406.37,57.48) ;
\draw   (403.93,57.48) .. controls (403.93,56.04) and (405.02,54.87) .. (406.37,54.87) .. controls (407.72,54.87) and (408.81,56.04) .. (408.81,57.48) .. controls (408.81,58.93) and (407.72,60.1) .. (406.37,60.1) .. controls (405.02,60.1) and (403.93,58.93) .. (403.93,57.48) -- cycle ;
\draw   (369.43,82.94) .. controls (369.43,81.49) and (370.52,80.32) .. (371.87,80.32) .. controls (373.22,80.32) and (374.31,81.49) .. (374.31,82.94) .. controls (374.31,84.38) and (373.22,85.55) .. (371.87,85.55) .. controls (370.52,85.55) and (369.43,84.38) .. (369.43,82.94) -- cycle ;
\draw [color={rgb, 255:red, 74; green, 144; blue, 226 }  ,draw opacity=1 ]   (421.8,42.79) -- (421.46,182.7) ;
\draw  [fill={rgb, 255:red, 0; green, 0; blue, 0 }  ,fill opacity=1 ] (419.61,141.56) .. controls (419.61,140.12) and (420.71,138.95) .. (422.05,138.95) .. controls (423.4,138.95) and (424.49,140.12) .. (424.49,141.56) .. controls (424.49,143.01) and (423.4,144.18) .. (422.05,144.18) .. controls (420.71,144.18) and (419.61,143.01) .. (419.61,141.56) -- cycle ;
\draw  [draw opacity=0][fill={rgb, 255:red, 74; green, 144; blue, 226 }  ,fill opacity=1 ][line width=0.75]  (422.24,63.33) -- (419.18,49.82) -- (424.52,49.65) -- cycle ;
\draw  [draw opacity=0][fill={rgb, 255:red, 74; green, 144; blue, 226 }  ,fill opacity=1 ][line width=0.75]  (421.58,152.25) -- (424.23,165.85) -- (418.89,165.85) -- cycle ;
\draw    (372.37,82.94) .. controls (402.64,83.71) and (392.64,141.21) .. (422.05,141.56) ;
\draw  [draw opacity=0][fill={rgb, 255:red, 0; green, 0; blue, 0 }  ,fill opacity=1 ][line width=0.75]  (399.82,118.28) -- (392.12,106.76) -- (397.05,104.7) -- cycle ;
\draw  [fill={rgb, 255:red, 0; green, 0; blue, 0 }  ,fill opacity=1 ] (508.61,83.06) .. controls (508.61,81.62) and (509.71,80.45) .. (511.05,80.45) .. controls (512.4,80.45) and (513.49,81.62) .. (513.49,83.06) .. controls (513.49,84.51) and (512.4,85.68) .. (511.05,85.68) .. controls (509.71,85.68) and (508.61,84.51) .. (508.61,83.06) -- cycle ;
\draw  [draw opacity=0][fill={rgb, 255:red, 74; green, 144; blue, 226 }  ,fill opacity=1 ][line width=0.75]  (510.63,173.3) -- (507.57,159.78) -- (512.91,159.62) -- cycle ;
\draw  [draw opacity=0][fill={rgb, 255:red, 74; green, 144; blue, 226 }  ,fill opacity=1 ][line width=0.75]  (510.58,46.25) -- (513.23,59.85) -- (507.89,59.85) -- cycle ;
\draw  [draw opacity=0][fill={rgb, 255:red, 74; green, 144; blue, 226 }  ,fill opacity=1 ][line width=0.75]  (610.74,172.33) -- (607.68,158.82) -- (613.02,158.65) -- cycle ;
\draw  [draw opacity=0][fill={rgb, 255:red, 74; green, 144; blue, 226 }  ,fill opacity=1 ][line width=0.75]  (611.08,46.75) -- (613.73,60.35) -- (608.39,60.35) -- cycle ;
\draw  [fill={rgb, 255:red, 0; green, 0; blue, 0 }  ,fill opacity=1 ] (608.43,144.59) .. controls (608.43,143.14) and (609.52,141.97) .. (610.87,141.97) .. controls (612.22,141.97) and (613.31,143.14) .. (613.31,144.59) .. controls (613.31,146.03) and (612.22,147.2) .. (610.87,147.2) .. controls (609.52,147.2) and (608.43,146.03) .. (608.43,144.59) -- cycle ;
\draw   (508.19,104.41) .. controls (508.19,102.97) and (509.28,101.8) .. (510.63,101.8) .. controls (511.98,101.8) and (513.07,102.97) .. (513.07,104.41) .. controls (513.07,105.86) and (511.98,107.03) .. (510.63,107.03) .. controls (509.28,107.03) and (508.19,105.86) .. (508.19,104.41) -- cycle ;
\draw    (510.37,104.44) .. controls (540.64,105.21) and (531.14,143.36) .. (560.55,143.71) ;
\draw  [draw opacity=0][fill={rgb, 255:red, 0; green, 0; blue, 0 }  ,fill opacity=1 ][line width=0.75]  (540,131.95) -- (531.07,121.36) -- (535.74,118.76) -- cycle ;
\draw   (558.11,143.71) .. controls (558.11,142.27) and (559.2,141.09) .. (560.55,141.09) .. controls (561.9,141.09) and (562.99,142.27) .. (562.99,143.71) .. controls (562.99,145.15) and (561.9,146.33) .. (560.55,146.33) .. controls (559.2,146.33) and (558.11,145.15) .. (558.11,143.71) -- cycle ;
\draw  [draw opacity=0][fill={rgb, 255:red, 0; green, 0; blue, 0 }  ,fill opacity=1 ][line width=0.75]  (592.51,144.09) -- (578.94,146.9) -- (578.88,141.56) -- cycle ;
\draw  [dash pattern={on 4.5pt off 4.5pt}]  (510.05,105.06) -- (541.87,79.44) ;
\draw   (539.43,79.44) .. controls (539.43,77.99) and (540.52,76.82) .. (541.87,76.82) .. controls (543.22,76.82) and (544.31,77.99) .. (544.31,79.44) .. controls (544.31,80.88) and (543.22,82.05) .. (541.87,82.05) .. controls (540.52,82.05) and (539.43,80.88) .. (539.43,79.44) -- cycle ;
\draw  [dash pattern={on 4.5pt off 4.5pt}]  (561.55,143.06) -- (593.37,117.44) ;
\draw   (590.93,117.44) .. controls (590.93,115.99) and (592.02,114.82) .. (593.37,114.82) .. controls (594.72,114.82) and (595.81,115.99) .. (595.81,117.44) .. controls (595.81,118.88) and (594.72,120.05) .. (593.37,120.05) .. controls (592.02,120.05) and (590.93,118.88) .. (590.93,117.44) -- cycle ;
\draw    (18.05,251.56) -- (60.81,251.47) ;
\draw   (58.37,251.47) .. controls (58.37,250.03) and (59.46,248.86) .. (60.81,248.86) .. controls (62.16,248.86) and (63.25,250.03) .. (63.25,251.47) .. controls (63.25,252.92) and (62.16,254.09) .. (60.81,254.09) .. controls (59.46,254.09) and (58.37,252.92) .. (58.37,251.47) -- cycle ;
\draw  [dash pattern={on 4.5pt off 4.5pt}]  (60.81,251.47) -- (78.49,269.62) ;
\draw   (76.04,269.62) .. controls (76.04,268.17) and (77.14,267) .. (78.49,267) .. controls (79.83,267) and (80.93,268.17) .. (80.93,269.62) .. controls (80.93,271.06) and (79.83,272.23) .. (78.49,272.23) .. controls (77.14,272.23) and (76.04,271.06) .. (76.04,269.62) -- cycle ;
\draw [color={rgb, 255:red, 74; green, 144; blue, 226 }  ,draw opacity=1 ]   (60.81,251.47) -- (103.1,230.04) ;
\draw  [fill={rgb, 255:red, 0; green, 0; blue, 0 }  ,fill opacity=1 ] (100.66,230.04) .. controls (100.66,228.59) and (101.76,227.42) .. (103.1,227.42) .. controls (104.45,227.42) and (105.55,228.59) .. (105.55,230.04) .. controls (105.55,231.48) and (104.45,232.65) .. (103.1,232.65) .. controls (101.76,232.65) and (100.66,231.48) .. (100.66,230.04) -- cycle ;
\draw    (151.05,211.5) -- (193.81,211.41) ;
\draw   (191.37,211.41) .. controls (191.37,209.97) and (192.46,208.8) .. (193.81,208.8) .. controls (195.16,208.8) and (196.25,209.97) .. (196.25,211.41) .. controls (196.25,212.86) and (195.16,214.03) .. (193.81,214.03) .. controls (192.46,214.03) and (191.37,212.86) .. (191.37,211.41) -- cycle ;
\draw  [dash pattern={on 4.5pt off 4.5pt}]  (193.81,211.91) -- (211.49,230.06) ;
\draw   (209.04,230.06) .. controls (209.04,228.61) and (210.14,227.44) .. (211.49,227.44) .. controls (212.83,227.44) and (213.93,228.61) .. (213.93,230.06) .. controls (213.93,231.5) and (212.83,232.67) .. (211.49,232.67) .. controls (210.14,232.67) and (209.04,231.5) .. (209.04,230.06) -- cycle ;
\draw [color={rgb, 255:red, 74; green, 144; blue, 226 }  ,draw opacity=1 ]   (193.81,211.41) -- (236.1,189.98) ;
\draw  [fill={rgb, 255:red, 0; green, 0; blue, 0 }  ,fill opacity=1 ] (233.66,189.98) .. controls (233.66,188.53) and (234.76,187.36) .. (236.1,187.36) .. controls (237.45,187.36) and (238.55,188.53) .. (238.55,189.98) .. controls (238.55,191.42) and (237.45,192.59) .. (236.1,192.59) .. controls (234.76,192.59) and (233.66,191.42) .. (233.66,189.98) -- cycle ;
\draw    (153.05,282.5) -- (195.81,282.41) ;
\draw  [dash pattern={on 4.5pt off 4.5pt}]  (196.41,282.8) -- (221.74,282.72) ;
\draw   (220.05,284.47) .. controls (219.01,283.47) and (218.93,281.87) .. (219.87,280.9) .. controls (220.8,279.93) and (222.4,279.96) .. (223.44,280.97) .. controls (224.48,281.97) and (224.56,283.57) .. (223.62,284.54) .. controls (222.69,285.51) and (221.09,285.48) .. (220.05,284.47) -- cycle ;
\draw  [fill={rgb, 255:red, 0; green, 0; blue, 0 }  ,fill opacity=1 ] (193.37,282.41) .. controls (193.37,280.97) and (194.46,279.8) .. (195.81,279.8) .. controls (197.16,279.8) and (198.25,280.97) .. (198.25,282.41) .. controls (198.25,283.86) and (197.16,285.03) .. (195.81,285.03) .. controls (194.46,285.03) and (193.37,283.86) .. (193.37,282.41) -- cycle ;
\draw    (102,119.06) -- (138.05,120.11) ;
\draw [shift={(140.05,120.16)}, rotate = 181.67] [color={rgb, 255:red, 0; green, 0; blue, 0 }  ][line width=0.75]    (10.93,-3.29) .. controls (6.95,-1.4) and (3.31,-0.3) .. (0,0) .. controls (3.31,0.3) and (6.95,1.4) .. (10.93,3.29)   ;
\draw [shift={(100,119)}, rotate = 1.67] [color={rgb, 255:red, 0; green, 0; blue, 0 }  ][line width=0.75]    (10.93,-3.29) .. controls (6.95,-1.4) and (3.31,-0.3) .. (0,0) .. controls (3.31,0.3) and (6.95,1.4) .. (10.93,3.29)   ;
\draw    (102.5,250.03) -- (138.05,250.63) ;
\draw [shift={(140.05,250.66)}, rotate = 180.96] [color={rgb, 255:red, 0; green, 0; blue, 0 }  ][line width=0.75]    (10.93,-3.29) .. controls (6.95,-1.4) and (3.31,-0.3) .. (0,0) .. controls (3.31,0.3) and (6.95,1.4) .. (10.93,3.29)   ;
\draw [shift={(100.5,250)}, rotate = 0.96] [color={rgb, 255:red, 0; green, 0; blue, 0 }  ][line width=0.75]    (10.93,-3.29) .. controls (6.95,-1.4) and (3.31,-0.3) .. (0,0) .. controls (3.31,0.3) and (6.95,1.4) .. (10.93,3.29)   ;
\draw    (438,102.51) -- (478.05,102.66) ;
\draw [shift={(480.05,102.66)}, rotate = 180.21] [color={rgb, 255:red, 0; green, 0; blue, 0 }  ][line width=0.75]    (10.93,-3.29) .. controls (6.95,-1.4) and (3.31,-0.3) .. (0,0) .. controls (3.31,0.3) and (6.95,1.4) .. (10.93,3.29)   ;
\draw [shift={(436,102.5)}, rotate = 0.21] [color={rgb, 255:red, 0; green, 0; blue, 0 }  ][line width=0.75]    (10.93,-3.29) .. controls (6.95,-1.4) and (3.31,-0.3) .. (0,0) .. controls (3.31,0.3) and (6.95,1.4) .. (10.93,3.29)   ;
\draw    (332.55,252.56) -- (375.31,252.47) ;
\draw   (372.87,252.47) .. controls (372.87,251.03) and (373.96,249.86) .. (375.31,249.86) .. controls (376.66,249.86) and (377.75,251.03) .. (377.75,252.47) .. controls (377.75,253.92) and (376.66,255.09) .. (375.31,255.09) .. controls (373.96,255.09) and (372.87,253.92) .. (372.87,252.47) -- cycle ;
\draw  [dash pattern={on 4.5pt off 4.5pt}]  (375.31,252.47) -- (392.99,270.62) ;
\draw   (390.54,270.62) .. controls (390.54,269.17) and (391.64,268) .. (392.99,268) .. controls (394.33,268) and (395.43,269.17) .. (395.43,270.62) .. controls (395.43,272.06) and (394.33,273.23) .. (392.99,273.23) .. controls (391.64,273.23) and (390.54,272.06) .. (390.54,270.62) -- cycle ;
\draw [color={rgb, 255:red, 0; green, 0; blue, 0 }  ,draw opacity=1 ]   (375.31,252.47) -- (417.6,231.04) ;
\draw  [fill={rgb, 255:red, 0; green, 0; blue, 0 }  ,fill opacity=1 ] (415.16,231.04) .. controls (415.16,229.59) and (416.26,228.42) .. (417.6,228.42) .. controls (418.95,228.42) and (420.05,229.59) .. (420.05,231.04) .. controls (420.05,232.48) and (418.95,233.65) .. (417.6,233.65) .. controls (416.26,233.65) and (415.16,232.48) .. (415.16,231.04) -- cycle ;
\draw [color={rgb, 255:red, 74; green, 144; blue, 226 }  ,draw opacity=1 ]   (498.55,249.56) -- (541.31,249.47) ;
\draw   (538.87,249.47) .. controls (538.87,248.03) and (539.96,246.86) .. (541.31,246.86) .. controls (542.66,246.86) and (543.75,248.03) .. (543.75,249.47) .. controls (543.75,250.92) and (542.66,252.09) .. (541.31,252.09) .. controls (539.96,252.09) and (538.87,250.92) .. (538.87,249.47) -- cycle ;
\draw  [dash pattern={on 4.5pt off 4.5pt}]  (541.31,249.47) -- (558.99,267.62) ;
\draw   (556.54,267.62) .. controls (556.54,266.17) and (557.64,265) .. (558.99,265) .. controls (560.33,265) and (561.43,266.17) .. (561.43,267.62) .. controls (561.43,269.06) and (560.33,270.23) .. (558.99,270.23) .. controls (557.64,270.23) and (556.54,269.06) .. (556.54,267.62) -- cycle ;
\draw [color={rgb, 255:red, 0; green, 0; blue, 0 }  ,draw opacity=1 ]   (541.31,249.47) -- (583.6,228.04) ;
\draw  [fill={rgb, 255:red, 0; green, 0; blue, 0 }  ,fill opacity=1 ] (581.16,228.04) .. controls (581.16,226.59) and (582.26,225.42) .. (583.6,225.42) .. controls (584.95,225.42) and (586.05,226.59) .. (586.05,228.04) .. controls (586.05,229.48) and (584.95,230.65) .. (583.6,230.65) .. controls (582.26,230.65) and (581.16,229.48) .. (581.16,228.04) -- cycle ;
\draw  [fill={rgb, 255:red, 0; green, 0; blue, 0 }  ,fill opacity=1 ] (496.66,249.54) .. controls (496.66,248.09) and (497.76,246.92) .. (499.1,246.92) .. controls (500.45,246.92) and (501.55,248.09) .. (501.55,249.54) .. controls (501.55,250.98) and (500.45,252.15) .. (499.1,252.15) .. controls (497.76,252.15) and (496.66,250.98) .. (496.66,249.54) -- cycle ;
\draw [color={rgb, 255:red, 0; green, 0; blue, 0 }  ,draw opacity=1 ]   (584.31,227.47) -- (626.6,206.04) ;
\draw  [fill={rgb, 255:red, 0; green, 0; blue, 0 }  ,fill opacity=1 ] (624.16,206.04) .. controls (624.16,204.59) and (625.26,203.42) .. (626.6,203.42) .. controls (627.95,203.42) and (629.05,204.59) .. (629.05,206.04) .. controls (629.05,207.48) and (627.95,208.65) .. (626.6,208.65) .. controls (625.26,208.65) and (624.16,207.48) .. (624.16,206.04) -- cycle ;
\draw  [dash pattern={on 4.5pt off 4.5pt}]  (584.31,228.47) -- (601.99,246.62) ;
\draw   (599.54,246.62) .. controls (599.54,245.17) and (600.64,244) .. (601.99,244) .. controls (603.33,244) and (604.43,245.17) .. (604.43,246.62) .. controls (604.43,248.06) and (603.33,249.23) .. (601.99,249.23) .. controls (600.64,249.23) and (599.54,248.06) .. (599.54,246.62) -- cycle ;
\draw    (442.5,252.03) -- (478.05,252.63) ;
\draw [shift={(480.05,252.66)}, rotate = 180.96] [color={rgb, 255:red, 0; green, 0; blue, 0 }  ][line width=0.75]    (10.93,-3.29) .. controls (6.95,-1.4) and (3.31,-0.3) .. (0,0) .. controls (3.31,0.3) and (6.95,1.4) .. (10.93,3.29)   ;
\draw [shift={(440.5,252)}, rotate = 0.96] [color={rgb, 255:red, 0; green, 0; blue, 0 }  ][line width=0.75]    (10.93,-3.29) .. controls (6.95,-1.4) and (3.31,-0.3) .. (0,0) .. controls (3.31,0.3) and (6.95,1.4) .. (10.93,3.29)   ;

\draw (40.5,57.9) node [anchor=north west][inner sep=0.75pt]    {$Y_{0}$};
\draw (90,32.9) node [anchor=north west][inner sep=0.75pt]    {$E$};
\draw (40,133.4) node [anchor=north west][inner sep=0.75pt]    {$P_{0}$};
\draw (216.5,80.9) node [anchor=north west][inner sep=0.75pt]    {$E$};
\draw (164.5,129.9) node [anchor=north west][inner sep=0.75pt]    {$P_{1}$};
\draw (301,87.9) node [anchor=north west][inner sep=0.75pt]    {$P_{0}$};
\draw (349,59.9) node [anchor=north west][inner sep=0.75pt]    {$Y_{1}$};
\draw (401.5,34.4) node [anchor=north west][inner sep=0.75pt]    {$E$};
\draw (397.5,142.9) node [anchor=north west][inner sep=0.75pt]    {$P_{0}$};
\draw (63.05,133.46) node [anchor=north west][inner sep=0.75pt]    {$c^{0}[ 0]$};
\draw (197.5,128.9) node [anchor=north west][inner sep=0.75pt]    {$c^{0}[n-1]$};
\draw (361,13.4) node [anchor=north west][inner sep=0.75pt]    {$[ 0]$};
\draw (51,13.4) node [anchor=north west][inner sep=0.75pt]    {$[ 0]$};
\draw (181,13.4) node [anchor=north west][inner sep=0.75pt]    {$[n-1]$};
\draw (549.5,14.4) node [anchor=north west][inner sep=0.75pt]    {$[n-1]$};
\draw (281,60.9) node [anchor=north west][inner sep=0.75pt]    {$c^{1}[ 0]$};
\draw (159.46,67.9) node [anchor=north west][inner sep=0.75pt]    {$\exists !$};
\draw (380.96,118.4) node [anchor=north west][inner sep=0.75pt]    {$\exists !$};
\draw (520.96,131.9) node [anchor=north west][inner sep=0.75pt]    {$\exists !$};
\draw (587.71,147.55) node [anchor=north west][inner sep=0.75pt]    {$P_{0}$};
\draw (539.5,149.4) node [anchor=north west][inner sep=0.75pt]    {$Y_{1}$};
\draw (536,56.9) node [anchor=north west][inner sep=0.75pt]    {$E$};
\draw (486.5,78.9) node [anchor=north west][inner sep=0.75pt]    {$P_{0}$};
\draw (485.46,104.4) node [anchor=north west][inner sep=0.75pt]    {$Y_{0}$};
\draw (587.5,94.9) node [anchor=north west][inner sep=0.75pt]    {$E$};
\draw (471,58.4) node [anchor=north west][inner sep=0.75pt]    {$c^{1}[n-1]$};
\draw (618.46,124.9) node [anchor=north west][inner sep=0.75pt]    {$c^{0}[n-1]$};
\draw (425.99,137.96) node [anchor=north west][inner sep=0.75pt]    {$c^{0}[ 0]$};
\draw (47.85,225.72) node [anchor=north west][inner sep=0.75pt]    {$Y_{0}$};
\draw (94.22,210.73) node [anchor=north west][inner sep=0.75pt]    {$P_{0}$};
\draw (77.54,269.46) node [anchor=north west][inner sep=0.75pt]    {$E$};
\draw (180.85,185.66) node [anchor=north west][inner sep=0.75pt]    {$Y_{0}$};
\draw (227.22,170.67) node [anchor=north west][inner sep=0.75pt]    {$P_{0}$};
\draw (210.54,229.9) node [anchor=north west][inner sep=0.75pt]    {$E$};
\draw (216.5,198.84) node [anchor=north west][inner sep=0.75pt]    {$0$};
\draw (186.72,260.23) node [anchor=north west][inner sep=0.75pt]    {$P_{1}$};
\draw (224.04,262.4) node [anchor=north west][inner sep=0.75pt]    {$E$};
\draw (362.35,226.72) node [anchor=north west][inner sep=0.75pt]    {$Y_{1}$};
\draw (408.72,211.73) node [anchor=north west][inner sep=0.75pt]    {$P_{0}$};
\draw (392.04,270.46) node [anchor=north west][inner sep=0.75pt]    {$E$};
\draw (528.35,223.72) node [anchor=north west][inner sep=0.75pt]    {$Y_{0}$};
\draw (568.22,203.23) node [anchor=north west][inner sep=0.75pt]    {$Y_{1}$};
\draw (558.04,267.46) node [anchor=north west][inner sep=0.75pt]    {$E$};
\draw (490.22,230.23) node [anchor=north west][inner sep=0.75pt]    {$P_{0}$};
\draw (601.04,246.46) node [anchor=north west][inner sep=0.75pt]    {$E$};
\draw (632.22,191.73) node [anchor=north west][inner sep=0.75pt]    {$P_{0}$};
\draw (400.46,239.4) node [anchor=north west][inner sep=0.75pt]    {$\exists !$};
\draw (551.46,219.9) node [anchor=north west][inner sep=0.75pt]    {$\exists !$};

\end{tikzpicture}

            \caption{Bijection of the~flow trees comparing the~trees contributing to the~differential of some $c[0]$ and $c[n-1]$. Here the~notation $c^i$ denotes the~Reeb chords of the~one-dimensional front given by a~critical point of the~local difference function of index $i$. The~bijection for trees with positive puncture at $c^0$ is analogous with slightly more complicated trees.}
            \label{fig:linearbijection}
        \end{figure}

        \item Quadratic terms: Any tree representing a~quadratic differential must contain a~$Y_0$-vertex, note that $P_1$-vertex contains a~$Y_0$-vertex and a~ghost. This $Y_0$ vertex must be performed at the~$[n-1]$-part, otherwise the~contribution would not have the~correct grading due to the~degree shift of the~positive puncture. This intersection is going to be the~intersection of the~radial trajectory and a~Morse-Bott unstable manifold. This means that for every contribution $c[0]d[0]$ of $\partial a[0]$ one obtains two contributions $c[n-1]d[0]$ and $c[0]d[n-1]$ of $\partial a[n-1]$. But the~differential of the~minimal part is linear so no quadratic term can occur.
        \item Higher terms: Here, a~similar statement for trees with several $Y_0$-vertices in the~$[n-1]$-part could be written and the~proof is analogous.
    \end{itemize}

    \end{proof}
    
     \subsection{Floor-crossing and Chekanov-Eliashberg algebra}\label{sec: Floor-crossing and CE}
     In this section, we compute the~Chekanov-Eliashberg algebra of the~$k$-twist spheres $\Lambda_k$ and prove Proposition~\ref{prop:Poincare computation}.
     
     The~following lemma shows the~utility of the~floor-crossing in the~computation of Chekanov-Eliashberg algebra under the~clasp move. Once one computes the~differential of the~original Legendrian before the~clasp move and determines the~moduli spaces with multiple positive punctures, one obtains the~differential of the~Legendrian after the~clasp move. However, in our case, these more complicated moduli spaces are all empty, which follows from Lemma~\ref{lem:no pos puncture} below.
     \newpage
    \begin{lemma}\label{lem: positive punctures at clasping chord}
        If $s^0_{(k+1)1}[0]$ is a~positive puncture of a~rigid flow-tree, then the~vertex representing the~puncture must be a~$P_1$-vertex.
    \end{lemma}
    \begin{proof}
        The~edge of the~trajectory emanating from the~puncture must follow the~negative gradient flow-line of the~function $f_{1}-f_{k+1}$, where $f_j$ is the~local parametrisation of the~sheet of the~front with Maslov potential $j$. However, this difference has a~minimum at the~critical point representing $s^0_{(k+1)1}$ and so no non-constant negative flow-lines emanate from it. This rules out the~possibility of a~positive puncture at a~$P_0$-vertex.
    \end{proof}
    \begin{lemma}\label{lem:no pos puncture}
        There is no rigid flow-tree with a~positive puncture at $s^0_{(k+1)1}[0]$.
    \end{lemma}
    \begin{proof}
        Consider three sheets $S_a,S_b,S_c$ of a~front of a~Legendrian that are parametrised by functions $f_a,f_b,f_c$ and locally
        $$f_a>f_b>f_c.$$ By previous Lemma~\ref{lem: positive punctures at clasping chord}, we know that the~positive puncture must be a~$P_1$-vertex. That is, the~outgoing edge follows the~negative gradient trajectory of $f_a-f_c$ and the~incoming edge follows negative gradient trajectory of
        \begin{itemize}
            \item $f_a-f_b$ if the~puncture is at the~critical point of $f_b-f_c$,
            \item $f_b-f_c$ if the~puncture is at the~critical point of $f_a-f_b$.
        \end{itemize}
        In what follows, we will rule out the~contributions depending on the~sheets $S_a,S_b,S_c$.
        
        \par \textbf{Case I:} \textit{ $S_a$ is the~top sheet of Maslov potential $2$, $S_b$ is the~sheet of Maslov potential $1$, where the~Reeb chord $s^0_{(k+1)1}[0]$ ends, and $S_c$ is the~sheet of Maslov potential $k+1$.} \newline Observe that every tree with positive puncture at $s^0_{(k+1)1}[0]$ must contain the~following partial tree depicted in Figure~\ref{fig:caseisubtree} in this case.

        \begin{figure}[htbp!]
            \centering
            \tikzset{every picture/.style={line width=0.75pt}} 
            \setlength{\unitlength}{0.1\textwidth}
    		\begin{picture}(10,2.2)
    			\put(1.5,0){\includegraphics[scale=1]{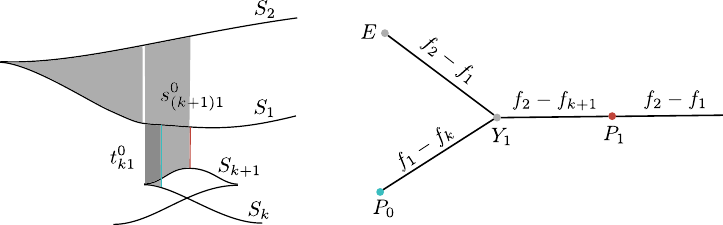}}
    		\end{picture}
            \caption{Case I partial tree.}
            \label{fig:caseisubtree}
        \end{figure}
        As there are no chords between $S_2$ and $S_1$, there must be another negative puncture at a~Reeb chord that ends on the~sheet $S_1$. These are the~chords $r^0_{m1}[i]$ for all $m$ and the~chord $s^1_{11}[i]$. Say that the~negative puncture is at $r^0_{m1}[i]$. The~flow-tree can contain the~partial tree above only if it includes a~partial tree $T$, where the~negative puncture is at a~$P_1$ vertex that is preceded by a~$Y_1$-vertex at the~singular set of the~front of between the~sheets $S_{m+1}$ and $S_m$. However, as the~branch of the~$Y_1$-vertex which gives us the~negative puncture is determined uniquely for $i=n-1$, the~$Y_1$-vertex is at a~point so that the~remaining flow-line generically misses the~point representing the~chord $s^0_{(k+1)1}$. For $i=0$, one observes that the~$Y_1$-vertex cannot be made so that the~resulting tree remains rigid.  In conclusion, any such tree would be non-generic. The~remaining case of the~negative puncture at $s^1_{11}[i]$ is analogous. And so, there are no such rigid generic flow-trees. 
        
        \par \textbf{Case II:} \textit{ $S_a$ is the~sheet of Maslov potential $1$, where the~Reeb chord $s^0_{(k+1)1}$ ends, $S_b$ is the~sheet of Maslov potential $k+1$, and $S_c$ is either one of the~sheets in the~part of the~front obtained by iterative application of RMI moves, or the~sheet of Maslov potential $0$. }
        \newline Observe that every tree with positive puncture at $s^0_{(k+1)1}$ must contain the~following partial tree depicted in Figure~\ref{fig:caseiisubtree} in this case. 
        \begin{figure}[htbp!]
            \centering
            \tikzset{every picture/.style={line width=0.75pt}} 
            \setlength{\unitlength}{0.1\textwidth}
    		\begin{picture}(10,2.2)
    			\put(3,0){\includegraphics[scale=1]{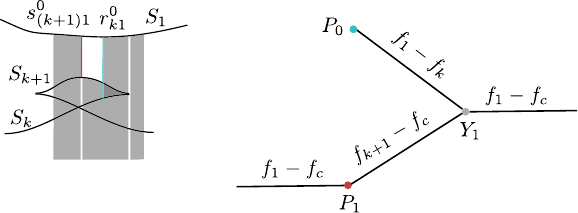}}
    		\end{picture}
            \caption{Case II partial tree.}
            \label{fig:caseiisubtree}
        \end{figure}
        The~argument here is almost vice-versa to the~Case I, and it one rules out the~possibility of a~rigid generic tree of this configuration. This finishes the~proof.
    \end{proof}
    \begin{lemma}\label{lem:applicationofclasp}
        Let $\Lambda_{-\mu}=\Lambda^u_k$ be a~$n$-dimensional Legendrian $k$-clef sphere, with the~clasping chord $s_{-\mu}=s^0_{(k+1)1}[0]$ of grading $|s_{-\mu}|=-(k+1)$. Denote by $(\mathcal{A}(\L_{\tau}),\partial_\tau)$ the~Chekanov-Eliashberg algebra of $\L_{\tau}$ for $\tau\in [-\mu,\mu]$. Let $\L_{\mu}=\Lambda_k$ be the~Legendrian obtained via the~clasp move centred at $s_{-\mu}$ as in Definition~\ref{def:clasp move} and let $(\mathcal{A}(\Lambda_{\mu}),\partial_\mu)$ be its Chekanov-Eliashberg algebra. The~differential $\partial_{\mu}$ is obtained from $\partial_{-\mu}$ substituting $s_{-\mu}$ by $0$ and defining $\partial s_{\mu}=0$.
    \end{lemma}
    \begin{proof}
        Lemma~\ref{lem:CZ lemma} implies that $|s_{\mu}|=n+k-1$ and Lemma~\ref{lem:nonewReebchords} imply that the~grading of all other Reeb chords of $\Lambda_\mu$ was left unchanged. Thus $s_\mu$ cannot be in the~image of $\partial_\mu$. Moreover, $\action{s_{\mu}}=|\mu|$ has the~smallest action among all Reeb chords of $\Lambda_\mu$ and so $\partial_\mu s_\mu=0$.

        Using the~floor-crossing, see Theorem~\ref{thm:floorcrossing}, and the~correspondence of rigid flow-trees and rigid $J$-holomorphic curves from \cite{Ekholmtrees}, we know that the~differential of $\Lambda_k$ acquires no new terms in view of Lemma~\ref{lem:no pos puncture}. Therefore, the~differential $\partial_\mu$ is given by the~same expressions on generators as $\partial_{-\mu}$, except for the~contributions containing $s_{-\mu}$, which will not contribute to the~differential $\partial_\mu$.
    \end{proof}
    \begin{lemma}\label{lem:augmentations}
        If $k>0$, then all graded augmentations $\varepsilon:\mathcal{A}(\Lambda_k)\to \Z_2$ of $\mathcal{A}(\Lambda_k)$ are dg-homotopic and this homotopy class is non-empty.
        If $k=0$, then there are two distinct dg-homotopy classes of augmentations $\varepsilon_0$ and $\varepsilon_1$.
    \end{lemma}
    \begin{proof}
         Note that two augmentations $\varepsilon_1$ and $\varepsilon_2$ are dg-homotopic if and only if there is a~map $K(\varepsilon_1,\varepsilon_2):C(\Lambda_k)\to \Z_2$ of degree $+1$ so that
        \begin{equation}\label{eq:antiderivation}
            \varepsilon_1-\varepsilon_2=K(\varepsilon_1,\varepsilon_2)\circ \partial^{\varepsilon_1\varepsilon_2},
        \end{equation} where $C(\Lambda_k)$ denotes the~vector space over $\Z_2$ generated by the~Reeb chords of $\Lambda_k$, and $\partial^{\varepsilon_1\varepsilon_2}$ is the~bilinearised differential. For the~proof of this fact, see \cite{geog}.

        Since our differential $\partial$ has no higher-order contribution but only the~linear and constant terms, we observe that the~bilinearised differential $D=\partial^{\varepsilon_1\varepsilon_2}$ is the~same for any two augmentations $\varepsilon_1,\varepsilon_2$. Observe that the~map $K$ from Equation \eqref{eq:antiderivation} must vanish on all Reeb chords of grading different from $-1$. 

        When $k=0$ and $n\geq1$, then $\Lambda_k$ has no negatively graded Reeb chords, this means that any $K$ satisfying \eqref{eq:antiderivation} must be the~zero map. Nevertheless, straightforward solution of the~equation 
        \begin{equation}\label{eq:augmentationequation}
            \varepsilon\circ\partial=0
        \end{equation} yields two augmentations $\varepsilon_j$ so that $\varepsilon_j(s^1_{11}[0])=j$ for $j=0,1$. This means that the~difference $\varepsilon_1-\varepsilon_2$ in nonzero and so \eqref{eq:antiderivation} is never satisfied.
        
        For the~other cases, we solve \eqref{eq:augmentationequation} and obtain the~following possible augmentations of $\mathcal{A}(\Lambda_k)$.
        
        \begin{table}[htbp!]
            \begin{tabular}{|l|cccccccccc|}
            \hline
             & \multicolumn{2}{c|}{$k=0$, $n\geq 2$}                         & \multicolumn{2}{c|}{$n-2>k>0$}                         & \multicolumn{2}{c|}{$n-2=k>0$}                         & \multicolumn{4}{c|}{$n-1\leq k$}                                                                         \\ \hline
             & \multicolumn{1}{c|}{$\varepsilon_0$} & \multicolumn{1}{c|}{$\varepsilon_1$} & \multicolumn{1}{c|}{$\varepsilon_0$} & \multicolumn{1}{c|}{$\varepsilon_1$} & \multicolumn{1}{c|}{$\varepsilon_0$} & \multicolumn{1}{c|}{$\varepsilon_1$} & \multicolumn{1}{c|}{$\varepsilon_{01}$} & \multicolumn{1}{c|}{$\varepsilon_{11}$} & \multicolumn{1}{c|}{$\varepsilon_{00}$} & \multicolumn{1}{c|}{$\varepsilon_{10}$} \\ \hline
             $c^0_{01}$&        1               &         1              &     1                  &        1               &   1                    &   1                    &          1             &        1               &        1               &    1                   \\ 
             $s^1_{11}[0]$&         0              &                1       &      0                 &     1                  & 0                      &     1                  &        0               &    1                   &  0                     &  1                     \\ 
             $d^0_{01}[0]$&                  1     &              0         &         1              &     0                  &1                       &    0                   &       1                &    0                   &   1                    &             0          \\ 
             $d^0_{12}[0]$&    1                   &    0                   &           1            &   0                    & 1                      &   0                    &    1                   &    0                   & 1                      &  0                     \\ 
             $s^0_{n1}[n-1]$&      -                 &   -                    &    -                &            -      & 0                    &             0       &         -              &       -                & -                      &               -        \\ 
             $t^0_{n1}[0]$&    -                   &              -         &                -       &      -                 & -                      &    -                   &   1                    &    0                   &       1                &        0               \\ 
             
             $r^0_{n1}[n-1]$&          -             &          -             &   -                    &   -                    &         -              &       -                &    1                   &    0                   & 1                      &          0             \\ \hline
            \end{tabular}
            \end{table}
        Straightforward computation yields that the~possible maps $K$ for pairs $(\varepsilon_0,\varepsilon_1)$, $(\varepsilon_{01},\varepsilon_{11})$,$(\varepsilon_{11},\varepsilon_{00})$, and $(\varepsilon_{00},\varepsilon_{10})$ would have to satisfy the~following systems
        \begin{align*}
            n-2\geq  k &>0\,\,\,\,\, \begin{cases}
                K(t^0_{11}[0])-K(r^0_{11}[0])=1 \\
                K(c)=0 \,\, \text{ for }\,\, c\neq t^0_{11}[0],r^0_{11}[0]
            \end{cases}\\
            n-1 =  k &>0\,\,\,\,\,\begin{cases}
                K(t^0_{11}[0])-K(r^0_{11}[0])=1, \\
                K(s^0_{(n+1)1}[n-1])=1,\\
                K(c)=0 \,\, \text{ for }\,\, c\neq t^0_{11}[0],r^0_{11}[0],s^0_{(n+1)1}[n-1].
            \end{cases}\\
            0<n-1&<k \,\,\,\,\, \begin{cases}
                K(t^0_{11}[0])-K(r^0_{11}[0])=1, \\
                K(t^0_{(n+1)1}[n-1])-K(r^0_{(n+1)1}[n-1])=1,\\
                K(c)=0 \,\, \text{ for }\,\, c\neq t^0_{(n+1)1}[n-1],r^0_{(n+1)1}[n-1],t^0_{11}[0],r^0_{11}[0].
            \end{cases}
        \end{align*}
        Observe that all of these systems have solutions. This finishes the~proof.
    \end{proof}

    \begin{proof}[Proof of Proposition~\ref{prop:Poincare computation}]
        The~result follows from Lemma~\ref{lem:applicationofclasp} and Lemma~\ref{lem:augmentations} together with a~straightforward linear algebra computation.
    \end{proof}

     \begin{proof}[Proof of Corollary~\ref{cor:Nonfillability}]
         Choose a~generic point $p$ of $\Lambda$ that is not an endpoint of a~Reeb chord. Take a~small neighbourhood around $p$ and perform a~small Reidemeister I move to introduce a~singular edge of the~local front projection. Therefore, we can perform the~cusp-connected sum of $\Lambda$ and a~small $\Lambda_k$ placed in a~Darboux ball near the~point $p$. Let us define
        $$\mathcal{T}(\Lambda,k)=\Lambda\#\Lambda_k.$$
        The~first property is obvious as we perform a~connected sum of $\Lambda$ with standard sphere $S^n$. The~results of \cite[Section 4.5.1]{Surgery} imply that the~rotation (Maslov) class does not change since we perform a~cusp connected sum with a~sphere that is simply connected. And since $n+1$ is even, we know (see \cite[Proposition 3.2]{EESnoniso}) that $$tb(\Lambda)=(-1)^{\frac{n}{2}+1}\frac{1}{2}\chi(\Lambda).$$
        
        The~results in \cite{Surgery} imply the~following. Let $\varepsilon:\mathcal{A}(\Lambda)\to \Z_2$ and $\varepsilon^\prime_k:\mathcal{A}(\Lambda_k)\to \Z_2$ be augmentations. We know that the~augmentations $\varepsilon$ and $\varepsilon^\prime_k$ induce an augmentation $$\varepsilon_k:\mathcal{A}(\mathcal{T}(\Lambda,k))\to \Z_2,$$ and so $\mathcal{T}(\Lambda,k)$ cannot be loose.
        On the~level of linearised Legendrian contact homology, surgery yields a~short exact sequence of chain complexes
        $$0\to (\Z_2\langle s \rangle,0)\to (C_\bullet(\mathcal{T}(\Lambda,k)),\partial^{\varepsilon_k})\to (C_\bullet(\Lambda),\partial^\varepsilon)\oplus (C_\bullet(\Lambda_k),\partial^{\varepsilon^\prime_k})\to 0.$$ As $|s|=n-1$, we obtain that, for $\bullet\neq n-1,n$, $$LCH_\bullet^{\varepsilon_k}(\mathcal{T}(\Lambda,k))=LCH^\varepsilon_\bullet(\Lambda)\oplus LCH_\bullet^{\varepsilon^\prime_k}(\Lambda_k).$$  Thus, $dim(LCH_{n+k-1}^{\varepsilon_k}(\mathcal{T}(\Lambda,k)))= dim(LCH_{n+k}^{\varepsilon}(\Lambda))+1$.  As all Poincaré-Chekanov polynomials of $\mathcal{T}(\Lambda,k)$ share this property for all the~augmentations of the~Chekanov-Eliashberg algebra of $\mathcal{T}(\Lambda,k)$, we obtain the~third statement.

        Note that as $\Lambda$ is horizontally displaceable, some small normal neighbourhood of $\Lambda$ is horizontally displaceable as well, and so $\mathcal{T}(\Lambda,k)$ is horizontally displaceable.
        
        Finally, for the~sake of contradiction assume that $\varepsilon_k=\varepsilon_{L}$ is a~geometric augmentation induced by an embedded exact Lagrangian filling $L$ of $\mathcal{T}_k=\mathcal{T}(\Lambda,k)$. The~Seidel-Ekholm-Dimitroglou Rizell isomorphism relates the~linearised Legendrian contact cohomology and relative homology of the~filling (see \cite{DR16Lifting}); $$LCH^\bullet_{\varepsilon_{L}}(\mathcal{T}_k;\Z_2)\cong H_{n-\bullet}(L;\Z_2),$$ here $n$ is the~dimension of the~Legendrian and note that the~right-hand side is supported in non-negative degrees $-1\leq \bullet\leq n$, which forces $LCH^\bullet_{\varepsilon_L}(\mathcal{T}_k;\Z_2)$ to vanish for $n+1<\bullet$ or $\bullet<-3$. The~duality long exact sequence of Legendrian contact homology for horizontally displaceable Legendrians was established in \cite{dualityseq} and reads;
    $$\cdots\to H_{\bullet+1}(\mathcal{T}_k;\Z_2)\to LCH^{n-1-\bullet}_{\varepsilon_L}(\mathcal{T}_k;\Z_2)\to LCH^{\varepsilon_L}_\bullet(\mathcal{T}_k;\Z_2)\to H_{\bullet}(\mathcal{T}_k;\Z_2)\to \cdots.$$ Thus, for $\bullet<-1$ or $\bullet>n$, observe that we obtain a~contradiction as $$0\cong H_{\bullet+1}(L;\Z_2)\cong LCH_{\varepsilon_L}^{n-1-\bullet}(\mathcal{T}_k;\Z_2)\cong LCH^{\varepsilon_L}_\bullet(\mathcal{T}_k;\Z_2),$$ and this cannot be satisfied for $k>2$ and $\bullet=-k,n-1+k$.
    \end{proof}

\section{Floor-crossing}\label{sec: Main Res}

The~goal of this section is to prove Theorem~\ref{thm:floorcrossing}. First, we establish the~compactness for moduli spaces of $J$-holomorphic curves with multiple positive punctures. Then, we prove that the~moduli spaces are sufficiently smooth to obtain the~result. 
\subsection{Uniform Compactness}\label{sec:COMPACTNESS}
    Let $\L$ be a~closed embedded Legendrian submanifold of $P\times\R$ so that its set of Reeb chords is finite. Let $\mu_0>0$ be a~real number.
    Let $s_{-\mu}$ be a~fixed Reeb chord of $\Lambda$ so that it is a~clasping chord for a~clasp move $(\Lambda_\mu)_{|\mu|<\mu_0}$ providing us with a~one-parameter family of Legendrians $\Lambda_\mu$.  By Lemma~\ref{lem:nonewReebchords}, there is a~natural bijection of the~sets of the~Reeb chords $\mathcal{R}(\L_\mu)$ and $\mathcal{R}(\L_{-\mu})$. We abuse the~notation and denote the~image and preimage with the~same letter, except for the~chord $s_\mu$. Write $\textbf{w}$ for a~word in Reeb chords of $\Lambda_\mu$ for $\mu$ fixed so that
    $$\textbf{w}=a\textbf{b}_1s_\mu\textbf{b}_2 \dots \textbf{b}_{p-1}s_\mu\textbf{b}_p,$$
    where
    \begin{itemize}
        \item $a$ is called the~exceptional positive letter,
        \item  $\textbf{b}_i$ are possibly empty words in Reeb chords distinct from $a$ and $s_\mu$ for all $0<i\leq p$.
    \end{itemize}
    Let $f:L\to \R$ be the~primitive on $L$. Let $c$ be a~Reeb chord with the~endpoints $c^+$ and $c^-$. We say that a~puncture at a~self-intersection $c^*$ representing $c$ is positive (negative) if
    $$f(c^+)-f(c^-)>0\,\,\,\,\,\,\,\, (f(c^+)-f(c^-)<0).$$
    
    Let $\dot{D}_m$ denote a~dics with $m$ boundary punctures, where $m=\ell(\textbf{w})$, where $\ell$ denotes the~number of letters in $\textbf{w}$. We denote by $\Mod_{\mu}(\textbf{w})$ the~space of pseudo-holomorphic curves
    $$u_\mu:(\dot{D}_m,\partial\dot{D}_m)\to (P,L_\mu)$$
    so that each $u_\mu\in \Mod_{\mu}(\textbf{w})$ has an exceptional positive puncture asymptotic to $a$ and negative punctures asymptotic to all of the~other letters of $\textbf{w}$. More precisely, these curves solve the~equation
    $$du_\mu+J_\mu\circ du_\mu\circ j=0$$ for an almost complex structure $J_\mu$ on $P$ defined in the~proof of Proposition~\ref{prop: line segment} and $j$ a~conformal structure on $\dot{D}_m$. The~punctures are ordered along the~boundary counterclockwise in the~same way as in the~word \textbf{w} with the~letter $a$ being always the~first letter. See Figure~\ref{fig:modulifloorcrossing} for illustration. Similarly, denote by $\Mod_{-\mu}(\textbf{w})$ the~moduli space of pseudo-holomorphic curves that have (possibly multiple) positive punctures; precisely one positively asymptotic to the~exceptional letter $a$ and the~rest of the~positive punctures asymptotic to the~chord $s_{-\mu}$ as many times as $s_{-\mu}$ appears in the~word $\textbf{w}$ in the~same order and position as in the~word with chords labelled with positive $\mu$. Note that there is no configuration where $s_\mu$ is both a~negative and positive puncture. In that case, our proof of compactness breaks down; see Remark~\ref{rem:stringtheory}.  

    The~expected dimension of the~moduli space follows from \cite[Section 6]{EEScontacthomology} and is:
    \begin{equation}
        \dim \Mod(a_1,\dots,a_j;b_1,\dots,b_k)=(1-j)n+\sum_{r=1}^j CZ(a_r)-\sum_{r=1}^kCZ(b_r)+\max\lbrace 0, j+k-3\rbrace,
    \end{equation} where $n$ is the~dimension of the~Legendrian, $j$ is the~number of positive punctures, $k$ is the~number of negative punctures. Now, Lemma~\ref{lem:CZ lemma} yields that $\dim \Mod_{-\mu}(\textbf{w})=\dim \Mod_\mu(\textbf{w})$ for all $\textbf{w}$.
    
    Fix a~word $\textbf{w}$ so that for each fixed $\mu>0$ the~moduli space $\Mod_\mu(\textbf{w})$ has expected dimension $0$. Such moduli spaces were proved to be compact metric space (see \cite{EEScontacthomology}) when the~dicss have precisely one positive puncture using the~standard Gromov compactness approach. We will extend this pointwise result to compactness of the~family $\Mod_\mu(\textbf{w})$. For interval $I$, let $$\Mod_{I}(\textbf{w})=\bigcup_{\mu\in I}\Mod_\mu(\textbf{w}).$$ Fix $\mu_0>0$ and define $\overline{\Mod}_{(0,\mu_0]}$ as the~sequential compactification, that is all the~possible limits of sequences of $J_\nu$-holomorphic curves $u_\nu\in \Mod_{\mu_\nu}(\textbf{w})$ where $\mu_\nu\to 0$, where we denoted $J_\nu=J_{\mu_\nu}$. In the~limit, the~action of $s_\mu$ goes to zero, thus, we might obtain broken ghost bubbles in the~limit.
    These are components of the~broken configuration of symplectic area equal to $0$.
    
    For dicss $u\in\Mod_{-\mu}$ with multiple positive punctures, one can import most of the~standard arguments proving Gromov compactness results from the~case of a~single positive puncture, however, there is a~class of so-called \textit{pinched} configurations that one has to consider as a~possible limit. For an example, see Figure~\ref{fig:Pinched}. 
    \begin{definition}
        Let $\dot{D}_m$ be a~disc with $m$ boundary punctures. Let $\gamma$ be a~curve $\gamma(s):[0,1]\to \dot{D}_m$ so that $\gamma(s)\in \partial \dot{D}_m$ for $s=0,1$ and $\gamma(s)\in \textnormal{int}\dot{D}_m$ for $s\in (0,1)$ and so there is a~positive puncture in each component of $\dot{D}_m\setminus \gamma([0,1])$. And let $\varphi_t:\dot{D}_m\to \dot{D}_m$ be a~family of continuous maps that contract the~curve $\gamma$ and so that $\varphi_t$ is a~diffeomorphism for $t\in [0,1)$. We say that a~$(J_\infty,j_\infty)$-holomorphic curve $u$ is a~pinched configuration if there is a~$(J_0,j_0)$-holomorphic curve $u^\prime$ and a~sequence $k\to \infty$ so that $(J_\nu,\varphi_{t_\nu}^*j)$-holomorphic curves  $u_\nu=u^\prime\circ \varphi_{t_\nu}$ converge to $u$ in Gromov topology as $t_\nu\to 1$.
    \end{definition}
    
    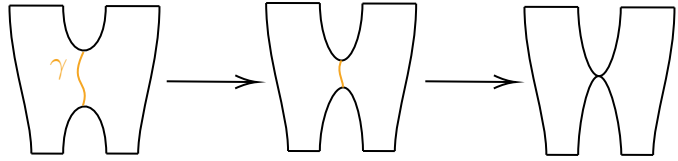
\begin{figure}[htbp]
        \centering

\tikzset{every picture/.style={line width=0.75pt}} 

\begin{tikzpicture}[x=0.75pt,y=0.75pt,yscale=-1,xscale=1]

\draw    (162.25,18.53) -- (189.2,18.58) ;
\draw    (210.59,18.79) -- (237.54,18.84) ;
\draw    (173.23,92.8) -- (189.2,92.85) ;
\draw    (210.86,92.8) -- (226.83,92.85) ;
\draw    (162.25,18.53) .. controls (162.29,47.22) and (173.4,73.48) .. (173.23,92.8) ;
\draw    (237.54,18.84) .. controls (237.58,47.54) and (227,73.53) .. (226.83,92.85) ;
\draw    (210.59,18.79) .. controls (210.63,47.49) and (189.2,49.61) .. (189.2,18.58) ;
\draw    (210.86,92.8) .. controls (210.49,61.02) and (189.2,61.55) .. (189.2,92.85) ;
\draw    (291.07,17.24) -- (318.02,17.29) ;
\draw    (339.41,17.51) -- (366.36,17.56) ;
\draw    (302.05,91.52) -- (318.02,91.57) ;
\draw    (339.68,91.52) -- (355.65,91.57) ;
\draw    (291.07,17.24) .. controls (291.11,45.94) and (302.22,72.2) .. (302.05,91.52) ;
\draw    (366.36,17.56) .. controls (366.4,46.26) and (355.82,72.25) .. (355.65,91.57) ;
\draw    (339.41,17.51) .. controls (339.43,32.2) and (334.36,46.22) .. (328.88,46.1) .. controls (323.41,45.98) and (318.02,32.44) .. (318.02,17.29) ;
\draw    (339.68,91.52) .. controls (339.49,75.26) and (335.02,59.19) .. (329.54,59.64) .. controls (324.05,60.08) and (318.02,76.27) .. (318.02,91.57) ;
\draw    (420.71,19.18) -- (447.66,19.23) ;
\draw    (469.05,19.44) -- (496,19.49) ;
\draw    (431.69,93.45) -- (447.66,93.5) ;
\draw    (469.32,93.45) -- (485.29,93.5) ;
\draw    (420.71,19.18) .. controls (420.75,47.88) and (431.86,74.14) .. (431.69,93.45) ;
\draw    (496,19.49) .. controls (496.04,48.19) and (485.46,74.19) .. (485.29,93.5) ;
\draw    (469.05,19.44) .. controls (469.07,34.53) and (463.56,54.49) .. (457.95,53.93) .. controls (452.33,53.36) and (447.66,33.94) .. (447.66,19.23) ;
\draw    (469.32,93.45) .. controls (469.14,77.69) and (463.27,53.98) .. (457.95,53.93) .. controls (452.62,53.87) and (447.66,77.72) .. (447.66,93.5) ;
\draw    (241.17,56.54) -- (284.49,56.91) ;
\draw [shift={(286.49,56.93)}, rotate = 180.49] [color={rgb, 255:red, 0; green, 0; blue, 0 }  ][line width=0.75]    (10.93,-3.29) .. controls (6.95,-1.4) and (3.31,-0.3) .. (0,0) .. controls (3.31,0.3) and (6.95,1.4) .. (10.93,3.29)   ;
\draw    (370.95,56.54) -- (414.28,56.91) ;
\draw [shift={(416.28,56.93)}, rotate = 180.49] [color={rgb, 255:red, 0; green, 0; blue, 0 }  ][line width=0.75]    (10.93,-3.29) .. controls (6.95,-1.4) and (3.31,-0.3) .. (0,0) .. controls (3.31,0.3) and (6.95,1.4) .. (10.93,3.29)   ;
\draw [color={rgb, 255:red, 245; green, 166; blue, 35 }  ,draw opacity=1 ]   (199.42,41.82) .. controls (191.27,58.64) and (203.34,56.32) .. (199.1,68.7) ;
\draw [color={rgb, 255:red, 245; green, 166; blue, 35 }  ,draw opacity=1 ]   (328.88,46.1) .. controls (325.95,51.68) and (330.18,56.71) .. (329.54,59.64) ;

\draw (181,44.03) node [anchor=north west][inner sep=0.75pt]  [color={rgb, 255:red, 245; green, 166; blue, 35 }  ,opacity=1 ]  {$\gamma $};

\end{tikzpicture}

        \caption{Example of a~pinched configuration}
        \label{fig:Pinched}
    \end{figure}
    
    As above, we define $\overline{\Mod}_{[-\mu_0,0)}$. Here, one has to prove that
    \begin{itemize}
        \item For $\mu> 0$ fixed, the~pinched configurations do not appear in the~broken limits $\overline{\Mod}_{-\mu}$.
        \item Neither pinched configurations nor ghosts appear in the~limit as $\mu_\nu\to 0$.
    \end{itemize}
    
    We can consider a~configuration of punctures on the~boundary of the~dics so that multiple consecutive punctures are asymptotic to $s^*$,the double point corresponding to the~Reeb chord $s_\mu$, then a~bigger portion of the~curve might enter the~ball $B^h_r(s^*)$ in the~clasping region, where $h$ is a~metric on $P$ (with finite geometry at infinity) compatible with $\omega$ and the~almost complex structure $J$. We can prevent this by shrinking the~ball so that it does not intersect the~complement of the~image of the~asymptotic neighbourhood of the~puncture. The~key step in the~proof of the~following lemma is to realise, that the~radius of the~ball is bounded from below. That provides us with the~information of how many times the~curve leaves the~neighbourhood. 
    \newpage\begin{lemma}\label{lem:pointsonbdryradius}
        Let $r>0$ be a~radius small enough so that $B_r(s^*)$ contains only the~local model of the~clasp move, that is, only two transversely intersecting Lagrangians $L_0$ and $L_1$. 
        For this $r$ and any word $\textbf{w}$ and every $u\in \overline{\Mod}_\mu(\textbf{w})$ it holds that $u^{-1}(\partial B_r)\cap \partial \dot{D}$ has at least $2k$-points for $k$ the~number of punctures asymptotic to $s^*$.
    \end{lemma}
    \begin{proof}
        \par \textit{Step 1: Fix $\mu\neq 0$ and let $u_\nu\to u_\infty$ be given as a~uniform limit of elements in $\Mod_\mu$. Then for all $\nu<\infty$ there is a~point $p_\nu\in [p_j,p_{j+1}]$ such that $p_\nu\not \in u_{\nu}^{-1}(B_{r})$.}
        \par For the~sake of contradiction assume that there is $k<\infty$ so that for all points $p\in [p_j,p_{j+1}]$ the~points are in the~pre-image $u_\nu^{-1}(B_r)$. Clearly, $u([p_j,p_{j+1}])=s^*$, otherwise, there is a~point $p$ distinct from $s^*$ on the~component of the~boundary joining $p_j$ and $p_{j+1}$, call image of this curve $\gamma$. This means that $\gamma$ starts at $s^*$ and follows the~Lagrangian $L_0$ (or $L_1$) in $B_{r}(s^*)$ and before it ends at $s^*$ it follows the~Lagrangian $L_1$ (or $L_0$) in $B_{r}(s^*)$ with no corner between. Since $\gamma$ is continuous, then $u(p)$ would have to lie on both $L_0$ and $L_1$, which is impossible since $(L_0\cup L_1)\setminus \lbrace s^*\rbrace$ is disconnected. And so, $\gamma$ has to leave and enter the~neighbourhood $B_r(s^*)$. But this implies $u_{\infty}([p_j,p_{j+1}])=s^*$ since $u_\infty$ is a~uniform limit. Therefore, the~component of $u_\infty$ containing $[p_j,p_{j+1}]$, denoted $u_c$, must be constant, thus $\area(u_c)=0$.

        Now, if all $s^*$-punctures are positive, then the~constant component cannot have a~negative puncture. If it had a~negative puncture asymptotic to an intersection point $b$ different from $s^*$, then $u$ would have to escape the~neighbourhood of $s^*$ and so could not be constant. If the~negative puncture were $s^*$, then $u_c$ would have to have at least two positive punctures and so $\area(u_c)\geq \mu>0$.

        If all $s^*$-punctures are negative, then the~positive puncture must be distinct from $s^*$ by the~action inequality. But then the~curve must leave a~neighbourhood of $s^*$ and so cannot be constant.

        \par \textit{Step 2: Let $\mu_\nu\to 0$ for $\mu_\nu\neq 0$, and $u_\nu\to u_\infty$ be a~uniform limit for $u_\nu\in \Mod_{\mu_\nu}$. Then for all $k\leq\infty$ there is a~point $p_k\in [p_j,p_{j+1}]$ such that $p_k\not \in u^{-1}(B_{r})$.}
        \par Let $\mu_\nu\to 0$. By uniform convergence, if $u_\infty([p_j,p_{j+1}])\in B_r(s^*)$, then for $k$ big enough $u_\nu([p_j,p_{j+1}])\in B_r(s^*)$. Consequently, we already obtain a~ghost breaking for $\mu_\nu\neq 0$, which was precluded in \textit{Step 1}.   
    \end{proof}
    
    \begin{lemma}\label{lem:monoton} There is a~$\hbar>0$ so that for every word $\textbf{w}$ all $u\in \overline{\Mod}_I(\textbf{w})$, for $I=[0,\mu_0)$ or $I=(-\mu_0,0]$, satisfy
    $$\area(u)\geq k\hbar,$$ where $k$ is the~number of letters of $\textbf{w}$ coinciding with $s^*$.   
    \end{lemma}
    \begin{proof}
        Let $r>0$ be the~radius of $B_r(s^*)$ as in Lemma~\ref{lem:pointsonbdryradius}. We can assume that there is a~constant $C>0$ that depends on the~explicit plateau function realizing the~clasp move so that $$(L_0\cup L_1)\cap B_{\frac{r}{C}}(s^*)$$ are real analytic. This means that there is a~sphere $S$ on $L_0$ so that $S=L_0\cap \partial B_\frac{r}{2C}(s^*)$ so that its neighbourhood is analytic. Now, denote by $\theta$ the~angle between Lagrangian planes $L_0$ and $L_1$, then for all $R$ satisfying $$R<\frac{r}{2C}\cos \theta$$ a~ball of radius $R$ centred at a~point $q$ of $S$ intersects only the~sheet $L_0$ and not the~other sheet $L_1$.

        Again thanks to Lemma~\ref{lem:pointsonbdryradius}, we know that we can choose asymptotic charts on the~domains of the~pseudo-holomorphic dicss, so that the~preimage $u^{-1}(B_r(s^*))$ is not going to contain a~non-compact section of these charts. Now, the~fact that the~boundary condition is real analytic, we can double our $J$-holomorphic curve. Assume that our $J$-holomorphic curve $u:\R\times [0,1]\to \mathbb{C}^n$ maps the~upper line of the~domain onto the~Lagrangian $L_0$, $$u((\R\times \lbrace 1\rbrace)\cap u^{-1}(B_R(q)))\subset L_0\cap B_R(q),$$ the~other boundary condition analogously. Because the~Lagrangian is real analytic and the~$u$ is a~holomorphic curve we use the~Schwarz reflection principle over $\R^n$ and extend the~curve to its double
        \begin{equation*}
            u^d(\tau+it)=\begin{cases}
                u(\tau+it) & \text{ for } 0\leq t\leq 1,\\
                -\Bar{u}(\tau+i(2-t)) & \text{ for } 1<t\leq 2. 
            \end{cases}
        \end{equation*}
        To see that the~double is holomorphic (thus smooth by elliptic regularity) we refer the~reader to \cite[Lemma 6.2]{EEScontacthomology}. Note that we can fix the~almost complex structure and background metric in $B_r(s^*)$ to be the~standard Hermitian pair.
        $$\area(u^d)=\int_{\R\times[0,1]}(u^d)^*\omega=\frac{1}{2}\int_{\R\times[0,2]}|\partial_\tau u^d|^2+|\partial_t u^d|^2\,\,\,d\tau\wedge d t$$
        Now, a~straightforward computation yields that $$(|\partial_\tau u^d|+|\partial_t u^d|)(\tau+i(2-t))=(|\partial_\tau u|+|\partial_t u|)(\tau+it)$$ and the~symmetry of the~domain now implies $\area(u^d)=2\area(u)$.

        Now the~monotonicity lemma used for double $u^d$ whose image contains $q$ and on the~interior of the~strip intersects the~ball $B_R(q)$ yields the~estimate $\area(u^d)\geq C^{\prime \prime}R^2.$ Finally, this yields $$\area(u)\geq \frac{C^{\prime\prime}}{2}R^2=\hbar.$$ The~proof is finished once one observes that for every puncture asymptotic to $s^*$ we have one such $q$ as above.
    \end{proof}
    \begin{lemma}
         There is $\mu_0>0$ so that for all $|\mu|<\mu_0$, only finitely many words $\textbf{w}$ can be realised in the~clasp move by a~pseudo-holomorphic curve:
        \begin{itemize}
            \item either having all the~punctures asymptotic to $s^*$ negatively,
            \item or having all the~punctures asymptotic to $s^*$ positively.
        \end{itemize}
    \end{lemma}
    \begin{proof}
        Let $M$ be the~maximal action of a~Reeb chord of $\Lambda_\mu$, then any curve with $k$ positive punctures asymptotic to $s^*$ has area 
        \begin{equation}\label{eq:ineq}
            M+k|\mu|\geq\area(u)\geq k\hbar,
        \end{equation}
        where the~second inequality follows from Lemma~\ref{lem:monoton}. Choose $|\mu|<\hbar$, then inequality \eqref{eq:ineq} is satisfied only for finitely many $0\leq k$. The~case of $k$ negative punctures is analogous and goes via the~inequality $M\geq k\hbar$. 
    \end{proof}
    \begin{corollary}
        The~pinched pseudo-holomorphic curves do not occur in the~compactification of curves $\overline{\Mod}_{-\mu}(\textbf{w})$ with multiple positive punctures for fixed $\mu>0$ small enough.
    \end{corollary}
    \begin{proof}
        Assume that we have a~limit of pseudo-holomorphic curves that converges to a~pinched configuration, then we can choose $|\mu|=\action{s_\mu}$ strictly smaller than $\hbar$ that we obtained in Lemma~\ref{lem:monoton}. Since the~pinched component has $k$ positive punctures at $s_\mu$ area is at most $k\mu$ by Stokes' theorem. Thus, we obtain a~contradiction  $$k\hbar\leq \area(u_{pinch})\leq k\action{s_\mu}=k|\mu|.$$
    \end{proof}
    \begin{corollary}
        Ghosts do not appear in the~limit as $\mu_\nu\to 0$.
    \end{corollary}
    \begin{proof}
        This follows by \textit{Step 2} of Lemma~\ref{lem:pointsonbdryradius}.
    \end{proof}
    \begin{remark}\label{rem:stringtheory}
        Note that the~assumption that the~$J$-holomorphic curve in our moduli spaces can have only one sign of the~asymptotic marker at all punctures asymptotic to the~clasping chord is crucial in the~proofs above. Otherwise, one cannot, in general, prevent the~appearance of ghosts and pinched configurations leading to beautiful connections with string topology, see \cite{cieliebak2007rolestringtopologysymplectic}, \cite{Ng_2010}, \cite{dukic2024extensionchekanoveliashbergalgebrausing}, \cite{Avdek_2025}.
    \end{remark}
    
    \subsection{Transversality}\label{sec:analsketch} At this point, we know that $\Mod_{[-\mu_0,0)}(\textbf{w})$ and $\Mod_{(0,\mu_0]}(\textbf{w})$ can be compactified. However, it is not obvious that these are one-dimensional manifolds with boundary and that the~compactifications match at $\mu=0$ and give rise to a~moduli space, which is a~one-dimensional manifold.

    \begin{proposition}\label{prop: line segment}
        For $\mu_0>0$ denote $$\Mod_{[-\mu_0,\mu_0]}(\textbf{w})=\overline{\Mod}_{[-\mu_0,0)}\cup \overline{\Mod}_{(0,\mu_0]}\,\, \textnormal{ and }\,\, \mathfrak{p}:\Mod_{[-\mu_0,\mu_0]}(\textbf{w})\to [-\mu_0,\mu_0],$$ where $\mathfrak{p}(u)=\mu$ if $u\in \overline{\Mod}_\mu(\textbf{w})$. There is a~$\mu_0>0$ so that if $\dim \Mod_\mu(\textbf{w})=0$ for all $|\mu|\leq \mu_0$, then the~projection map $\mathfrak{p}$ is a~local diffeomorphism onto $[-\mu_0,\mu_0]$.  
    \end{proposition}
    Now, it is clear that Theorem~\ref{thm:floorcrossing}  follows directly from the~Proposition~\ref{prop: line segment}.
    
    The~analytic framework for Legendrian contact homology of \textit{embedded} Legendrians was established in \cite{EEScontacthomology}. However, at $\mu=0$, our $\L_0$ is \textit{immersed}. To model the~clasp move, one has to define a~vector field $V$ (called the~clasping vector field, see Appendix~\ref{ap: Extension of V}) that points in the~\textit{normal} direction along the~boundary. And so, to set up the~analytical framework, one has to
    \begin{enumerate}
        \item extend the~set-up of \cite{EEScontacthomology} to include immersed Legendrians (see Appendix~\ref{sssec: Family of metrics on the~pull-back bundle}),
        \item thicken the~relative Banach manifolds to account for the~normal variation given by the~clasping vector field $V$ (see Appendix~\ref{ap:thickening}).
    \end{enumerate}
    The~details of the~proof strategy below fit into the~analytic set-up of \cite{EEScontacthomology}. No new analytical methods have to be developed as the~proofs of \cite{EEScontacthomology} apply verbatim to our set-up. And so, as the~methods used below are rather standard to experts in Legendrian contact homology, we defer the~details of the~following proof to Appendix~\ref{ap:Analysis}.
    \begin{proof}
        Fix $\mu=0$ and consider $u_0$ to be a~$J$-holomorphic map from a~punctured disc that is asymptotic to the~double point $s$ corresponding to the~clasping chord and with boundary on $L_0=\Pi_P(\Lambda_0)$. That is, $u$ is a~zero of a~Banach section $$\mathcal{F}_0:\mathcal{W}_0\to \mathcal{E},$$ where $\mathcal{W}_0$ is the~Banach manifold of so-called candidate maps, and $\mathcal{F}_0$ is locally modeled by the~Cauchy-Riemann operator $\delbar_{J,j}$. The~proof of the~surjectivity of $D\mathcal{F}_0$ follows verbatim from those in \cite{EEScontacthomology} as all the~strategies achieving this result concentrate close to the~exceptional positive puncture that is not asymptotic to $s$. Then we extend the~problem to $\mathcal{F}_\mu:\mathcal{W}_\mu\to \mathcal{E}$ for $\mu\in [-\mu_0,\mu_0]$, see Appendix~\ref{ap:thickening}. This produces a~family of Fredholm operators $D\mathcal{F}_\mu$ that are $C^1$ in $\mu$. The~surjectivity is now obvious. The~uniform bound on right inverses is then easily obtained from the~$\mu$-derivative of $D\mathcal{F}_\mu$ and the~Neumann series. The~standard implicit function theorem finishes the~proof.
        
        The~claim for words $\textbf{w}$ not containing the~letter $s_\mu$ follows from the~enhanced transversality (additional requirements on marked points along the~boundary, see \cite[Section 7.10]{EEScontacthomology}), whose proof is standard and follows the~proof of \cite[Proposition 4.10]{EESnoniso}. Note that this part of the~proof fails for Legendrians of dimension $1$ as \cite[Proposition 4.10]{EESnoniso} cannot be used here.
    \end{proof}
    \bibliographystyle{amsalpha}
    \bibliography{main}   
\appendix    

\newpage
\section{Flow-tree computation}\label{sec:Flowtreecomputation} Here we provide pictorial computation used in the~proof of Lemma~\ref{lem:clefstrand}. The~shaded regions correspond to the~front projection of the~$j^1$-lift of a~flow-tree. The~white gaps signify a~$3$-valent vertex.
    \begin{figure}[htbp!]
        \centering
        \tikzset{every picture/.style={line width=0.75pt}} 
        \setlength{\unitlength}{0.1\textwidth}
		\begin{picture}(10,11.5)
			\put(0,-2){\includegraphics[scale=0.9]{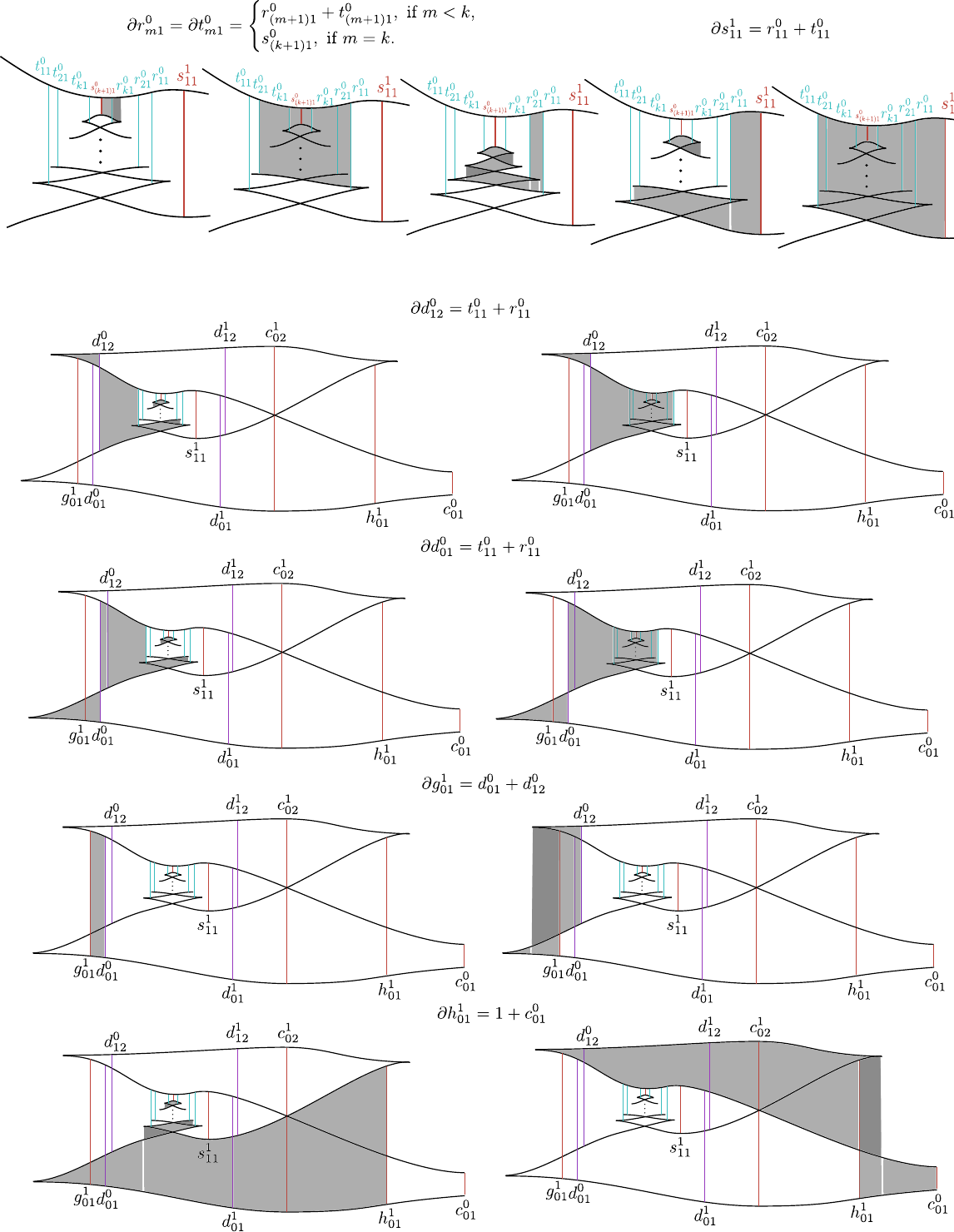}}
		\end{picture}
        \label{fig:Computation 1}
    \end{figure}
    \newpage
    \begin{figure}[htbp!]
        \centering
        \tikzset{every picture/.style={line width=0.75pt}} 
        \setlength{\unitlength}{0.1\textwidth}
		\begin{picture}(10,12)
			\put(0,-2){\includegraphics[scale=0.9]{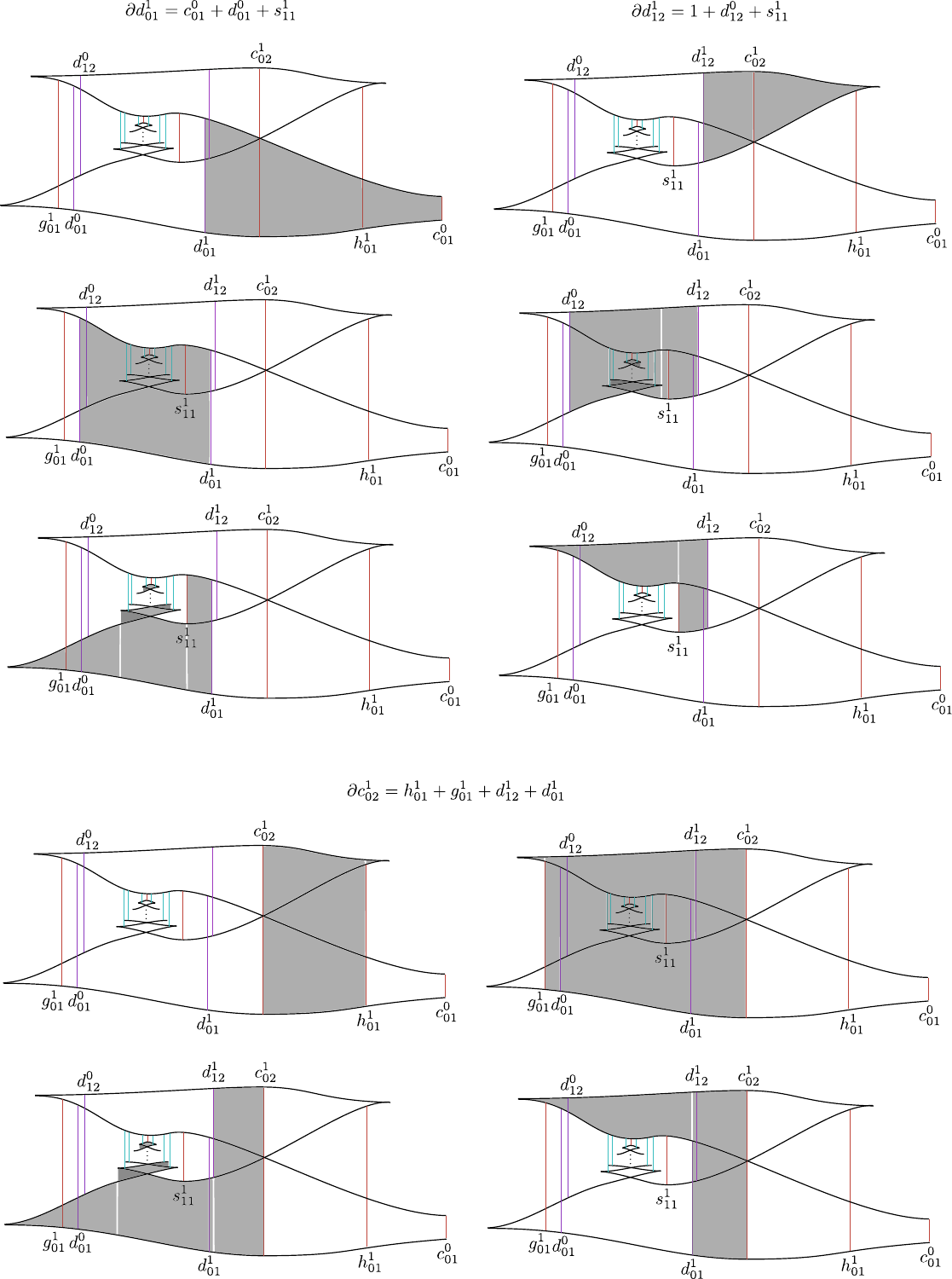}}
		\end{picture}
        \label{fig:Computation 2}
    \end{figure}

\newpage

\section{Analytical set-up}\label{ap:Analysis}
\subsection{Family of metrics on the~pull-back bundle}\label{sssec: Family of metrics on the~pull-back bundle}
         Here we construct a~domain-dependent family of metrics on the~pull-back bundle $u^*(TP)$ for suitable $u:\Dot{D}_m\to P$ so that the~family of metrics is constant near the~boundary of $\Dot{D}_m$ and the~partial derivatives of the~family are uniformly bounded. In particular, these metrics are not going to depend on the~action of Reeb chords, unlike the~ones originally used in \cite[Section 5.3]{EEScontacthomology}, where these choices make the~analytic set-up (starting with Banach manifold charts) defined only for Legendrians with Reeb chords of positive action. 
         
        We diverge from the~set-up of \cite[Section 5.3]{EEScontacthomology}. In particular, we are not constructing a~family of metrics in $P$ (that uses the~embedding property of $\Lambda$, the~Legendrian lift to the~contactisation $P\times \R$). Instead of the~geometric approach, we formulate the~problem more virtually and construct a~domain-dependent family of metrics on the~pull-back bundle $\Tilde{u}^*(T\Dot{D}_m\times TP)$. The~sought family of metrics will be the~product metric of the~family of metrics on the~domain from \cite{EEScontacthomology} and the~metrics constructed below.
        
        Recall that we denote by $h$ a~background metric on $P$ as in Section~\ref{sec:COMPACTNESS}.
        Let $\iota:T^*\Lambda\to P\times \R$ be the~map identifying the~horizontal part of the~Weinstein neighbourhood of $\Lambda$ with $T^*\Lambda$. Denote $\Pi:P\times\R\to P$ the~Lagrangian projection and denote by $N_r$ the~image of $D^*_r\Lambda$ in $P$ under the~composition $\Pi\circ\iota$. We choose $r$ smaller than the~injectivity radius of the~metric $\hat{g}$ on $T^*\Lambda$, constructed in \cite[Section 5.2]{EEScontacthomology}.
        
        Now, we are going to construct an interpolating family of domain-dependent metrics so that we keep track of the~Lagrangian boundary condition and special form in the~asymptotic neighbourhoods of the~punctures as in \cite[Section 5.3]{EEScontacthomology} while keeping the~derivatives of the~family uniformly bounded. Let $N_r(u)=u^{-1}(N_r)$, this forms a~neighbourhood of the~boundary $\partial \Dot{D}_m$.
        The~boundary $\partial \Dot{D}_m$ decomposes into components $\partial_j\Dot{D}_m$ running from the~puncture $p_j$ to the~puncture $p_{j+1}$. For each $\partial_j\Dot{D}_m$ consider $Q_j=(\Pi_P\circ \iota)^{-1}(u(\partial_j\Dot{D}_m))$ a~subset of the~zero section of $D^{*}_r\Lambda$, denote by $g^j$ the~push-forward of the~metric $\hat{g}$ under the~map $\Pi_P\circ \iota$ restricted to $D^*_rQ_j$ and write $N_r^j(u)$ for the~components of the~set $u^{-1}(\Pi_P\circ \iota(D^*_rQ_j))$ that have non-empty intersection with $\partial_j\Dot{D}_m$. If the~sets $N_r^j(u)$ are not pairwise disjoint outside of the~asymptotic neighbourhoods $\bigcup_{j=0}^m E_j[1]$, then choose $r$ smaller to contain all the~intersections in the~respective asymptotic neighbourhoods. Let $\mathcal{G}$ be the~space of Riemannian metrics on $P$. There is a~smooth function $g^\sigma:\Dot{D}_m\to \mathcal{G}$ so that
        \begin{equation*}
            g^\sigma(\zeta)=\begin{cases}
                h & \text{ if } \zeta\not\in \bigcup_{j=1}^m N_r^j(u),\\
                g^j & \text{ if } \zeta \in \bigcup_{j=1}^m N_{\frac{r}{2}}^j(u)\setminus \bigcup_{j=1}^m E_j[1],\\
                g^j_{\tau+it} & \text{ if } \zeta=\tau+it \in\bigcup_{j=1}^m E_j[1],
            \end{cases}
        \end{equation*}
        where $g^j_{\tau+it}$ is the~asymptotic family of metrics constructed as follows; assume that the~lines $\R$ and $i\R$ map to the~components $Q_{j_0}$ and $Q_{j_1}$
        \begin{itemize}
            \item $g^j_{\tau+it}=g^{j_0}$ for $t$ in a~neighbourhood of $0$,
            \item $g^j_{\tau+it}=g^{j_1}$ for $t$ in a~neighbourhood of $1$,
            \item the~family of metrics interpolates between $g^j_{\tau+it}=g^{j_0}$ and   $g^j_{\tau+it}=g^{j_1}$ in $t$ so that the~derivatives with respect to the~coordinates $\tau,t$ are uniformly bounded,
            \item and, as we approach the~boundary of the~asymptotic neighbourhood with the~compact part of the~domain the~neighbourhoods $N_{\frac{r}{2}}^{j_0}(u)$ and $N_{\frac{r}{2}}^{j_1}(u)$ become disjoint and we interpolate between them to make $g^\sigma$ smooth.
        \end{itemize}
        The~construction above was done for a~fixed Lagrangian condition $L$. For a~smooth family of Lagrangians parametrised by $\mu$, one needs to repeat the~construction above for every $\mu$ separately to obtain the~family of metrics $g^\sigma_\mu(\zeta)$.
        
    \subsection{Modification of the~standard set-up}
        First let us recall some details from \cite{EESPR}. We wish to construct a~structure of Banach manifold of the~space $\mathcal{W}_{2,\epsilon}(\textbf{b};\kappa)$ consisting of pairs of maps $$(u,f):(\Dot{D}_m,\partial \Dot{D}_m)\to (P,\R)$$ satisfying the~following two conditions;
        \begin{align}
            (u,f)(\zeta)\in \Lambda\subset P\times \R, & \text{ for all }\zeta\in \partial \Dot{D}_m;\\
            \int_{\partial \Dot{D}_m} \langle \delbar_J(u),v \rangle ds=0, & \text{ for all } v\in C_0^0(\partial \Dot{D}_m, T^{*0,1}\Dot{D}_m\otimes u^*(TP)).
        \end{align}
        We denote by $F$ the~extension of $f:\partial \Dot{D}_m\to \R$ to $F:\Dot{D}_m\to \R$. The~construction of the~$C^1$-chart around some $(u,f)$ goes via exponentiation of the~elements of the~tangent space at $(u,f)$,
        \begin{align*}
            \Psi: T_{(u,f)}\mathcal{W}_{2,\epsilon}(\textbf{b};\kappa)\to \mathcal{W}_{2,\epsilon}(\textbf{b};\kappa),\\
            \Psi(v)(\zeta)=\exp_{u(\zeta)}^{\rho(F(\zeta))}(v(\zeta)),
        \end{align*}
        where $T_{(u,f)}\mathcal{W}_{2,\epsilon}(\textbf{b};\kappa)$ consists of all $v\in \mathcal{H}_{2,\epsilon}(D,u^*(TP))$ so that
        \begin{align}
            v(\zeta)\in T_{u(\zeta)}L \subset T_{u(\zeta)}P & \text{ for all } \zeta\in \partial D,\\
            \int_{\partial \Dot{D}_m} \langle \Bar{\nabla}_{J,j} v,w \rangle ds=0 & \text{ for all } w\in C_0^0(\partial \Dot{D}_m, u^*(TP)),
        \end{align}
        denoting
        \begin{equation}\label{eq:nablabar}
            \Bar{\nabla}_{J,j} v= \nabla v+J\circ \nabla v \circ j,
        \end{equation}
         where $\nabla$ is the~connection on $u^*(TP)$ given by the~point-wise dependent metric $g^{\rho(F(\zeta))}$. This is well-defined only in some $\eta$-ball around $(u,f)$ where $\eta$ depends continuously on the~minimal injectivity radius of the~metrics $g^\sigma$ for $0\leq \sigma \leq 1$. Note that this construction depends on the~extension $F$.
        \begin{remark}
            The~choice of a~family of metrics in Section~\ref{sssec: Family of metrics on the~pull-back bundle} induces an analytical set-up which is almost identical to the~one of \cite{EESPR} and \cite{EEScontacthomology}, so we use notation established there with one important difference. Instead of choosing a~smooth function $\rho:P\times \R\to [0,1]$ which yields a~$(P\times \R)$-dependent family of metrics $g^\rho$ on $P$, we defined a~family of domain-dependent metrics $g^\sigma$ on the~bundle $u^*(TP)$. We eliminated this dependence carefully so that the~analytic package of \cite{EESPR} can still be used.
            All the~proofs of \cite{EESPR} translate \textit{verbatim} into this slightly more general setup. This is possible since, in the~crucial steps, the~dependence of the~family of metrics on the~lift in $P\times \R$ is purely of convenience. Nevertheless, for the~sake of completeness, we cover some of the~crucial steps below.
        
            \textit{Gluing:} In \cite{EESPR} asymptotic neighbourhood of the~punctures on the~disc has canonical neighbourhoods $E_j[0]$ conformally equivalent to $[0,+\infty)_\tau\times [0,1]_t$ and a~complex structure so that $\zeta=\tau +it$ for $\zeta\in E_j[0]$. To be able to glue $J$-holomorphic curves together, one changes the~asymptotic coordinates on $E_j[0]$ near negative asymptotics to $(-\infty,0]\times [0,1]$ using the~conformal equivalence $$\tau+it\mapsto -\tau+i(1-t).$$ Then to construct the~approximate solution one uses the~domain-dependent family of metrics $g^{\rho(F(\sigma))}$ to define the~approximate solution on the~interpolating region. Compare with \cite[Section 4.3]{EESPR}. Note that here, the~dependence on the~lift $F$ gives us a~globally defined prescription of how the~metrics change under the~conformal equivalence of the~asymptotic neighbourhoods $E_j[0]$. Therefore, in our more general set-up, we need to keep track of which boundary component gets mapped to the~lower or upper sheet, which is the~reason for the~introduction of the~sets $Q_j$ above.
        
            \textit{Fredholm problem:} In Section 4.1. of \cite{EESPR}, one writes the~linearization of the~$\delbar_{J,j}$ at $(u,f)$ in the~form
            $$D\delbar_{J,j}[v]=\Bar{\nabla}_{J,j}v-\frac{1}{2}J\circ(\nabla_{v}J)\circ \partial_{J,j}u=\Bar{\nabla}_{J,j}v+Kv,$$ where $v$ is the~tangent vector of the~Banach manifold of candidate maps $\mathcal{W}_{\epsilon,\textbf{b}}$. Then one realises that for $D\delbar_{J,j}$ to be Fredholm, it is enough to prove that $K$ is a~compact operator. That is true, provided $h$ is a~smooth matrix-valued function with all derivatives bounded satisfying $Kv=h(v)\cdot du\cdot v$. The~derivatives of $h$ depend on the~metric $g^{\sigma}$, which we constructed carefully to retain this control.
        
            \textit{Transversality:} In Section 4.2. of \cite{EESPR}, one can achieve transversality by perturbing $J$ near special regular oriented curves close to one distinguished positive puncture. Here, the~linearization reads
            $$D\delbar_{J,j}[v,\lambda]= \delbar_{J,j}v+Kv+K_S,$$
            where the~term $K_S$ describing the~variation of $J$ does not depend on $g^\sigma$. 
        \end{remark}    
    \subsection{Extension of the~clasping vector field}\label{ap: Extension of V}
        We can model the~clasp move $(\Lambda_\mu)_{\mu\in \R}$ in the~Lagrangian projection as $$L_\mu=\hat{\iota}(\exp_\Lambda^{\hat{g}}(\mu J\nabla\rho)),$$ where $\hat{g}$ is the~Euclidean metric in the~neighbourhood of the~double point so that in this neighbourhood $J$ is integrable, $\hat{\iota}=\Pi\circ \iota$ is a~biholomorphic identification of this neighbourhood in $P$ with $T^*\Lambda$, and $\rho$ is the~plateau function defining the~clasp move from Section~\ref{sec:Clasp move}. Note that the~vector field $V=J\nabla\rho$ that models the~clasp-move is defined only along the~zero section. We call $V$ the~clasping vector field.
        
        Let $g_\mu$ be the~metric on $\Lambda$ given as the~restriction of the~metric $\hat{g}$ onto $\Lambda_\mu$. Construct a~family of almost complex structures $J_\mu$ on $P$ (as in \cite[Lemma 5.6]{EEScontacthomology}) so that the~induced metric $\hat{g}_\mu$ on $T^*\Lambda_\mu$ satisfies that for a~vector field $T_\mu$ tangent along $L_\mu$ and any vector field $X_\mu$ we have the~identity
            \begin{equation}\label{eq: boundary kahler}
                \nabla_{T_\mu}^\mu J_\mu X_\mu=J_\mu\nabla^\mu_{T_\mu}X_\mu.
            \end{equation} 

        We have to make the~extension of $V$ depend on the~domain to be able to define it from the~boundary into the~interior of $\Dot{D}_m$. This is impossible in a~non-graphical context, as the~image of $u$ might not be embedded, see below.
         
         Let $u_\mu(\zeta)=\exp^{\hat{g}}_{u_0(\zeta)}(\mu V(u_0(\zeta)))$ for $\zeta\in \partial \Dot{D}_m$. There exists a~$\mu$-family of cut-off functions $\alpha_\mu$ supported in the~neighbourhood of the~boundary component where $V(u_\mu(\zeta))$ is supported, and a~$\mu$-family of functions $f_\mu$ so that $f_\mu(\zeta)=0$ for $\zeta\in \partial \Dot{D}_m$, $\partial_\tau f_\mu=0$, and $\partial_tf_\mu=1$.
            Moreover, the~vector field defined as$$V_\mu(\zeta,u_\mu(\zeta))=\Big(V(u_\mu(\zeta))+f_\mu(\zeta)J_\mu\nabla^\mu_{\hat{\partial}_\tau}V(u_\mu(\zeta))\Big)\,\alpha_\mu(\zeta)$$
            satisfies the~following two conditions for all $\zeta\in \partial \Dot{D}_m$ and $q\in P$ close to the~image of the~original map $u$;
            \begin{align}
                &V_\mu(\zeta,u_\mu(\zeta))=V(u_\mu(\zeta)),\\
                &\nabla^\mu V_\mu(\zeta,q)+J_\mu\circ \nabla^\mu V_\mu(\zeta,q)\circ j=0. 
            \end{align}
    \subsection{Regularity of family of Fredholm operators}\label{ap:thickening}
        On the~Banach manifold side, the~existence of the~extension of $V$ above allows us to ''thicken'' our source Banach manifold $\mathcal{W}$ to $$\mathcal{W}_\mu\to\mathcal{N}\xrightarrow{\mathfrak{p}} \R,$$ an affine Banach fibre bundle of candidate maps for the~$J_\mu$-holomorphic curves, where $\mathcal{W}_\mu=\mathfrak{p}^{-1}(\mu)$ describes the~set of maps with Dirichlet and Neumann boundary conditions:
    \begin{align} 
        \label{eq:boundary condition} v_\mu(\zeta)\in T_{u_\mu(\zeta)}L_\mu,& \text{ for all } \zeta\in \partial \Dot{D}_m,\\
        \label{eq:del boundary conditions}\int_{\partial \Dot{D}_m}\langle \overline{\nabla}_{J_\mu,j_\kappa} v_\mu, w_\mu \rangle&=0, \text{ for all } w_\mu \in C_0^0(\partial \Dot{D}_m,T^{*0,1}\Dot{D}_m\otimes u_\mu^*(T\Dot{D}_m\times TP)),
    \end{align}
    where $\overline{\nabla}_{J_\mu,j_\kappa} v_\mu$ is defined analogously to the~operator in \eqref{eq:nablabar}, and  $v_\mu=\Pi_{\mu V} \, (v)$ is the~parallel transport along $V$. Now, define Banach manifold chart $$\Psi_\mu[v](\zeta)=\exp^{g_{pr,\mu}}_{\exp_{u(\zeta)(\mu V(\zeta))}}(\zeta)(v_\mu).$$ One can observe that $\Psi$ is a~$C^1$-map of Banach manifolds, which is $C^2$-with respect to the~$\mu$-coordinate. This can be shown as in \cite[Lemma 5.12]{EEScontacthomology}.
        Let $\mathcal{F}:\mathcal{N}\to \mathcal{E}$ be the~section of affine Banach bundles so that $$\mathcal{F}_\mu:\mathcal{W}_\mu\to \mathcal{E}$$ is the~section corresponding to the~operator $\delbar_{J_\mu,j}(\Psi_\mu(\cdot))$. We want to prove that $\mathcal{F}_\mu$ is a~family of sections so that $D\mathcal{F}_\mu$ is a~family of Fredholm operators that is $\mathcal{C}^1$ in the~$\mu$ variable, that is
        \begin{equation}\label{eq:family of sections}
            |\Psi_{-\mu}\circ \mathcal{F}_\mu\circ \Phi_\mu-\mathcal{F}_0|\leq C_0(\mu) \text{ and } |D_{u_\mu}\mathcal{F}_\mu-D_{u_0}\mathcal{F}_0|\leq C_1(\mu),
        \end{equation}
        where $C_i$ are continuous functions that vanish at $\mu=0$. The~proof of the~inequalities~\ref{eq:family of sections} follows by expressing the~first-order variations of $\Psi_{-\mu}\circ \mathcal{F}_\mu\circ \Phi_\mu$ up to order $1$ in the~coordinate $\mu$. By lengthy and straightforward computation, one proves that the~operator $\mathcal{F}_\mu=\delbar_{J_\mu,j}(u_\mu)$ has the~Taylor expansion at $\mu=0$ equal to
            \begin{equation}
                \delbar_{J_\mu,j}(u_\mu)=\delbar_{J,j}(u)+\mu (\overline{\nabla}_{J,j}V+\nabla_V J\cdot du\cdot j-\frac{1}{2}J\mathcal{L}_V J\cdot \partial_{J,j}(u))+\mathcal{O}(\mu^2).
            \end{equation} Then an easy linearisation shows that the~$D\mathcal{F}_\mu$ are a~regular enough path in Fredholm operators. Therefore, we obtain that $D_{u_\mu}\mathcal{F}_\mu$ must be surjective for $\mu$ small enough (from surjectivity of $D\mathcal{F}_0$ established in \cite[Lemma 4.5]{EESPR}), as surjectivity is an open condition. The~standard application of the~infinite-dimensional implicit function theorem now finishes the~proof of Proposition~\ref{prop: line segment} for words $\textbf{w}$ containing letters $s_\mu$. The~higher regularity in $\mu$ was needed for the~uniform estimate for right inverses.
    
    
\end{document}